\documentclass[11pt]{article}

\usepackage{geometry}
\usepackage{latexsym}
\usepackage{mathrsfs}
\usepackage{amsfonts}
\usepackage{amssymb}
\usepackage{subfigure}    %for parallel figures
\usepackage{graphicx}
\usepackage{epstopdf}
\usepackage{float}
\usepackage{multirow}
\usepackage{slashbox}
\usepackage{bm}
\usepackage{amsmath}
\usepackage{amsthm}
\usepackage{color}
\usepackage[subnum]{cases}
\usepackage{rotating}
\usepackage{enumitem}
\usepackage{accents}
\usepackage{algorithm}
\usepackage{algorithmic}
\usepackage[title]{appendix}
\usepackage{mathtools}
\usepackage{caption}

\usepackage{url}
\usepackage[pdftex,
plainpages=false,
bookmarks,
bookmarksnumbered,
colorlinks=true,
linkcolor=green,
citecolor=green,
urlcolor=green,
filecolor=black
]
{hyperref}

\makeatletter
\@addtoreset{equation}{section}
\makeatother

\makeatletter
\def \wideubar{\underaccent{{\cc@style\underline{\mskip15mu}}}}
\def \widebar{\accentset{{\cc@style\underline{\mskip10mu}}}}
\makeatother

\newcommand{\email}[1]{\protect\href{mailto:#1}{#1}}

\definecolor{blue}{rgb}{0,0,0}
\definecolor{red}{rgb}{0.9,0,0}
\definecolor{green}{rgb}{0.1,0.5,0.1}
\newcommand{\blue}[1]{\begin{color}{blue}#1\end{color}}

\graphicspath{{images/}{images/ent/}{images/qua/}{images/new/}}

\begin{document}

\newtheorem{property}{Property}[section]
\newtheorem{proposition}{Proposition}[section]
\newtheorem{append}{Appendix}[section]
\newtheorem{definition}{Definition}[section]
\newtheorem{lemma}{Lemma}[section]
\newtheorem{corollary}{Corollary}[section]
\newtheorem{theorem}{Theorem}[section]
\newtheorem{remark}{Remark}[section]
\newtheorem{problem}{Problem}[section]
\newtheorem{example}{Example}[section]
\newtheorem{assumption}{Assumption}
\renewcommand*{\theassumption}{\Alph{assumption}}

\title{Bregman Proximal Point Algorithm Revisited: A New Inexact Version and its Inertial Variant}

\author{Lei Yang\footnotemark[1] \and Kim-Chuan Toh\footnotemark[2]}

\renewcommand{\thefootnote}{\fnsymbol{footnote}}
\footnotetext[1]{Department of Applied Mathematics, The Hong Kong Polytechnic University, Hung Hom, Kowloon, Hong Kong, China (\email{yanglei.math@gmail.com}). }
\footnotetext[2]{Department of Mathematics, and Institute of Operations Research and Analytics, National University of Singapore (\email{mattohkc@nus.edu.sg}). This research is supported by the Ministry of Education, Singapore, under its
Academic Research Fund Tier 1 Grant: R-146-000-336-114.
}
\renewcommand{\thefootnote}{\arabic{footnote}}

\maketitle

\begin{abstract}
We study a general convex optimization problem, which covers various classic problems in different areas and particularly includes many optimal transport related problems arising in recent years. To solve this problem, we revisit the classic Bregman proximal point algorithm (BPPA) and introduce a new inexact stopping condition for solving the subproblems, which can circumvent the underlying feasibility difficulty often appearing in existing inexact conditions when the problem has a complex feasible set. Our inexact condition also covers several existing inexact conditions as special cases and hence makes our inexact BPPA (iBPPA) more flexible to fit different scenarios in practice. As an application to the standard optimal transport (OT) problem, our iBPPA with the entropic proximal term can bypass some numerical instability issues that usually plague the popular Sinkhorn's algorithm in the OT community, since our iBPPA does not require the proximal parameter to be very small for obtaining an accurate approximate solution. The iteration complexity of $O(1/k)$ and the convergence of the sequence are also established for our iBPPA under some mild conditions. Moreover, inspired by Nesterov's acceleration technique, we develop an inertial variant of our iBPPA, denoted by V-iBPPA, and establish the iteration complexity of $O(1/k^{\lambda})$, where $\lambda\geq1$ is a quadrangle scaling exponent of the kernel function. In particular, when the proximal parameter is a constant and the kernel function is strongly convex with Lipschitz continuous gradient (hence $\lambda=2$), our V-iBPPA achieves a faster rate of $O(1/k^2)$ just as existing accelerated inexact proximal point algorithms. Some preliminary numerical experiments for solving the standard OT problem are conducted to show the convergence behaviors of our iBPPA and V-iBPPA under different inexactness settings. The experiments also empirically verify the potential of our V-iBPPA on improving the convergence speed.

\vspace{5mm}
\noindent {\bf Keywords:} Proximal point algorithm; Bregman distance; inexact condition; Nesterov's acceleration; optimal transport.
\end{abstract}

%%%%%%%%%%%%%%%%%%%%%%%%%%%%%%%%%%%%%%%
\section{Introduction}

We consider the following convex optimization problem
\begin{equation}\label{orgpro}
\min_{\bm{x}}~f(\bm{x}) \quad \mathrm{s.t.} \quad \bm{x} \in \overline{\mathcal{C}},
\end{equation}
where $f: \mathbb{E} \rightarrow (-\infty,\infty]$ is a proper closed convex function, $\mathcal{C}\subseteq\mathbb{E}$ is a nonempty convex \textit{open} set, $\overline{\mathcal{C}}$ is the closure of $\mathcal{C}$ and $\mathbb{E}$ is a real finite dimensional Euclidean space equipped with an inner product $\langle\cdot,\cdot\rangle$ and its induced norm $\|\cdot\|$. Problem \eqref{orgpro} can cover a large class of convex optimization problems in various areas. We are particularly interested in optimization problems over the nonnegative octant arising in the area of optimal transport (OT); for example, the standard OT problem (see problem \eqref{otproblem}), the capacity constrained OT problem and the problem of computing Wasserstein barycenters, just to name a few. All these problems have found many applications and attracted considerable attention in recent years. We refer interested readers to a recent comprehensive survey \cite{pc2019computational} (mainly from the computational point of view) and references therein for more details on OT and its related problems.

Among different solution methods designed for solving problem \eqref{orgpro}, the proximal point algorithm (PPA) is arguably the most fundamental one that basically generates a sequence $\{\bm{x}^k\}$ via the following scheme
\begin{equation}\label{genericiter}
\bm{x}^{k+1} \approx \arg\min\limits_{\bm{x}}\left\{f(\bm{x}) + \gamma_k\,d(\bm{x}, \,\bm{x}^k) : \bm{x} \in \overline{\mathcal{C}}\right\},
\end{equation}
where $d(\cdot,\cdot)$ is a certain proximity measure, $\gamma_k>0$ is a given proximal parameter and ``$\approx$" means that $\bm{x}^{k+1}$ is only required to be an approximate solution of \eqref{genericiter} in some sense (to be specified later) since computing an exact solution of \eqref{genericiter} is in general too expensive. With the choice of $d(\bm{x},\bm{y})=\frac{1}{2}\|\bm{x}-\bm{y}\|^2$, the above iterative scheme exactly reduces to the classical (inexact) PPA which enjoys comprehensive convergence results; see, for example, \cite{g1991on,g1992new,m1970regularisation,m1965proximite,r1976augmented,r1976monotone}. Besides the half squared Euclidean distance, various researchers have also considered using some other non-Euclidean proximity measures in \eqref{genericiter}; see, for example, \cite{a1995interior,cz1992proximal,cz1997parallel,ct1993convergence,e1993nonlinear,e1998approximate,e1990multiplicative,ist1994entropy,i1995convergence,k1997proximal,t1992entropic,t1997convergence}. Such an idea stems not only from natural algorithmic generalizations, but also from practical considerations on some specific applications. In particular, we find that, for OT and its related problems, a proper choice of $d$ (specifically, the entropic proximal term) may capture the geometry/structure of the problem, which allows one to eliminate the constraint $\bm{x}\in\overline{\mathcal{C}}$ during the PPA iterations and leads to a simpler subproblem \eqref{genericiter}. To date, there exist a variety of general proximity measures such as the Bregman distance \cite{b1967relaxation} and the $\psi$-divergence \cite{c1967information}; see \cite{at2006interior} for a comprehensive study on various proximity measures. In this paper, we will focus on the scheme \eqref{genericiter} based on the Bregman distance, which has a long history of being incorporated in proximal-type methods and is still very popular nowadays (see, for example, \cite{bbt2017descent,bstv2018first,lfn2018relatively,t2018simplified}), but some results developed later can also be extended to other proximity measures.

We now consider the scheme \eqref{genericiter} with the choice of $d(\bm{x},\,\bm{y})=\mathcal{D}_{\phi}(\bm{x},\,\bm{y})$, namely,
\begin{equation}\label{Bregiter}
\bm{x}^{k+1} \approx \arg\min\limits_{\bm{x}}\left\{f(\bm{x}) + \gamma_k\,\mathcal{D}_{\phi}(\bm{x},\,\bm{x}^k)\right\},
\end{equation}
where $\mathcal{D}_{\phi}(\cdot,\cdot)$ is the Bregman distance associated with the kernel function $\phi$ (see next section for definition). This scheme is a generic template for an inexact Bregman proximal point algorithm (iBPPA); see, for example, \cite{at2006interior,cz1992proximal,cz1997parallel,ct1993convergence,e1993nonlinear,e1998approximate,k1997proximal,ss2000inexact}. In particular, we are interested in a class of kernel functions $\phi$ satisfying certain desirable properties including $\mathrm{dom}\,\phi = \overline{\mathcal{C}}$ (see Assumption \ref{assumps}(iii)) so that the sequence $\{\bm{x}^k\}$ is forced to stay in $\mathcal{C}$ and thus the constraint $\bm{x} \in \overline{\mathcal{C}}$ is automatically eliminated in \eqref{Bregiter}. But, even without such a constraint, the subproblem \eqref{Bregiter} is still generally nontrivial to solve. Therefore, for the algorithm to be implementable and practical, it must allow one to solve the subproblem approximately and the corresponding stopping condition must be practically verifiable.
This consideration then gives rise to different variants of the iBPPA. In the literature, a commonly used inexact framework is based on the $\nu$-subdifferential of $f$
\cite{bis1997enlargement,k1997proximal,t1997convergence}. Basically, the iterate $\bm{x}^{k+1}$ needs to satisfy
\begin{equation}\label{BPPAcond1-Te}
0 \in \partial_{\nu_k} f(\bm{x}^{k+1}) + \gamma_k\big(\nabla\phi(\bm{x}^{k+1}) - \nabla\phi(\bm{x}^{k})\big),
\end{equation}
which allows the approximate computation of the subdifferential of $f$ at $\bm{x}^{k+1}$. Another widely used inexact framework was first considered by Rockafellar \cite{r1976monotone} for the classic PPA and later extended by Eckstein \cite{e1998approximate} for the BPPA. Specifically, it requires $\bm{x}^{k+1}$ to satisfy
\begin{equation}\label{BPPAcond1-Ec}
\Delta^{k} \in \partial f(\bm{x}^{k+1}) + \gamma_k\big(\nabla\phi(\bm{x}^{k+1}) - \nabla\phi(\bm{x}^{k})\big)
~~\mathrm{with}~~\|\Delta^{k}\| \leq \eta_k,
\end{equation}
which is typically easier to check than the $\nu$-subdifferential-based condition \eqref{BPPAcond1-Te}. However, we should be mindful that both conditions \eqref{BPPAcond1-Te} and \eqref{BPPAcond1-Ec} implicitly require that the approximate solution $\bm{x}^{k+1}$ must satisfy $\bm{x}^{k+1}\in\mathrm{dom}\,f$ (for the nonemptyness of $\partial_{\nu_k} f(\bm{x}^{k+1})$ or $\partial f(\bm{x}^{k+1})$) and $\bm{x}^{k+1}\in\mathrm{dom}\,\nabla\phi$ (for the well-definedness of $\nabla\phi(\bm{x}^{k+1})$) at the same time. But in practice, such a requirement may be nontrivial to satisfy, especially when $\mathrm{dom}\,f$ is not a simple convex set. Thus checking whether condition \eqref{BPPAcond1-Te} or \eqref{BPPAcond1-Ec} holds could be very expensive, if not impossible.
In \cite{ss2000inexact}, Solodov and Svaiter proposed another inexact framework for the BPPA for which the stopping condition is more practical and constructive when $\nabla \phi$ is explicitly invertible. Specifically, this approach requires a triple $(\bm{x}^{k+1},\bm{y}^k,\bm{v}^k)$ to satisfy
\begin{equation}\label{BPPAcond1-SS}
\begin{aligned}
&\bm{x}^{k+1} = (\nabla \phi)^{-1} (\nabla \phi(\bm{x}^k)-\gamma_k^{-1}\bm{v}^k), \;\; \bm{v}^k \in \partial f(\bm{y}^{k}), \\
&\mathcal{D}_{\phi}(\bm{y}^k, \,\bm{x}^{k+1}) \leq \sigma^2\mathcal{D}_{\phi}(\bm{y}^k, \,\bm{x}^{k}),
\end{aligned}
\end{equation}
where $\bm{y}^k$ is an intermediary point and $\sigma\in[0,1)$ is a constant. Note that it needs the exact computation of an element $\bm{v}^k$ in $\partial f(\bm{y}^k)$, which sometimes could be difficult to satisfy when $f$ is not a simple function. We should point out that when $\phi$ is chosen as the classical half squared Euclidean norm, this exactness requirement has been relaxed by allowing an element in $\partial_{\nu} f$ for some $\nu\geq0$ (see, e.g., \cite{ms2010complexity,ss1999hybridap}), but it is not clear whether such an exactness requirement can be relaxed for a general kernel function.

The aforementioned feasibility difficulty of requiring $\bm{x}^{k+1}\in\mathrm{dom}\,f\cap\mathrm{dom}\,\nabla\phi$ in \eqref{BPPAcond1-Te} or \eqref{BPPAcond1-Ec} and the difficulty of computing an element of $\partial f (\bm{y}^k)$ in \eqref{BPPAcond1-SS} thus motivate us to propose a new inexact framework (see \eqref{BPPAcond-gen}), which relaxes the previous stringent requirements by allowing $\partial_{\nu_k} f$ and $\nabla \phi$ to be computed at two slightly different points, respectively. Though the idea is simple, it is surprising that it has not been explored before.
Later in Section \ref{sec-ot}, we show by a concrete application to the standard OT problem that the verification of our inexact condition \eqref{BPPAcond-gen} is implementable and more practical. Moreover, our iBPPA with the entropic proximal term can bypass some numerical instability issues that often plague the popular Sinkhorn's algorithm used in the OT community. This is because in contrast to Sinkhorn's algorithm, our iBPPA does not require the proximal parameter to be very small in order to obtain an accurate approximate solution, as evident from our numerical results in Section \ref{secnum}.

Over the last few decades, Nesterov's series of seminal works \cite{n1983a,n1988on,n2005smooth} (see also \cite{n2003introductory}) on accelerated gradient methods have inspired various extensions and variants, such as the classical  accelerated proximal point method of G\"{u}ler \cite{g1992new} as well as its recent Bregman extension \cite{yh2020bregman}, the accelerated interior gradient algorithm of Auslender and Teboulle \cite{at2006interior}, and the recent inertial variants of the Bregman proximal gradient method in \cite{gp2018perturbed,hrx2018accelerated}. Motivated by these studies, it is natural for us to explore whether and how our iBPPA can be accelerated. Here, we should point out that the convergence rate (in terms of the objective function value) of  PPA-type methods, including our iBPPA, can usually be improved by simply choosing smaller proximal parameters (see Remark \ref{rek-comp-BPPA}). However, a smaller proximal parameter often leads to a harder and possibly more ill-conditioned subproblem, which may not be efficiently solvable as in the case of many OT related problems. Therefore, it is important to develop a possibly accelerated variant of our iBPPA without explicitly resorting to using smaller proximal parameters.

The contributions of this paper are summarized as follows. \vspace{1mm}
\begin{itemize}[leftmargin=1.15cm]
\item[{\bf 1.}] We have proposed a new stopping condition for inexactly solving the subproblems in iBPPA. This condition can circumvent the difficulty of demanding the interior feasibility or requiring the exact computation of $\partial f$ in existing inexact conditions. Moreover, it is flexible enough to fit different scenarios, and covers conditions \eqref{BPPAcond1-Te} and \eqref{BPPAcond1-Ec} as special cases. The iteration complexity of $O(1/k)$ and the convergence of the sequence are also established for our iBPPA under some mild conditions; see Section \ref{sec-iBPPA}.

\vspace{0.5mm}
\item[{\bf 2.}] We have developed an inertial variant of our iBPPA, denoted by V-iBPPA, based on Nesterov's acceleration technique. By making use of the quadrangle scaling property of the Bregman distance (see Definition \ref{defQSP}), we show that the V-iBPPA possesses an iteration complexity of $O(1/k^{\lambda})$ under a proper inexactness control, where $\lambda\geq1$ is a quadrangle scaling exponent; see Theorem \ref{thm-comp-fk2}. Moreover, when the proximal parameter is a constant and the kernel function is strongly convex with Lipschitz continuous gradient (hence $\lambda=2$), our V-iBPPA achieves a faster rate of $O(1/k^2)$ just like the existing accelerated inexact proximal point algorithms in, for example, \cite{g1992new,ms2013accelerated}.

\vspace{0.5mm}
\item[{\bf 3.}] We have also conducted numerical experiments to evaluate the performances of our iBPPA and V-iBPPA under different inexactness settings, in comparison to the inexact hybrid proximal extragradient methods of Solodov and Svaiter \cite{ss1999hybridap,ss2000inexact}. The computational results empirically verify the improved performance of our V-iBPPA and demonstrate the promising potential of (V-)iBPPA for solving OT-related problems.
\end{itemize}

\vspace{1mm}
The rest of this paper is organized as follows. We present notation and preliminaries in Section \ref{secnot}. We then describe a new iBPPA for solving \eqref{orgpro} and establish the convergence results in Section \ref{sec-iBPPA}. A concrete application of our iBPPA to the standard OT problem is given in Section \ref{sec-ot}. We next develop an inertial variant of our iBPPA by employing Nesterov's acceleration technique in Section \ref{sec-accBPPA}. Some preliminary numerical results are reported in Section \ref{secnum}, with some concluding remarks given in Section \ref{seccon}.

%%%%%%%%%%%%%%%%%%%%%%%%%%%%%%%%%%%%%%%%%%%%%%
\section{Notation and preliminaries}\label{secnot}

Assume that $f: \mathbb{E} \rightarrow (-\infty, \infty]$ is a proper closed convex function. For a given $\nu \geq 0$, the $\nu$-subdifferential of $f$ at $\bm{x}\in{\rm dom}\,f :=\{\bm{x}\in\mathbb{E} : f(\bm{x}) < \infty\}$
is defined by $\partial_\nu f(\bm{x}):=\{\bm{d}\in\mathbb{E}: f(\bm{y}) \geq f(\bm{x}) + \langle \bm{d}, \,\bm{y}-\bm{x}  \rangle -\nu, ~\forall\,\bm{y}\in\mathbb{E}\}$, and when $\nu=0$, $\partial_\nu f$ is simply denoted by $\partial f$. The conjugate function of $f$ is the function $f^*: \mathbb{E} \rightarrow (-\infty,\infty]$ defined by $f^*(\bm{y}):=\sup\left\{\langle \bm{y},\,\bm{x}\rangle-f(\bm{x}): \bm{x}\in\mathbb{E}\right\}$. A proper closed convex function $f$ is essentially smooth if (i) $\mathrm{int}\,\mathrm{dom}\,f$ is not empty; (ii) $f$ is differentiable on $\mathrm{int}\,\mathrm{dom}\,f$; (iii) $\|\nabla f(\bm{x}_k)\|\to\infty$ for every sequence $\{\bm{x}_k\}$ in $\mathrm{int}\,\mathrm{dom}\,f$ converging to a boundary point of $\mathrm{int}\,\mathrm{dom}\,f$; see \cite[page 251]{r1970convex}.

For a vector $\bm{x}\in\mathbb{R}^n$, $x_i$ denotes its $i$-th entry, $\mathrm{Diag}(\bm{x})$ denotes the diagonal matrix whose $i$th diagonal entry is $x_i$, $\|\bm{x}\|$ denotes its Euclidean norm. For a matrix $A\in\mathbb{R}^{m\times n}$, $a_{ij}$ denotes its $(i,j)$th entry, $A_{:j}$ denotes its $j$th column, $\|A\|_F$ denotes its Fr\"{o}benius norm. For a closed convex set $\mathcal{X}\subseteq\mathbb{E}$, its indicator function $\delta_{\mathcal{X}}$ is defined by $\delta_{\mathcal{X}}(\bm{x})=0$ if $\bm{x}\in\mathcal{X}$ and $\delta_{\mathcal{X}}(\bm{x})=+\infty$ otherwise. The distance from a point $\bm{x}$ to $\mathcal{X}$ is defined by
$\mathrm{dist}(\bm{x},\,\mathcal{X}):=\inf_{\bm{y}\in\mathcal{X}}\|\bm{y}-\bm{x}\|$.

Given a proper closed strictly convex function $\phi: \mathbb{E} \rightarrow (-\infty,\infty]$, finite at $\bm{x}$, $\bm{y}$ and differentiable at $\bm{y}$, the Bregman distance \cite{b1967relaxation} between $\bm{x}$ and $\bm{y}$ associated with the kernel function $\phi$ is defined as
\begin{equation*}
\mathcal{D}_{\phi}(\bm{x}, \,\bm{y}) := \phi(\bm{x}) - \phi(\bm{y}) - \langle \nabla\phi(\bm{y}), \,\bm{x} - \bm{y} \rangle.
\end{equation*}
It is easy to see that $D_{\phi}(\bm{x}, \,\bm{y})\geq0$ and equality holds if and only if $\bm{x}=\bm{y}$ due to the strictly convexity of $\phi$. When $\mathbb{E}=\mathbb{R}^n$ and $\phi(\cdot)=\frac{1}{2}\|\cdot\|^2$, $\mathcal{D}_{\phi}(\cdot,\cdot)$ recovers the half squared Euclidean distance. Moreover, one can easily verify the following identity.

\begin{lemma}[{\bf Four points identity}]\label{lemfourId}
Suppose that a proper closed strictly convex function $\phi: \mathbb{E} \rightarrow (-\infty,\infty]$ is finite at $\bm{a},\,\bm{b},\,\bm{c},\,\bm{d}$ and differentiable at $\bm{a},\,\bm{b}$. Then,
\begin{equation}\label{fourId}
\langle \nabla\phi(\bm{a})-\nabla\phi(\bm{b}),\,\bm{c}-\bm{d} \rangle = \mathcal{D}_{\phi}(\bm{c},\,\bm{b}) + \mathcal{D}_{\phi}(\bm{d},\,\bm{a}) - \mathcal{D}_{\phi}(\bm{c},\,\bm{a}) - \mathcal{D}_{\phi}(\bm{d},\,\bm{b}).
\end{equation}
\end{lemma}

We next recall the definition of a Bregman function, which plays an important role in the convergence analysis of the Bregman-distance-based method.

\begin{definition}[{\bf Bregman function
\cite[Definition 2.1]{cl1981iterative}}]\label{defBregfun}
Let $\mathcal{S}\subseteq\mathbb{E}$ be a nonempty open convex set with its closure denoted as $\overline{\mathcal{S}}$. We say that $\phi:\overline{\mathcal{S}} \mapsto \mathbb{R}$ is a Bregman function with zone $\mathcal{S}$ if the following conditions hold. \vspace{1mm}
\begin{itemize}
\item[{\bf (B1)}] $\phi$ is strictly convex and continuous on $\overline{\mathcal{S}}$.

\item[{\bf (B2)}] $\phi$ is continuously differentiable on $\mathcal{S}$.

\item[{\bf (B3)}] The left partial level set $\mathcal{L}(\bm{y},\,\alpha)=\left\{\bm{x}\in\overline{\mathcal{S}}: \mathcal{D}_{\phi}(\bm{x},\,\bm{y})\leq\alpha\right\}$ is bounded for every $\bm{y}\in\mathcal{S}$ and $\alpha\in\mathbb{R}$. Moreover, the right partial level set $\mathcal{R}(\bm{x},\,\alpha)=\left\{\bm{y}\in\mathcal{S}:
    \mathcal{D}_{\phi}(\bm{x},\,\bm{y})\leq\alpha\right\}$ is bounded for every $\bm{x}\in\overline{\mathcal{S}}$ and $\alpha\in\mathbb{R}$.

\item[{\bf (B4)}] If $\{\bm{y}^k\}\subseteq\mathcal{S}$ converges to some $\bm{y}^{*}\in\overline{\mathcal{S}}$, then $\mathcal{D}_{\phi}(\bm{y}^{*},\,\bm{y}^k)\to0$.

\item[{\bf (B5)}] \textbf{(Convergence consistency)} If $\{\bm{x}^k\}\subseteq\overline{\mathcal{S}}$ and $\{\bm{y}^k\}\subseteq\mathcal{S}$ are two sequences such that $\{\bm{x}^k\}$ is bounded, $\bm{y}^k\to \bm{y}^{*}$ and $\mathcal{D}_{\phi}(\bm{x}^k,\,\bm{y}^k)\to0$, then $\bm{x}^k\to \bm{y}^{*}$.
\end{itemize}
\end{definition}

Some remarks are in order concerning this definition. The above definition was originally introduced by Censor and Lent \cite{cl1981iterative}. However, it has already been noticed (for example, by Eckstein
\cite[Section 2]{e1998approximate}) that the condition on the boundedness of the \textit{left} partial level set in (\textbf{B3}) is redundant because it follows automatically from the observation that $\mathcal{L}(\bm{y},\,0)=\{\bm{y}\}$ for all $\bm{y}\in\mathcal{S}$, the convexity of $\mathcal{D}_{\phi}(\cdot,\,\bm{y})$ and \cite[Corollary 8.7.1]{r1970convex}. Moreover, Solodov and Svaiter have shown in \cite[Theorem 2.4]{ss2000inexact} that the convergence consistency ({\bf B5}) also holds automatically as a consequence of the other conditions. But for ease of future reference, we still keep the left partial level-boundedness and ({\bf B5}) in the definition. When $\mathbb{E}=\mathbb{R}^n$, two popular Bregman functions are $\phi(\bm{x}):=\frac{1}{2}\|\bm{x}\|^2$ with zone $\mathbb{R}^n$ and  $\phi(\bm{x}):=\sum_{i=1}^nx_{i}(\log x_{i}-1)$ with zone $\mathbb{R}^n_{+}$. We refer the reader to \cite{bb1997legendre,bbc2003redundant} and
\cite[Chapter 2]{cz1997parallel} for more details and examples, as well as a brief historical review on Bregman functions.

Next, we give three supporting lemmas.

\begin{lemma}[{\cite[Section 2.2]{p1987introduction}}]\label{lemseqcon}
Suppose that $\{\alpha_k\}_{k=0}^{\infty}\subseteq\mathbb{R}$ and $\{\beta_k\}_{k=0}^{\infty}\subseteq\mathbb{R}_+$ are two sequences such that $\{\alpha_k\}$ is bounded from below, $\sum_{k=0}^{\infty} \beta_k < \infty$, and $\alpha_{k+1} \leq \alpha_{k} + \beta_k$ holds for all $k$. Then, $\{\alpha_k\}$ is convergent.
\end{lemma}

\begin{lemma}[{\cite[Lemma 3.5]{l1995on}}]\label{lemseqcon2}
Suppose that $\{\lambda_n\}_{n=0}^{\infty}\subseteq\mathbb{R}_+$ and $\{\alpha_n\}_{n=0}^{\infty}\subseteq\mathbb{R}$ are two sequences. Let $t_n:=\sum^n_{k=0}\lambda_k$ and $\beta_n:=t_n^{-1}\sum^n_{k=0}\lambda_k\alpha_k$. If $t_n\to+\infty$, then \vspace{1mm}
\begin{itemize}[leftmargin=0.6cm]
\item[{\rm(i)}] $\liminf\limits_{n\to+\infty}\alpha_n \leq \liminf\limits_{n\to+\infty}\beta_n \leq \limsup\limits_{n\to+\infty}\beta_n \leq \limsup\limits_{n\to+\infty}\alpha_n$;

\item[{\rm(ii)}] moreover, if $\alpha:=\lim\limits_{n\to+\infty}\alpha_n$ exists, then $\beta_n\to\alpha$. (Silverman-Toeplitz theorem).
\end{itemize}
\end{lemma}

\begin{lemma}\label{lem-supp}
Let $g: \mathbb{E}\to (-\infty,\infty]$ be a proper closed convex function and $\phi$ be a convex, essentially smooth function. For any $\bm{y}\in\mathrm{int}\,\mathrm{dom}\,\phi$ and $\varepsilon>0$, let $P^{\varepsilon}_{\bm{y}}(\bm{x}):=g(\bm{x})+\varepsilon\mathcal{D}_{\phi}(\bm{x},\,\bm{y})$. Suppose that an optimal solution (denoted by $\mathcal{J}_{\varepsilon}\bm{y}$) of problem $\min\limits_{\bm{x}}\big\{P^{\varepsilon}_{\bm{y}}(\bm{x})\big\}$ exists. Then, we have
\begin{equation}\label{suppineq}
P^{\varepsilon}_{\bm{y}}(\bm{u}) - P^{\varepsilon}_{\bm{y}}(\mathcal{J}_{\varepsilon}\bm{y}) \geq \varepsilon\mathcal{D}_{\phi}(\bm{u}, \,\mathcal{J}_{\varepsilon}\bm{y}), \quad \forall\,\bm{u}\in\mathrm{dom}\,g\cap\mathrm{dom}\,\phi.
\end{equation}
Moreover, if $g$ is an affine function, then the above inequality holds with equality.
\end{lemma}
\begin{proof}
%Since the proof is straightforward, we omit it to save some space.
Since $\phi$ is essentially smooth, then $\mathcal{J}_{\varepsilon}\bm{y}$ must lie in $\mathrm{int}\,\mathrm{dom}\,\phi$ and satisfy
\begin{equation*}
0 \in \partial g(\mathcal{J}_{\varepsilon}\bm{y}) + \varepsilon \left(\nabla \phi(\mathcal{J}_{\varepsilon}\bm{y}) - \nabla \phi(\bm{y})\right) ~ \Longleftrightarrow ~
-\varepsilon \left(\nabla \phi(\mathcal{J}_{\varepsilon}\bm{y}) - \nabla \phi(\bm{y})\right) \in \partial g(\mathcal{J}_{\varepsilon}\bm{y}).
\end{equation*}
From the convexity of $g$, for any $\bm{u}\in\mathrm{dom}\,g\cap\mathrm{dom}\,\phi$,
\begin{equation*}
\begin{aligned}
g(\bm{u})
&\geq g(\mathcal{J}_{\varepsilon}\bm{y}) - \varepsilon\langle \nabla \phi(\mathcal{J}_{\varepsilon}\bm{y}) - \nabla \phi(\bm{y}), \,\bm{u} - \mathcal{J}_{\varepsilon}\bm{y} \rangle \\
&= g(\mathcal{J}_{\varepsilon}\bm{y}) - \varepsilon \left(\mathcal{D}_{\phi}(\bm{u},\,\bm{y})- \mathcal{D}_{\phi}(\bm{u},\,\mathcal{J}_{\varepsilon}\bm{y})
-\mathcal{D}_{\phi}(\mathcal{J}_{\varepsilon}\bm{y},\,\bm{y})\right),
\end{aligned}
\end{equation*}
where the equality follows from \eqref{fourId}. Then, rearranging the above inequality results in \eqref{suppineq}. Moreover, when $g$ is an affine function, it is easy to see that $g(\bm{u})=g(\mathcal{J}_{\varepsilon}\bm{y}) + \langle \nabla g(\mathcal{J}_{\varepsilon}\bm{y}), \,\bm{u} - \mathcal{J}_{\varepsilon}\bm{y} \rangle$ for any $\bm{u}\in\mathrm{dom}\,\phi$. This together with augments similar to those just presented above implies the equality in \eqref{suppineq}. We \blue{completed} the proof.
\end{proof}

Finally, we make some blanket assumptions on our problem \eqref{orgpro} and the kernel function $\phi$, which are essential for guaranteeing the well-definedness of our problem and subproblems as well as the convergence of the presented algorithms.

\begin{assumption}\label{assumps}
Problem \eqref{orgpro} and the kernel function $\phi$ satisfy the following assumptions.
\begin{itemize}[leftmargin=0.8cm]
\item[{\rm(i)}] $\mathrm{dom}f\cap\mathcal{C}$ is nonempty.

\item[{\rm(ii)}] $\rho:=\max\left\{\|\bm{x}-\bm{y}\| : \bm{x},\,\bm{y}\in\mathrm{dom}\,f\cap\overline{\mathcal{C}}\right\}<\infty$.

\item[{\rm(iii)}] $\mathrm{dom}\,\phi = \overline{\mathcal{C}}$, $\phi$ is a Bregman function with zone $\mathcal{C}$ and $\phi$ is essentially smooth.
\end{itemize}
\end{assumption}

One can see from Assumption \ref{assumps}(i)\&(ii) that $\mathrm{dom}(f+\delta_{\overline{\mathcal{C}}})$ is nonempty and $f+\delta_{\overline{\mathcal{C}}}$ is level-bounded. Hence, a solution of problem \eqref{orgpro} exists; see, for example, \cite[Theorem 1.9]{rw1998variational}. Note also that Assumption \ref{assumps}(ii) actually requires the feasible set of problem \eqref{orgpro} to be bounded. This property then ensures the existence of a solution of each subproblem and the boundedness of sequence generated by our algorithm. Some weaker assumptions are possible, but involve a bit more analysis when we deal with the convergence of the iBPPA; see Remark \ref{rek-bd-BPPA}. Here, we simply impose Assumption \ref{assumps}(ii). This assumption can be satisfied by many practical problems, for example, the standard OT problem \eqref{otproblem} and its various related problems \cite{pc2019computational}.

%%%%%%%%%%%%%%%%%%%%%%%%%%%%%%%%%%%%%%%%%%%%%%%%%%%%%%%%
\section{A new inexact Bregman proximal point algorithm}\label{sec-iBPPA}

In this section, we develop a new inexact Bregman proximal point algorithm (iBPPA) for solving problem \eqref{orgpro}. The complete framework is presented as Algorithm \ref{algBPPA-gen}.

\begin{algorithm}[h]
\caption{An inexact Bregman proximal point algorithm (iBPPA) for \eqref{orgpro}}\label{algBPPA-gen}
\textbf{Input:} Let $\{\gamma_k\}_{k=0}^{\infty}$, $\{\eta_k\}_{k=0}^{\infty}$, $\{\mu_k\}_{k=0}^{\infty}$, $\{\nu_k\}_{k=0}^{\infty}$ be four sequences of nonnegative scalars. Choose $\bm{x}^0=\widetilde{\bm{x}}^{0}\in\mathcal{C}$ arbitrarily and a kernel function $\phi$. Set $k=0$.  \\
\textbf{while} a termination criterion is not met, \textbf{do} %\vspace{-2mm}
\begin{itemize}[leftmargin=2cm]
\item[\textbf{Step 1}.] Find a pair $(\bm{x}^{k+1},\,\widetilde{\bm{x}}^{k+1})$ by approximately solving the following problem
    \begin{equation}\label{BPPAsubpro-gen}
    \min\limits_{\bm{x}}~
    f(\bm{x}) + \gamma_k\mathcal{D}_{\phi}(\bm{x}, \,\bm{x}^k), \vspace{-1mm}
    \end{equation}
    such that $\bm{x}^{k+1}\in\mathcal{C}$, $\widetilde{\bm{x}}^{k+1}\in\mathrm{dom}\,f\cap\overline{\mathcal{C}}$ and
    \begin{equation}\label{BPPAcond-gen}
    \begin{aligned}
    &\Delta^{k} \in \partial_{\nu_k} f(\widetilde{\bm{x}}^{k+1}) + \gamma_k\big(\nabla\phi(\bm{x}^{k+1}) - \nabla\phi(\bm{x}^{k})\big) \\
    &~~\mathrm{with}~~\|\Delta^{k}\| \leq \eta_k, ~~\mathcal{D}_{\phi}(\widetilde{\bm{x}}^{k+1}, \,\bm{x}^{k+1}) \leq \mu_k.
    \end{aligned}
    \end{equation}

\item [\textbf{Step 2}.] Set $k = k+1$ and go to \textbf{Step 1}.
\end{itemize}
\textbf{end while}  \\
\textbf{Output}: $(\bm{x}^{k},\,\widetilde{\bm{x}}^{k})$ \vspace{0.5mm}
\end{algorithm}

In the spirit of the PPA-type method, our iBPPA in Algorithm \ref{algBPPA-gen} basically solves the original problem \eqref{orgpro} via approximately solving a sequence of subproblems \eqref{BPPAsubpro-gen} each involving a Bregman proximal term associated with the kernel function $\phi$. Since $\mathrm{dom}\,\phi = \overline{\mathcal{C}}$ by Assumption \ref{assumps}(iii), the constraint $\bm{x} \in \overline{\mathcal{C}}$ can be removed in \eqref{BPPAsubpro-gen}. Moreover, under Assumption \ref{assumps}, one can see that, at the $k$-th iteration, the solution $\bm{x}^{k,*}$ of subproblem \eqref{BPPAsubpro-gen} exists and lies in $\mathcal{C}$ ($=\mathrm{dom}\,\phi$). Indeed, Assumption \ref{assumps}(ii) and $\mathrm{dom}\,\phi = \overline{\mathcal{C}}$ imply that the objective function in subproblem \eqref{BPPAsubpro-gen} is level-bounded. Thus, a solution exists
(by \cite[Theorem 1.9]{rw1998variational}) and must be unique since $\phi$ is strictly convex (by Assumption \ref{assumps}(iii) and condition ({\bf B1})). The essential smoothness of $\phi$ (by Assumption \ref{assumps}(iii)) and Assumption \ref{assumps}(i) further imply that $\bm{x}^{k,*}$ cannot be at the boundary of $\mathcal{C}$. Hence, the subproblem and iterate are well-defined. Our inexact condition \eqref{BPPAcond-gen} always holds at $\bm{x}^{k+1}=\widetilde{\bm{x}}^{k+1}=\bm{x}^{k,*}$ and thus it is achievable.

The inexact condition \eqref{BPPAcond-gen} is rather broad for covering some existing approximation conditions. When $\nu_k\equiv\eta_k\equiv\mu_k\equiv0$, $\bm{x}^{k+1}$ ($=\widetilde{\bm{x}}^{k+1}$) is obviously the exact optimal solution of subproblem \eqref{BPPAsubpro-gen}. In this case, our iBPPA reduces to the classical exact BPPA \cite{cz1992proximal,cz1997parallel,ct1993convergence,e1993nonlinear}. When $\eta_k\equiv\mu_k\equiv0$, condition \eqref{BPPAcond-gen} reduces to condition \eqref{BPPAcond1-Te} studied in \cite{bis1997enlargement,k1997proximal,t1997convergence}. Moreover, when $\nu_k\equiv\mu_k\equiv0$, condition \eqref{BPPAcond-gen} reduces to condition \eqref{BPPAcond1-Ec} studied by Eckstein in \cite{e1998approximate}. More importantly, the inexact condition \eqref{BPPAcond-gen} can bypass the underlying difficulty of demanding interior feasibility, which appears to be often overlooked in the literature.

As we have mentioned in the introduction, to check either condition \eqref{BPPAcond1-Te} or \eqref{BPPAcond1-Ec}, one has to compute an approximate solution $\bm{x}^{k+1}$ that belongs to both $\mathrm{dom}\,f$ (for the nonemptyness of $\partial_{\nu_k}f(\bm{x}^{k+1})$ or $\partial f(\bm{x}^{k+1})$) and $\mathrm{dom}\,\nabla\phi$ (for the well-definedness of $\nabla\phi(\bm{x}^{k+1})$). However, in practice, even finding a point in $\mathrm{dom}\,f\cap\mathrm{dom}\,\nabla\phi$ can be nontrivial when $\mathrm{dom}\,f$ is not a simple convex set. Thus, in this case, condition \eqref{BPPAcond1-Te} or \eqref{BPPAcond1-Ec} may no longer be suitable. Our inexact condition \eqref{BPPAcond-gen}  allows one to evaluate $\partial_{\nu_k}f$ and $\nabla\phi$ at two different points to deal with $\mathrm{dom}\,f$ and $\mathrm{dom}\,\nabla\phi$ separately. It is also interesting to compare our condition with condition \eqref{BPPAcond1-SS}. Both conditions allow the error tolerance criteria to be checked at two different points. But the mechanisms are different. Our condition \eqref{BPPAcond-gen} aims to relax the stringent requirement $\bm{x}^{k+1}\in\mathrm{dom}\,f\cap\mathrm{dom}\,\nabla\phi$, while condition \eqref{BPPAcond1-SS} inherits the idea of a hybrid approach developed by Solodov and Svaiter \cite{ss1999hybridap,ss1999hybridpr,ss2000error,ss2001unified}
(now known as the hybrid proximal extragradient (HPE) method \cite{ms2010complexity,ms2013accelerated}) to use an intermediary point for computing $\bm{x}^{k+1}$. The latter condition is constructive and does not need the usual summable-error requirement. However, it generally needs the exact computation of an element in $\partial f$ at an intermediary point. Note that, when $\phi$ is chosen as the classical squared Euclidean norm, this exactness requirement has been relaxed by allowing an element in $\partial_{\nu} f$ for some $\nu\geq0$ (see, e.g., \cite{ms2010complexity,ss1999hybridap}), but, to our knowledge, it is still not clear whether such a requirement can be relaxed for a general kernel function. This exactness requirement may limit the application of condition \eqref{BPPAcond1-SS}. Moreover, when employing condition \eqref{BPPAcond1-SS}, one has to compute $\bm{x}^{k+1}$ via an extragradient step to guarantee the convergence. In contrast, our condition \eqref{BPPAcond-gen} appears to be more straightforward and flexible. Later, we shall illustrate the potential advantages of our condition through a concrete example on the standard OT problem in Section \ref{sec-ot}.

We next establish the convergence of our iBPPA in Algorithm \ref{algBPPA-gen}. Our analysis is inspired by several existing works (see, for example, \cite{e1998approximate,t1997convergence}). We start by establishing a sufficient-descent-like property in the following lemma.

\begin{lemma}[\textbf{Sufficient-descent-like property}]\label{lem-suffdes-iBPPA}
Let $\{\bm{x}^{k}\}$ and $\{\widetilde{\bm{x}}^{k}\}$ be the sequences generated by the iBPPA in Algorithm \ref{algBPPA-gen}. Then, for any $\bm{u}\in\mathrm{dom}\,f\cap\overline{\mathcal{C}}$,
\begin{equation}\label{desineq-iBPPA}
\begin{aligned}
f(\widetilde{\bm{x}}^{k+1})
&\leq f(\bm{u}) + \gamma_k\big( \mathcal{D}_{\phi}(\bm{u},\,\bm{x}^k)
-\mathcal{D}_{\phi}(\bm{u},\,\bm{x}^{k+1})
-\mathcal{D}_{\phi}(\widetilde{\bm{x}}^{k+1},\,\bm{x}^k)\big) \\
&\quad + \langle \Delta^{k}, \,\widetilde{\bm{x}}^{k+1}-\bm{u} \rangle + \gamma_k\mu_k + \nu_k.
\end{aligned}
\end{equation}
\end{lemma}
\begin{proof}
From condition \eqref{BPPAcond-gen}, there exists a
$\bm{d}^{k+1}\in\partial_{\nu_k} f(\widetilde{\bm{x}}^{k+1})$ such that
$\Delta^{k} = \bm{d}^{k+1} + \gamma_k\big(\nabla\phi(\bm{x}^{k+1}) - \nabla\phi(\bm{x}^{k})\big)$. Then, for any $\bm{u}\in\mathrm{dom}f\cap\overline{\mathcal{C}}$, we see that
\begin{equation*}
\begin{aligned}
f(\bm{u})
&\geq f(\widetilde{\bm{x}}^{k+1}) + \langle \bm{d}^{k+1}, \,\bm{u} - \widetilde{\bm{x}}^{k+1} \rangle - \nu_k \\
&= f(\widetilde{\bm{x}}^{k+1}) + \langle \Delta^{k} - \gamma_k\big(\nabla\phi(\bm{x}^{k+1}) - \nabla\phi(\bm{x}^{k})\big), \,\bm{u}-\widetilde{\bm{x}}^{k+1} \rangle - \nu_k,
\end{aligned}
\end{equation*}
which implies that
\begin{equation*}
f(\widetilde{\bm{x}}^{k+1}) \leq f(\bm{u}) + \gamma_k\langle \,\nabla\phi(\bm{x}^{k+1}) - \nabla\phi(\bm{x}^{k}), \,\bm{u}-\widetilde{\bm{x}}^{k+1} \,\rangle + \langle \Delta^{k}, \,\widetilde{\bm{x}}^{k+1}-\bm{u} \rangle + \nu_k.
\end{equation*}
Note from the four points identity \eqref{fourId} and $\mathcal{D}_{\phi}(\widetilde{\bm{x}}^{k+1},\,\bm{x}^{k+1})\leq\mu_k$ in \eqref{BPPAcond-gen} that
\begin{equation*}
\langle\nabla\phi(\bm{x}^{k+1}) - \nabla\phi(\bm{x}^{k}), \,\bm{u}-\widetilde{\bm{x}}^{k+1}\rangle
\leq \mathcal{D}_{\phi}(\bm{u},\,\bm{x}^k)
-\mathcal{D}_{\phi}(\bm{u},\,\bm{x}^{k+1})
-\mathcal{D}_{\phi}(\widetilde{\bm{x}}^{k+1},\,\bm{x}^k) + \mu_k.
\end{equation*}
Combining the above two inequalities, we obtain \eqref{desineq-iBPPA}.
\end{proof}

Based on the sufficient-descent-like property, we can estimate the iteration complexity of our iBPPA in terms of the function value as follows.

\begin{theorem}[\textbf{Iteration complexity of the iBPPA}]\label{thmcompBPPA}
Let $\{\bm{x}^{k}\}$ and $\{\widetilde{\bm{x}}^{k}\}$ be the sequences generated by the iBPPA in Algorithm \ref{algBPPA-gen}. Then, for any optimal solution $\bm{x}^*$ of problem \eqref{orgpro}, we have
\begin{equation}\label{ineq-compBPPA}
\begin{aligned}
f(\widetilde{\bm{x}}^{N}) - f(\bm{x}^*) 
\leq \sigma_{N-1}^{-1}\!\left(\mathcal{D}_{\phi}(\bm{x}^*,\bm{x}^{0})
+ {\textstyle\sum^{N-1}_{k=0}}\mu_{k}
+ {\textstyle\sum^{N-1}_{k=0}}\gamma_{k}^{-1}(\rho\eta_{k}+\nu_{k})
+ {\textstyle\sum^{N-1}_{k=0}}\sigma_{k-1}\xi_{k}\right)\!,
\end{aligned}
\end{equation}
where $\sigma_{-1}:=0$, $\sigma_k:=\sum^k_{i=0}\gamma_{i}^{-1}$ and $\xi_k := f(\widetilde{\bm{x}}^{k+1}) - f(\widetilde{\bm{x}}^{k}) \leq \gamma_k(\mu_{k-1}+\mu_k) + \rho\eta_k + \nu_k$ for every integer $k\geq0$. Moreover, if the summable-error condition that $\max\big\{\sum\mu_k, \,\sum\gamma_k^{-1}\nu_k, \,\sum\gamma_k^{-1}\eta_k$, $\sum\sigma_{k-1}\xi_{k}\big\}<\infty$ holds, then
\begin{equation*}
f(\widetilde{\bm{x}}^{N}) - f(\bm{x}^*) \leq O\left(\frac{1}{\sum^{N-1}_{k=0}\gamma_{k}^{-1}}\right).
\end{equation*}
\end{theorem}
\begin{proof}
First, we see from \eqref{desineq-iBPPA} in Lemma \ref{lem-suffdes-iBPPA} with $\bm{u}=\widetilde{\bm{x}}^k$ that
\begin{equation}\label{desineq-Pik}
\begin{aligned}
&\xi_k
:=\, f(\widetilde{\bm{x}}^{k+1}) - f(\widetilde{\bm{x}}^{k}) \\
&\leq\, \gamma_k\big( \mathcal{D}_{\phi}(\widetilde{\bm{x}}^k,\,\bm{x}^k)
-\mathcal{D}_{\phi}(\widetilde{\bm{x}}^k,\,\bm{x}^{k+1}) -\mathcal{D}_{\phi}(\widetilde{\bm{x}}^{k+1},\,\bm{x}^k) \big)
+ \langle\Delta^{k},\,\widetilde{\bm{x}}^{k+1}-\widetilde{\bm{x}}^{k}\rangle + \gamma_k\mu_k + \nu_k \\
&\leq\,\gamma_k\mathcal{D}_{\phi}(\widetilde{\bm{x}}^k,\,\bm{x}^k)
+ |\langle\Delta^{k},\,\widetilde{\bm{x}}^{k+1}-\widetilde{\bm{x}}^{k}\rangle| + \gamma_k\mu_k + \nu_k \\
&\leq\, \gamma_k(\mu_{k-1}+\mu_k) + \rho\eta_k + \nu_k,
\end{aligned}
\end{equation}
where the last inequality follows from condition \eqref{BPPAcond-gen} and $\|\widetilde{\bm{x}}^{k+1}-\widetilde{\bm{x}}^{k}\|\leq\rho$ (due to $\widetilde{\bm{x}}^{k+1},\,\widetilde{\bm{x}}^{k}\in\mathrm{dom}f\,\cap\,\overline{\mathcal{C}}$ and Assumption \ref{assumps}(ii)). Moreover, for any $k\geq0$,
\begin{equation*}
\begin{aligned}
f(\widetilde{\bm{x}}^{k+1}) = f(\widetilde{\bm{x}}^{k}) + \xi_{k}
&~\Longrightarrow~
(\sigma_{k}-\gamma_{k}^{-1})f(\widetilde{\bm{x}}^{k+1}) = \sigma_{k-1}f(\widetilde{\bm{x}}^{k}) + \sigma_{k-1}\xi_{k} \\
&~\Longrightarrow~
\gamma_{k}^{-1}f(\widetilde{\bm{x}}^{k+1}) = \sigma_{k}f(\widetilde{\bm{x}}^{k+1}) - \sigma_{k-1}f(\widetilde{\bm{x}}^{k}) - \sigma_{k-1}\xi_{k}.
\end{aligned}
\end{equation*}
Summing the above equality from $k=0$ to $k=N-1$ results in
\begin{equation}\label{sumFbd-gen}
{\textstyle\sum^{N-1}_{k=0}}\gamma_{k}^{-1}f(\widetilde{\bm{x}}^{k+1}) = \sigma_{N-1} f(\widetilde{\bm{x}}^{N}) - {\textstyle\sum^{N-1}_{k=0}}\sigma_{k-1}\xi_{k}.
\end{equation}
Let $\bm{x}^*$ be an arbitrary optimal solution of problem \eqref{orgpro}. Then, using \eqref{desineq-iBPPA} with $\bm{u}=\bm{x}^*$ again, we see that, for all $k\geq 0$,
\begin{equation*}
\begin{aligned}
&\quad f(\widetilde{\bm{x}}^{k+1}) - f(\bm{x}^*) \\
&\leq \gamma_{k}\big( \mathcal{D}_{\phi}(\bm{x}^*,\bm{x}^{k})
-\mathcal{D}_{\phi}(\bm{x}^*,\bm{x}^{k+1})
-\mathcal{D}_{\phi}(\widetilde{\bm{x}}^{k+1},\bm{x}^{k})\big)
+ \langle\Delta^{k},\,\widetilde{\bm{x}}^{k+1}-\bm{x}^*\rangle
+ \gamma_{k}\mu_{k} + \nu_{k} \\
&\leq \gamma_{k}\big( \mathcal{D}_{\phi}(\bm{x}^*,\bm{x}^{k})
-\mathcal{D}_{\phi}(\bm{x}^*,\bm{x}^{k+1})\big) + \rho\eta_{k}
+ \gamma_{k}\mu_{k} + \nu_{k},
\end{aligned}
\end{equation*}
where the last inequality follows from $\Delta^{k}\leq\eta_k$ and $\|\widetilde{\bm{x}}^{k+1}-\bm{x}^*\|\leq\rho$. Thus, we get
\begin{equation*}
\gamma_{k}^{-1}f(\widetilde{\bm{x}}^{k+1}) - \gamma_{k}^{-1}f(\bm{x}^*)
\leq \mathcal{D}_{\phi}(\bm{x}^*,\,\bm{x}^{k})
-\mathcal{D}_{\phi}(\bm{x}^*,\,\bm{x}^{k+1}) + \mu_{k} + \gamma_{k}^{-1}(\rho\eta_{k} + \nu_{k}).
\end{equation*}
Summing the above inequality from $k=0$ to $k=N-1$, we obtain that
\begin{equation}\label{sumbdnew}
\begin{aligned}
&\quad~{\textstyle\sum^{N-1}_{k=0}}\gamma_{k}^{-1}f(\widetilde{\bm{x}}^{k+1}) - \sigma_{N-1} f(\bm{x}^*) \\
&\leq \mathcal{D}_{\phi}(\bm{x}^*,\,\bm{x}^{0})-\mathcal{D}_{\phi}(\bm{x}^*,\,\bm{x}^{N}) + {\textstyle\sum^{N-1}_{k=0}}\mu_{k} + {\textstyle\sum^{N-1}_{k=0}}\gamma_{k}^{-1}(\rho\eta_{k}+\nu_{k}) \\
&\leq \mathcal{D}_{\phi}(\bm{x}^*,\,\bm{x}^{0})
+ {\textstyle\sum^{N-1}_{k=0}}\mu_{k}
+ {\textstyle\sum^{N-1}_{k=0}}\gamma_{k}^{-1}(\rho\eta_{k}+\nu_{k}).
\end{aligned}
\end{equation}
This together with \eqref{sumFbd-gen} implies that
\begin{equation*}
\begin{aligned}
&\sigma_{N-1}(f(\widetilde{\bm{x}}^N)-f(\bm{x}^*))
= {\textstyle\sum^{N-1}_{k=0}}\gamma_{k}^{-1}f(\widetilde{\bm{x}}^{k+1}) + {\textstyle\sum^{N-1}_{k=0}}\sigma_{k-1}\xi_{k} - \sigma_{N-1}f(\bm{x}^*) \\
&\leq \mathcal{D}_{\phi}(\bm{x}^*,\,\bm{x}^{0})
+ {\textstyle\sum^{N-1}_{k=0}}\mu_{k}
+ {\textstyle\sum^{N-1}_{k=0}}\gamma_{k}^{-1}(\rho\eta_{k}+\nu_{k})
+ {\textstyle\sum^{N-1}_{k=0}}\sigma_{k-1}\xi_{k}.
\end{aligned}
\end{equation*}
Dividing the above inequality by $\sigma_{N-1}$, we can obtain \eqref{ineq-compBPPA}. The remaining result readily follows from \eqref{ineq-compBPPA} under given conditions. We then complete the proof.
\end{proof}

\begin{remark}[\textbf{Comments on iteration complexity}]\label{rek-comp-BPPA}
We see from Theorem \ref{thmcompBPPA} that, under the summable-error condition, the convergence rate of $\{f(\widetilde{\bm{x}}^k)\}$ is mainly determined by $\left(\sum\gamma_k^{-1}\right)^{-1}$. Since the choice of $\{\gamma_k\}$ can be quite flexible,  one can obtain different convergence rates of $\{f(\widetilde{\bm{x}}^k)\}$. For example, \vspace{1mm}
\begin{itemize}[leftmargin=0.8cm]
\item if $0<\underline{\gamma}\leq\gamma_k\leq\overline{\gamma}<\infty$, then
    $f(\widetilde{\bm{x}}^{N}) - f(\bm{x}^*) \leq O\left(\frac{1}{N}\right)$;

\vspace{1mm}
\item if $\gamma_k = \frac{1}{1+k}$, then $f(\widetilde{\bm{x}}^{N}) - f(\bm{x}^*) \leq O\left(\frac{1}{N^2}\right)$;

\vspace{1mm}
\item if $\gamma_k = \gamma_0\varrho^k$ with $0<\varrho<1$, then $f(\widetilde{\bm{x}}^{N}) - f(\bm{x}^*) \leq O\big(\varrho^N\big)$. \vspace{1mm}
\end{itemize}
Indeed, it is not hard to see that an arbitrarily fast convergence rate can be achieved with a proper decreasing sequence of $\{\gamma_k\}$. However, for a fast decreasing sequence of $\{\gamma_k\}$, the corresponding choices of $\{\mu_k\}$, $\{\nu_k\}$ and $\{\eta_k\}$ also become more stringent to guarantee the summable-error conditions. Thus, when applying the iBPPA for solving a specific problem, one needs to make a tradeoff between the convergence rate and the tolerable inexactness. In addition, we should mention that condition $\sum\sigma_{k-1}\xi_{k}<\infty$ is not as restrictive as it appears. For example, consider the case $0<\underline{\gamma}\leq\gamma_k\leq\overline{\gamma}<+\infty$ and for some $p>1$, $\mu_k \leq O\big(k^{-p}\big)$, $\nu_k \leq O\big(k^{-p}\big)$, $\eta_k \leq O\big(k^{-p}\big)$. Then it follows from \eqref{desineq-Pik} that $\xi_k\leq\overline{\gamma}(\mu_{k-1}+\mu_k) + \rho\eta_k + \nu_k \leq O\big(k^{-p}\big)$. This together with $\sigma_k:=\sum^k_{i=0}\gamma_{i}^{-1}\leq (k+1)\underline{\gamma}^{-1}$ implies that $\sum\sigma_{k-1}\xi_{k} \leq O(\sum k^{1-p})$. Hence, condition $\sum\sigma_{k-1}\xi_{k}<\infty$ holds whenever $p>2$. Moreover, if the function values decrease monotonically along the sequence $\{\widetilde{\bm{x}}^k\}$, as we often observe in our experiments, then $\xi_k:=f(\widetilde{\bm{x}}^{k+1})-f(\widetilde{\bm{x}}^{k})\leq0$ and the condition $\sum\sigma_{k-1}\xi_{k}<\infty$ is automatically met.
\end{remark}

We next present the main convergence results for our iBPPA.

\begin{theorem}[\textbf{Convergence of the iBPPA}]\label{thmconverBPPA}
Suppose that Assumption \ref{assumps} holds. Let $\{\bm{x}^{k}\}$ and $\{\widetilde{\bm{x}}^{k}\}$ be the sequences generated by the iBPPA in Algorithm \ref{algBPPA-gen}, and $f^*:=\min\left\{f(\bm{x}):\bm{x}\in\overline{\mathcal{C}}\right\}$. Then, the following statements hold. \vspace{1mm}
\begin{itemize}[leftmargin=0.6cm]
\item[{\rm (i)}] If $\sup_k\{\gamma_k\}<\infty$, $\sum\mu_k<\infty$, $\sum\nu_k<\infty$ and $\sum\eta_k<\infty$, then $f(\widetilde{\bm{x}}^k) \to f^*$. \vspace{1mm}

\item[{\rm (ii)}] If $\sup_k\{\gamma_k\}<\infty$, $\sum\mu_k<\infty$, $\sum\gamma_k^{-1}\nu_k<\infty$ and $\sum\gamma_k^{-1}\eta_k<\infty$, then the sequences $\{\bm{x}^{k}\}$ and $\{\widetilde{\bm{x}}^{k}\}$ converge to the same limit that is an optimal solution of problem \eqref{orgpro}.
\end{itemize}
\end{theorem}
\begin{proof}
\textit{Statement (i)}.
Let $\bm{x}^*$ be an arbitrary optimal solution of problem \eqref{orgpro}. Then, from \eqref{sumbdnew}, we have for any nonnegative integer $n$,
\begin{equation}\label{sumbdnew2}
\begin{aligned}
&\quad \sigma_{n}^{-1}{\textstyle\sum^{n}_{k=0}}\gamma_{k}^{-1}f(\widetilde{\bm{x}}^{k+1}) \\
&\leq f(\bm{x}^*) + \sigma_{n}^{-1}\mathcal{D}_{\phi}(\bm{x}^*,\,\bm{x}^{0}) + \sigma_{n}^{-1}{\textstyle\sum^{n}_{k=0}}\,\mu_{k} + \sigma_{n}^{-1}{\textstyle\sum^{n}_{k=0}}\gamma_{k}^{-1}\big(\rho\eta_{k} + \nu_{k}\big).
\end{aligned}
\end{equation}
where $\sigma_n:=\sum^n_{k=0}\gamma_{k}^{-1}$ for $n=0,1,2,\ldots$. Note that $\sigma_n\to+\infty$ since $\sup_k\{\gamma_k\}<+\infty$, and $\rho\eta_{k} + \nu_{k}\to0$ since $\sum\nu_k<\infty$ and $\sum\eta_k<\infty$. Thus, from Lemma \ref{lemseqcon2}(ii), we see that $\sigma_{n}^{-1}{\textstyle\sum^{n}_{k=0}}\gamma_{k}^{-1}(\rho\eta_{k} + \nu_{k})\to0$. This together with \eqref{sumbdnew2}, $\sum\mu_k<\infty$ and Lemma \ref{lemseqcon2}(i) implies that
\begin{equation*}
\liminf\limits_{n\to\infty}\,f(\widetilde{\bm{x}}^{n+1}) \leq \liminf\limits_{n\to\infty}\,\sigma_{n}^{-1}{\textstyle\sum^{n}_{k=0}}\gamma_{k}^{-1}f(\widetilde{\bm{x}}^{k+1})\leq f(\bm{x}^*).
\end{equation*}
Note also that $f(\widetilde{\bm{x}}^{n+1})\geq f(\bm{x}^*)$ for all $n$ since $\widetilde{\bm{x}}^{n+1}\in\mathrm{dom}\,f\cap\overline{\mathcal{C}}$. Then, we have that $\liminf_{n\to\infty}\,f(\widetilde{\bm{x}}^{n+1})=f(\bm{x}^*)$. On the other hand, $\{f(\widetilde{\bm{x}}^k)\}$ is bounded from below since $\widetilde{\bm{x}}^k\in\mathrm{dom}\,f\cap\overline{\mathcal{C}}$ for all $k$ and the solution set of problem \eqref{orgpro} is nonempty (by Assumption \ref{assumps}(i)\&(ii)). Finally, from \eqref{desineq-Pik} and Lemma \ref{lemseqcon}, together with $\sup_k\{\gamma_k\}<\infty$ and the summability of $\{\mu_k\}$, $\{\nu_k\}$, $\{\eta_k\}$, we see that $\{f(\widetilde{\bm{x}}^k)\}$ is convergent and hence $f(\widetilde{\bm{x}}^k) \to f(\bm{x}^*)$. This proves statement (i).

\textit{Statement (ii)}. First, since $\sup_k\{\gamma_k\}<\infty$, then $\inf_k\{\gamma_k^{-1}\}>0$. This together with $\sum\gamma_k^{-1}\nu_k<\infty$ and $\sum\gamma_k^{-1}\eta_k<\infty$ implies that $\sum\nu_k<\infty$ and $\sum\eta_k<\infty$. Thus, statement (i) holds. Since $\{\widetilde{\bm{x}}^k\}$ is bounded (due to $\widetilde{\bm{x}}^k\in\mathrm{dom}\,f\cap\overline{\mathcal{C}}$ and Assumption \ref{assumps}(ii)), it has at least one cluster point. Suppose that $\widetilde{\bm{x}}^{\infty}$ is a cluster point and $\{\widetilde{\bm{x}}^{k_i}\}$ is a convergent subsequence such that $\lim_{i\to\infty} \widetilde{\bm{x}}^{k_i} = \widetilde{\bm{x}}^{\infty}$. Then, from the closedness of $f$, we have that $f(\widetilde{\bm{x}}^{\infty}) \leq \liminf_{i\to\infty}f(\widetilde{\bm{x}}^{k_i}) = f^*$. Note that $\widetilde{\bm{x}}^{\infty}\in\mathrm{dom}\,f\cap\overline{\mathcal{C}}$ since $\mathrm{dom}\,f\cap\overline{\mathcal{C}}$ is closed. Hence, $\widetilde{\bm{x}}^{\infty}$ must be an optimal solution of \eqref{orgpro}.

Next, let $\bm{x}^*$ be an arbitrary optimal solution of \eqref{orgpro}. Obviously, $f(\bm{x}^*) \leq f(\widetilde{\bm{x}}^{k+1})$ for all $k\geq0$ since $\widetilde{\bm{x}}^{k+1}\in\mathrm{dom}\,f\cap\overline{\mathcal{C}}$. By setting $\bm{u}=\bm{x}^*$ in \eqref{desineq-iBPPA} and recalling $\|\widetilde{\bm{x}}^{k+1}-\bm{x}^*\|\leq\rho$ (by Assumption \ref{assumps}(ii)), we see that
\begin{equation}\label{desineq-Pistar}
\begin{aligned}
0&\leq\mathcal{D}_{\phi}(\bm{x}^*,\bm{x}^{k+1}) \\
&\leq \mathcal{D}_{\phi}(\bm{x}^*,\bm{x}^k) + \gamma_k^{-1}\big(f(\bm{x}^*)\!-\!f(\widetilde{\bm{x}}^{k+1})\big) \!-\!\mathcal{D}_{\phi}(\widetilde{\bm{x}}^{k+1},\bm{x}^k) + \mu_k + \gamma_k^{-1}(\rho\eta_k \!+\! \nu_k) \\
&\leq \mathcal{D}_{\phi}(\bm{x}^*,\bm{x}^k) + \mu_k + \gamma_k^{-1}(\rho\eta_k + \nu_k).
\end{aligned}
\end{equation}
Thus, we can conclude from \eqref{desineq-Pistar}, $\max\big\{\sum\mu_k, \,\sum\gamma_k^{-1}\nu_k, \,\sum\gamma_k^{-1}\eta_k\big\}<\infty$ and Lemma \ref{lemseqcon} that $\{\mathcal{D}_{\phi}(\bm{x}^*,\,\bm{x}^k)\}$ is convergent. From this fact and condition ({\bf B3}) in Definition \ref{defBregfun}, we further see that $\{\bm{x}^k\}$ is bounded and hence it has at least one cluster point. Suppose that $\bm{x}^{\infty}$ is a cluster point and $\{\bm{x}^{k_j}\}$ is a convergent subsequence such that $\lim_{j\to\infty} \bm{x}^{k_j} = \bm{x}^{\infty}$. Then, from the fact that $\mathcal{D}_{\phi}(\widetilde{\bm{x}}^{k_j}, \,\bm{x}^{k_j})\leq\mu_{k_j-1}\to0$, the boundedness of $\{\widetilde{\bm{x}}^{k_j}\}$ and the convergence consistency of $\phi$ (see condition ({\bf B5}) in Definition \ref{defBregfun}), we have that $\lim_{j\to\infty}\widetilde{\bm{x}}^{k_j}= \bm{x}^{\infty}$. Therefore, from what we have proved in the last paragraph, $\bm{x}^{\infty}$ is an optimal solution of \eqref{orgpro}. Moreover, by using \eqref{desineq-Pistar} with $\bm{x}^*$ replaced by $\bm{x}^{\infty}$, we can conclude that $\{\mathcal{D}_{\phi}(\bm{x}^{\infty},\,\bm{x}^k)\}$ is convergent. On the other hand, it follows from $\lim_{j\to\infty} \bm{x}^{k_j} = \bm{x}^{\infty}$ and condition ({\bf B4}) of the Bregman function that $\mathcal{D}_{\phi}(\bm{x}^{\infty}, \,\bm{x}^{k_j})\to0$. Consequently, $\{\mathcal{D}_{\phi}(\bm{x}^{\infty},\,\bm{x}^k)\}$ must converge to zero. Now, let $\widehat{\bm{x}}^{\infty}$ be any cluster point of $\{\bm{x}^k\}$ with a subsequence $\{\bm{x}^{k'_j}\}$ such that $\bm{x}^{k'_j}\to\widehat{\bm{x}}^{\infty}$. Since $\mathcal{D}_{\phi}(\bm{x}^{\infty},\,\bm{x}^k)\to0$, we have $\mathcal{D}_{\phi}(\bm{x}^{\infty},\,\bm{x}^{k'_j})\to0$. Using the convergence consistency of $\phi$ again, we see that $\bm{x}^{\infty}=\widehat{\bm{x}}^{\infty}$. Since $\widehat{\bm{x}}^{\infty}$ is arbitrary, we can conclude that $\lim_{k\to\infty}\bm{x}^k=\bm{x}^{\infty}$. This, together with the boundedness of $\{\widetilde{\bm{x}}^k\}$, $\mathcal{D}_{\phi}(\widetilde{\bm{x}}^{k}, \,\bm{x}^{k})\to0$ and the convergence consistency of $\phi$, implies that $\{\widetilde{\bm{x}}^k\}$ also converges to $\bm{x}^{\infty}$.
This completes the proof.
\end{proof}

\begin{remark}[\textbf{Comments on the boundedness of $\mathrm{dom}\,f\cap\overline{\mathcal{C}}$}]\label{rek-bd-BPPA}
From the analysis in this section, one can see that the boundedness of $\mathrm{dom}\,f\cap\overline{\mathcal{C}}$ in Assumption \ref{assumps}(ii) is used to guarantee the existence of solutions of problem \eqref{orgpro} and the subproblem \eqref{BPPAsubpro-gen}, as well as the boundedness of $\{\widetilde{\bm{x}}^k\}$, which is a key fact for developing the convergence of the sequence in Theorem \ref{thmconverBPPA}. Here, we would like to comment on some other (possibly weaker) assumptions in place of the boundedness assumption. Indeed, one could just assume that $f+\delta_{\overline{\mathcal{C}}}$ is level-bounded and $\sum|\langle\Delta^{k},\,\widetilde{\bm{x}}^{k+1}-\widetilde{\bm{x}}^{k}\rangle|<\infty$. The former together with Assumption \ref{assumps}(i) will ensure that the original problem and the subproblem have solutions, while the latter, together with $\sup_k\{\gamma_k\}<\infty$, the summability of $\{\mu_k\}$ and $\{\nu_k\}$, \eqref{desineq-Pik} and Lemma \ref{lemseqcon}, can ensure that $\{f(\widetilde{\bm{x}}^k)\}$ is convergent. Then, the convergence of $\{f(\widetilde{\bm{x}}^k)\}$ and the level-boundedness of $f+\delta_{\overline{\mathcal{C}}}$ further imply that
$\{\widetilde{\bm{x}}^k\}$ is bounded. With these facts, one can establish the same results as in Theorems \ref{thmcompBPPA} and \ref{thmconverBPPA}. Note that condition $\sum|\langle\Delta^{k},\,\widetilde{\bm{x}}^{k+1}-\widetilde{\bm{x}}^{k}\rangle|<\infty$ can often be met without much difficulty. One simple case is when $\mathrm{dom}\,f\cap\overline{\mathcal{C}}$ is bounded and $\sum\eta_k<\infty$, as considered in this paper. Moreover, when $\Delta^{k}\equiv0$, as is the case in application to the optimal transport problem (see the next section for more details), $\sum|\langle\Delta^{k},\,\widetilde{\bm{x}}^{k+1}-\widetilde{\bm{x}}^{k}\rangle|<\infty$ holds trivially. In addition, one could check one more condition $|\langle\Delta^{k},\,\widetilde{\bm{x}}^{k+1}-\widetilde{\bm{x}}^{k}\rangle |\leq\widetilde{\eta}_k$ along with condition \eqref{BPPAcond-gen} at each iteration, where $\{\widetilde{\eta}_k\}$ is a given summable nonnegative sequence. This then enforces $\sum|\langle\Delta^{k},\,\widetilde{\bm{x}}^{k+1}-\widetilde{\bm{x}}^{k}\rangle|<\infty$.
\end{remark}

%%%%%%%%%%%%%%%%%%%%%%%%%%%%%%%%%%%%%%%%%%%%%%
\section{Application to the optimal transport problem}\label{sec-ot}

In this section, we present a concrete application to the optimal transport (OT) problem to show the potential advantages of our iBPPA in Algorithm \ref{algBPPA-gen}. The discrete OT problem is a classical optimization problem that has received great attention in recent years. We refer interested readers to a recent comprehensive survey \cite{pc2019computational} (mainly from the computational point of view) and references therein for more details. Mathematically, the discrete OT problem is given as follows:
\begin{equation}\label{otproblem}
\min_{X}~\langle C, \,X\rangle ~~\mathrm{s.t.} ~~ X\in\Omega := \big\{X\in\mathbb{R}^{m \times n} : X\bm{e}_{n} = \bm{a}, ~X^{\top}\bm{e}_{m} = \bm{b}, ~X \geq 0\big\},
\end{equation}
where $C\in\mathbb{R}^{m \times n}_+$ is a given cost matrix, $\bm{a}:=(a_1,\cdots,a_m)^{\top}\in\Sigma_m$ and  $\bm{b}:=(b_1,\cdots,b_n)^{\top}\in\Sigma_n$ are given probability vectors with $\Sigma_{m}$ (resp. $\Sigma_{n}$) denoting the $m$ (resp. $n$)-dimensional unit simplex, and $\bm{e}_{m}$ (resp. $\bm{e}_{n}$) denotes the $m$ (resp. $n$)-dimensional vector of all ones. It is obvious that the OT problem \eqref{otproblem} falls into the form of \eqref{orgpro} via some simple reformulations and thus our iBPPA in Algorithm \ref{algBPPA-gen} is applicable. We will consider the following two cases.

\subsection{iBPPA with the quadratic proximal term}\label{sec-OTquad}

In this case, we equivalently reformulate \eqref{otproblem} as
\begin{equation}\label{proreform1}
\min_{X}~ \delta_{\Omega}(X) + \langle C, \,X\rangle \quad \mathrm{s.t.} \quad X\in\mathbb{R}^{m \times n},
\end{equation}
which obviously takes the form of \eqref{orgpro} with $f(X)=\delta_{\Omega}(X) + \langle C, \,X\rangle$ and $\mathcal{C}=\mathbb{R}^{m \times n}$. Then, we can apply our iBPPA with the quadratic kernel function $\phi(X)=\frac{1}{2}\|X\|^2_F$ to solve \eqref{proreform1}. The subproblem at each iteration takes the following generic form
\begin{equation}\label{proreform1-subprogen}
\min_{X}~\delta_{\Omega}(X) + \langle C, \,X\rangle + \frac{\gamma}{2}\|X-S\|_F^2
\end{equation}
for some given $S\in\mathbb{R}^{m \times n}$ and $\gamma>0$, which is equivalent to
\begin{equation}\label{proreform1-subpro}
\min_{X}~\frac{1}{2}\|X - S + \gamma^{-1}C\|_F^2 \quad \mathrm{s.t.} \quad X\in\Omega.
\end{equation}
Thus, solving the subproblem \eqref{proreform1-subprogen} amounts to computing the projection of $G:=S-\gamma^{-1}C$ over $\Omega$. To the best of our knowledge, the state-of-the-art method for computing such a projection is the semismooth Newton conjugate gradient ({\sc Ssncg}) method  proposed recently by Li, Sun and Toh \cite{lst2020on}. Specifically, they consider the following dual problem of \eqref{proreform1-subpro}:
\begin{equation}\label{proreform1-subpro-dual}
\min\limits_{\bm{y}}~\Psi(\bm{y}):=\frac{1}{2}\|\Pi_+(\mathcal{A}^*(\bm{y})+G)\|_F^2 - \langle\bm{y}, \,\bm{c}\rangle - \frac{1}{2}\|G\|_F^2 \quad \mathrm{s.t.} \quad \bm{y} \in \mathrm{Ran}(\mathcal{A}),
\end{equation}
where $\bm{y}\in\mathbb{R}^{m+n}$ is the dual variable, $\mathcal{A}:\mathbb{R}^{m \times n} \to \mathbb{R}^{m+n}$ is the linear operator defined by $\mathcal{A}(X):=[X\bm{e}_n; X^{\top}\bm{e}_m]$, $\mathcal{A}^*$ is the adjoint operator of $\mathcal{A}$, $\mathrm{Ran}(\mathcal{A})$ is the range space of $\mathcal{A}$, $\Pi_+: \mathbb{R}^{m \times n} \to \mathbb{R}^{m \times n}_+$ is the projection operator over $\mathbb{R}^{m \times n}_+$,  and $\bm{c}:=[\bm{a}; \bm{b}]$. It is easy to verify that if $\bar{\bm{y}}$ is a solution of the nonsmooth equation
\begin{equation*}
\nabla\Psi(\bm{y})=\mathcal{A}\,\Pi_+(\mathcal{A}^*(\bm{y})+G)-\bm{c}=0, \quad \bm{y} \in \mathrm{Ran}(\mathcal{A}),
\end{equation*}
then $\bar{\bm{y}}$ solves \eqref{proreform1-subpro-dual} and $\widebar{X} := \Pi_+(\mathcal{A}^*(\bar{\bm{y}})+G)$ solves \eqref{proreform1-subpro}. In view of this, {\sc Ssncg} is then adapted to solve the above nonsmooth equation. Indeed, started from $\bm{y}^0 \in \mathrm{Ran}(\mathcal{A})$, {\sc Ssncg} ensures that the generated sequence $\{\bm{y}^t\}$ always lies in $\mathrm{Ran}(\mathcal{A})$ and $\|\nabla\Psi(\bm{y}^t)\| \to 0$ (see \cite[Theorem 2]{lst2020on}). Thus, in practice, an approximate solution $X^t:=\Pi_+(\mathcal{A}^*(\bm{y}^t)+G)$ of \eqref{proreform1-subpro} can be returned when $\|\nabla\Psi(\bm{y}^t)\|\leq\varepsilon$ for a given tolerance $\varepsilon>0$. Extensive numerical results have  been reported in \cite{lst2020on} to show the high efficiency of {\sc Ssncg} for computing the projection over $\Omega$.  Hence, it is natural to use {\sc Ssncg} as a subroutine for our iBPPA employing the quadratic kernel function.

A possible feasibility issue, however, may occur when one tries
to verify the stopping condition for solving the subproblem
\eqref{proreform1-subprogen} via {\sc Ssncg}, because an approximate solution $X^t = \Pi_+(\mathcal{A}^*(\bm{y}^t)+G)$ returned by {\sc Ssncg} may not be exactly feasible (indeed, we only have $\|\mathcal{A}(X^t)-\bm{c}\|\leq\varepsilon$). Therefore, an additional projection \textit{or} rounding procedure may be needed to produce a feasible point in $\Omega$ when performing a certain inexact rule. But its computation is in general nontrivial especially for a complicated feasible region $\Omega$. Fortunately, in our iBPPA, we are able to avoid explicitly computing a feasible point and allow an approximately feasible $X^t$ to be the next proximal point based on the observations given in the next two paragraphs.

We first assume that there is a procedure, denoted by $\mathcal{G}_{\Omega}$, such that for any $X\geq0$, after performing $\mathcal{G}_{\Omega}$ on $X$, we can obtain that $\mathcal{G}_{\Omega}(X)\in\Omega$ and $\|X-\mathcal{G}_{\Omega}(X)\|_F\leq c\,\|\mathcal{A}(X)-\bm{c}\|$ for some constant $c>0$. Since $\Omega$ is a polyhedron, such a procedure is indeed achievable. One natural example is the projection operator denoted by $\mathcal{P}_{\Omega}$. By the Hoffman error bound theorem \cite{h1952on}, there must exist a constant $c>0$ such that $\|X- \mathcal{P}_{\Omega}(X)\|_F \leq c\,\|\mathcal{A}(X)-\bm{c}\|$ for any $X\geq0$. Moreover, one can also consider the rounding procedure in \cite[Algorithm 2]{awr2017near} as $\mathcal{G}_{\Omega}$, which can be computationally more efficient than the projection.

Next we discuss how the stopping condition \eqref{BPPAcond-gen} for the subproblem
\eqref{proreform1-subprogen} in our iBPPA can be verified. When an approximate solution $X^t = \Pi_+(\mathcal{A}^*(\bm{y}^t)+G)\geq0$ is returned by {\sc Ssncg}, with the aid of $\mathcal{G}_{\Omega}$, we have that
\begin{equation}\label{hofferrbd-OT}
\|X^t - \mathcal{G}_{\Omega}(X^t)\|_F \leq c\,\|\mathcal{A}(X^t)-\bm{c}\| = c\,\|\nabla\Psi(\bm{y}^t)\|.
\end{equation}
Thus, for any $Y\in\Omega$, we see that
\begin{equation}\label{subdiff-ineq}
\begin{aligned}
&\quad \langle-C-\gamma(X^t-S), \,Y-\mathcal{G}_{\Omega}(X^t)\rangle \\
&= \gamma\,\langle G - X^t, \,Y-\mathcal{G}_{\Omega}(X^t)\rangle
= \gamma\,\langle \mathcal{A}^*(\bm{y}^t) + G - X^t, \,Y-\mathcal{G}_{\Omega}(X^t)\rangle \\
&= \gamma\,\langle \mathcal{A}^*(\bm{y}^t) + G - X^t, \,Y-X^t\rangle + \gamma\,\langle \mathcal{A}^*(\bm{y}^t) + G - X^t, \,X^t-\mathcal{G}_{\Omega}(X^t)\rangle \\
&\leq \gamma\,\langle \mathcal{A}^*(\bm{y}^t) + G - X^t, \,X^t-\mathcal{G}_{\Omega}(X^t)\rangle \\
&\leq \gamma\,\|\min\{\mathcal{A}^*(\bm{y}^t)+G, \,0\}\|_F\|X^t - \mathcal{G}_{\Omega}(X^t)\|_F
\leq \gamma\,c'\,c\,\|\nabla\Psi(\bm{y}^t)\|,
\end{aligned}
\end{equation}
where the first equality follows from $G:=S-\gamma^{-1}C$, the second equality follows from $\langle \mathcal{A}^*(\bm{y}^t),\,Y-\mathcal{G}_{\Omega}(X^t)\rangle=\langle\bm{y}^t,
\,\mathcal{A}(Y)-\mathcal{A}(\mathcal{G}_{\Omega}(X^t))\rangle=0$, and the first inequality follows from $X^t=\Pi_+(\mathcal{A}^*(\bm{y}^t)+G)$ and $Y\geq0$. The last inequality follows from \eqref{hofferrbd-OT} and the fact that $\{\bm{y}^t\}$ is convergent
\cite[Theorem 2]{lst2020on}, and hence $\|\min\{\mathcal{A}^*(\bm{y}^t)+G, \,0\}\|_F$ must be bounded from the above by some constant $c'>0$. Then, for any $\nu\geq0$ such that $\gamma\,c'\,c\,\|\nabla\Psi(\bm{y}^t)\| \leq \nu$, we can obtain from \eqref{subdiff-ineq} that
\begin{equation*}
0\in\partial_{\nu}\delta_{\Omega}(\mathcal{G}_{\Omega}(X^t)) + C + \gamma(X^t - S).
\end{equation*}
In view of this relation and \eqref{hofferrbd-OT}, our inexact condition \eqref{BPPAcond-gen} is checkable at the pair of points $(X^t, \,\mathcal{G}_{\Omega}(X^t))$ and it can be satisfied as long as $\|\nabla\Psi(\bm{y}^t)\|$ is sufficiently small. It is worth noting that, though the procedure $\mathcal{G}_{\Omega}$ is used in above discussion, it turns out that one does not need to explicitly compute $\mathcal{G}_{\Omega}(X^t)$ and a possibly infeasible point $X^t$ is allowed to be the next proximal point within our framework.

In contrast, the classic inexact conditions $0\in\partial_{\nu}\delta_{\Omega}(X)+C+\gamma(X-S)$ (condition \eqref{BPPAcond1-Te}) and $\mathrm{dist}(0,\,\partial\delta_{\Omega}(X)+C+\gamma(X-S))\leq\eta$ (condition \eqref{BPPAcond1-Ec}) have to be checked at a \textit{single} feasible point. Note that, for any $Y\in\Omega$,
\begin{equation}\label{subdiff-ineq-te}
\begin{aligned}
&\quad \langle-C-\gamma(\mathcal{G}_{\Omega}(X^t)-S), \,Y-\mathcal{G}_{\Omega}(X^t)\rangle \\
&= \langle-C-\gamma(X^t-S), \,Y-\mathcal{G}_{\Omega}(X^t)\rangle + \gamma\langle X^t - \mathcal{G}_{\Omega}(X^t), \,Y-\mathcal{G}_{\Omega}(X^t)\rangle \\
&\leq \langle-C\!-\!\gamma(X^t\!-\!S), Y\!-\!\mathcal{G}_{\Omega}(X^t)\rangle
\!+\! \gamma c''\|X^t \!-\! \mathcal{G}_{\Omega}(X^t)\|_F
\leq \gamma c(c'\!+\!c'')\|\nabla\Psi(\bm{y}^t)\|,
\end{aligned}
\end{equation}
where the first inequality follows from $\|Y-\mathcal{G}_{\Omega}(X^t)\|_F \leq c''$ for some constant $c''>0$ (since $\Omega$ is bounded) and the last inequality follows from \eqref{hofferrbd-OT} and \eqref{subdiff-ineq}. Then, for any $\nu\geq0$ such that $\gamma\,c\,(c'+c'')\|\nabla\Psi(\bm{y}^t)\|\leq\nu$, the inequality \eqref{subdiff-ineq-te} implies that
\begin{equation*}
0\in\partial_{\nu}\delta_{\Omega}(\mathcal{G}_{\Omega}(X^t)) + C + \gamma(\mathcal{G}_{\Omega}(X^t) - S),
\end{equation*}
from which we see that condition \eqref{BPPAcond1-Te} is verifiable at $\mathcal{G}_{\Omega}(X^t)$ and can also be satisfied as long as $\|\nabla\Psi(\bm{y}^t)\|$ is sufficiently small. However, within this framework, one has to compute $\mathcal{G}_{\Omega}(X^t)$ explicitly and use it as the next proximal point, which can bring more computational burden.

Next, we consider the hybrid proximal extragradient (HPE) method, which is developed and studied in \cite{ms2010complexity,ms2013accelerated,ss1999hybridap,ss1999hybridpr,ss2000error,ss2001unified} as a constructive variant of the inexact proximal point algorithm (using the quadratic proximal term). In HPE, a relative error criteria is used for the subproblem involved. In our context, for a given $\sigma\in[0,1)$, one needs to find a triple $(Y, \,V, \,\varepsilon)\in\mathbb{R}^{m\times n}\times\mathbb{R}^{m\times n}\times\mathbb{R}_+$
such that
\begin{equation*}
V \in \partial_{\varepsilon}\big(\delta_{\Omega}+\langle C,\,\cdot\rangle\big)(Y),
\quad \|\gamma^{-1}V + Y - S\|_F^2 + 2\gamma^{-1}\varepsilon \leq \sigma^2 \|Y - S\|_F^2.
\end{equation*}
Indeed, recall \eqref{subdiff-ineq}, we have that
\begin{equation*}
V^t:=-\gamma(X^t-S)
\in \partial_{\varepsilon_t}\delta_{\Omega}(\mathcal{G}_{\Omega}(X^t)) + C
= \partial_{\varepsilon_t}\big(\delta_{\Omega}+\langle C,\,\cdot\rangle\big)\big(\mathcal{G}_{\Omega}(X^t)\big)
\end{equation*}
with $\varepsilon_t:=\gamma\,c'c\|\nabla\Psi(\bm{y}^t)\|$. Thus, the above relative error criterion is verifiable at $(\mathcal{G}_{\Omega}(X^t),\,V^t,\,\varepsilon_t)$ and can be satisfied whenever
\begin{equation*}
\|X^t - \mathcal{G}_{\Omega}(X^t)\|_F^2 + 2c'c\|\nabla\Psi(\bm{y}^t)\| \leq \sigma^2 \|\mathcal{G}_{\Omega}(X^t) - S\|_F^2.
\end{equation*}
But this criterion may not be easy to check \textit{directly} since the constants $c$ and $c'$ are generally unknown. We now recall \eqref{hofferrbd-OT} and the fact that $c^2\|\nabla\Psi(\bm{y}^t)\|^2\leq\|\nabla\Psi(\bm{y}^t)\|$ holds for all sufficiently large $t$ (since $\|\nabla\Psi(\bm{y}^t)\|\to0$ along the sequence generated by {\sc Ssncg}). Thus, one could check whether $\|\nabla\Psi(\bm{y}^t)\| \leq \tilde{\sigma}^2\|\mathcal{G}_{\Omega}(X^t) - S\|_F^2$ for some $\tilde{\sigma}\in[0,\sigma]$ in order to guarantee the above condition. Note that, in practical implementations, one could simply choose any $\tilde{\sigma}\in[0,1)$ since $\sigma$ can be any number in $[0,1)$. After obtaining such a triple, an extragradient step is performed to compute the new point $S-\gamma^{-1}V^t$, which is exactly $X^t$ in this case (since $V^t=-\gamma(X^t-S)$). Thus, similar to our framework, HPE also allows the possibly infeasible point $X^t$ to be the next proximal point, but the quantity $\|\mathcal{G}_{\Omega}(X^t) - S\|_F^2$ requires the explicit computation of $\mathcal{G}_{\Omega}(X^t)$ (if not impossible) for the verification of the relative error criteria, which brings more computational burden. Finally, note that here we only focus on the comparisons with the primal application of the HPE (as presented above), which is more straightforward for problem \eqref{proreform1} and is closer to our approach.

%%%%%%%%%%%%%%%%%%%%%%%%%%%%%%%%
\subsection{iBPPA with the entropic proximal term}\label{sec-OTentr}

In this case, we equivalently
reformulate \eqref{otproblem} as
\begin{equation}\label{proreform2}
\min\limits_{X}~ \delta_{\Omega^{\circ}}(X) + \langle C, \,X\rangle \quad \mathrm{s.t.} \quad X\geq0,
\end{equation}
where $\Omega^{\circ} := \big\{X\in\mathbb{R}^{m \times n}: X\bm{e}_{n}=\bm{a}, \,X^{\top}\bm{e}_{m}=\bm{b}\big\}$ is an affine space. This problem takes the form of \eqref{orgpro} with $f(X)=\delta_{\Omega^{\circ}}(X) + \langle C, \,X\rangle$ and $\mathcal{C}=\mathbb{R}^{m \times n}_{++}$. Then, we apply our iBPPA with the entropy kernel function $\phi(X)=\sum_{ij}x_{ij}(\log x_{ij}-1)$ for solving \eqref{proreform2}. The subproblem  involved
at each iteration takes the following generic form
\begin{equation*}
\min_{X}~\delta_{\Omega^{\circ}}(X) + \langle C, \,X\rangle + \gamma\,\mathcal{D}_{\phi}(X, \,S)
\end{equation*}
for some given $S\in\mathbb{R}^{m \times n}$ and $\gamma>0$, which is equivalent to
\begin{equation}\label{proreform2-subpro}
\min\limits_{X}~ \langle M, \,X\rangle + \gamma\,{\textstyle\sum_{ij}}x_{ij}(\log x_{ij}-1), \quad \mathrm{s.t.} \quad X\bm{e}_{n} = \bm{a}, ~X^{\top}\bm{e}_{m} = \bm{b},
\end{equation}
where $M:=C-\gamma\log S$. Note that the constraint $X\geq0$ is implicitly imposed  by $\mathrm{dom}\,\phi=\mathbb{R}^{m \times n}_+$. Moreover, the subproblem \eqref{proreform2-subpro} has the same form as the entropic regularized OT problem and hence can be readily solved by the popular Sinkhorn's algorithm
\cite[Section 4.2]{pc2019computational}. Specifically, let $K:=e^{-M/\gamma}$. Then, given an arbitrary initial positive vector $\bm{v}^0$, the iterative scheme is given by
\begin{equation}\label{sinkalg}
\bm{u}^{t} = \bm{a} ./ K\bm{v}^{t-1}, \quad
\bm{v}^{t} = \bm{b} ./ K^{\top}\bm{u}^{t},
\end{equation}
where `$./$' denotes the entrywise division between two vectors. When a pair $(\bm{u}^t, \,\bm{v}^t)$ is obtained based on a certain stopping criterion, an approximate solution of \eqref{proreform2-subpro} can be recovered by setting $X^{t}:= \mathrm{Diag}(\bm{u}^{t})\,K\,\mathrm{Diag}(\bm{v}^{t})$. Sinkhorn's algorithm in \eqref{sinkalg} only involves matrix-vector multiplications/divisions with $O(m+n)$ memory complexity and hence can be implemented highly efficiently in practice. However, it should be noted that Sinkhorn's algorithm may suffer from severe numerical instabilities (due to loss of accuracy involving overflow/underflow operations) and very slow convergence speed when the proximal parameter $\gamma$ takes a small value. The former issue can partially be alleviated by some stabilization techniques (e.g., the log-sum-exp operation \cite[Section 4.4]{pc2019computational}) at the expense of losing some computational efficiency, while the latter is hard to circumvent. Fortunately, in our iBPPA, we have the freedom not to choose a small $\gamma$ and thus the aforementioned two issues can be avoided. More details on Sinkhorn's algorithm for solving the entropic regularized OT problem can be found in
\cite[Section 4]{pc2019computational}.

We next discuss how to use Sinkhorn's algorithm as a subroutine in our iBPPA employing the entropic proximal term. Note that an approximate solution $X^{t}:=\mathrm{Diag}(\bm{u}^{t})K\mathrm{Diag}(\bm{v}^{t})$ returned by Sinkhorn's algorithm is in general not exactly feasible. Thus, some existing inexact conditions such as \eqref{BPPAcond1-Te} and \eqref{BPPAcond1-Ec} cannot be directly verified at $X^t$. Therefore, a certain projection \textit{or} rounding procedure is needed. Moreover, such a procedure would be more restrictive than that in the case of using the quadratic proximal term because conditions like \eqref{BPPAcond1-Te} and \eqref{BPPAcond1-Ec} can only be checked at a point in $\Omega^{\circ}\cap\mathbb{R}^{m \times n}_{++}$, that is, the \textit{relative interior} of $\Omega$. Therefore, one needs to have a procedure, denoted by $\mathcal{G}_{\Omega^+}$, such that $\mathcal{G}_{\Omega^+}(X^t)\in\mathrm{rel}\,\mathrm{int}\,\Omega$, which is in general more difficult to construct than a procedure, denoted by $\mathcal{G}_{\Omega}$, such that $\mathcal{G}_{\Omega}(X^t)\in\Omega$. Fortunately, our iBPPA only requires the latter procedure $\mathcal{G}_{\Omega}$. Recall that $M=C-\gamma\log S$,  $X^{t}=\mathrm{Diag}(\bm{u}^{t})\,K\,\mathrm{Diag}(\bm{v}^{t})$ and $K=e^{-M/\gamma}$. Then, for any $Y\in\Omega^{\circ}$, we see that
\begin{equation*}
\begin{aligned}
&\quad \langle -C-\gamma\,(\log X^t - \log S), \,Y-\mathcal{G}_{\Omega}(X^t)\rangle
= \langle - M - \gamma \log X^t, \,Y-\mathcal{G}_{\Omega}(X^t)\rangle \\
&= \langle - M - \gamma \log (\mathrm{Diag}(\bm{u}^{t})\,K\,\mathrm{Diag}(\bm{v}^{t})), \,Y-\mathcal{G}_{\Omega}(X^t)\rangle \\
&= \langle -\gamma\,(\log \bm{u}^t)\,\bm{e}_n^{\top} - \gamma\,\bm{e}_m\,(\log \bm{v}^t)^{\top}, \,Y-\mathcal{G}_{\Omega}(X^t)\rangle \\
&= -\gamma\,\langle\,\log \bm{u}^t, \,Y\bm{e}_n-\mathcal{G}_{\Omega}(X^t)\bm{e}_n\,\rangle - \gamma\,\langle\,\log \bm{v}^t, \, Y^{\top}\bm{e}_m - (\mathcal{G}_{\Omega}(X^t))^{\top}\bm{e}_m\,\rangle
= 0,
\end{aligned}
\end{equation*}
where the last equality follows from $Y\bm{e}_n=\bm{a}=\mathcal{G}_{\Omega}(X^t)\bm{e}_n$ and $Y^{\top}\bm{e}_m=\bm{b}=(\mathcal{G}_{\Omega}(X^t))^{\top}\bm{e}_m$. This relation implies that
\begin{equation}\label{OTincl-ent}
0\in\partial\delta_{\Omega^{\circ}}(\mathcal{G}_{\Omega}(X^t)) + C + \gamma\big(\log X^t - \log S\big).
\end{equation}
In this case, the quantity $\Delta^k$ on the left-hand-side of \eqref{BPPAcond-gen} is $0$. Thus, our inexact condition \eqref{BPPAcond-gen} is verifiable at the pair $(X^t, \,\mathcal{G}_{\Omega}(X^t))$ and can be satisfied when $\mathcal{D}_{\phi}(\mathcal{G}_{\Omega}(X^t), \,X^{t})$ is sufficiently small. Moreover, we further have $\|X^t - \mathcal{G}_{\Omega}(X^t)\|_F \leq c\,\big(\|X^t\bm{e}_{n}-\bm{a}\|+\|(X^t)^{\top}\bm{e}_{m}-\bm{b}\|\big)$ for some $c>0$ as in subsection \ref{sec-OTquad}. Thus, when the feasibility violation $\|X^t\bm{e}_{n}-\bm{a}\|+\|(X^t)^{\top}\bm{e}_{m}-\bm{b}\|$ is small,  the quantity $\mathcal{D}_{\phi}(\mathcal{G}_{\Omega}(X^t), \,X^{t})$ is
also likely to be small. Indeed, we can observe from Figure \ref{figCompInex} that both quantities decrease in tandem. Thus, in practice, one may only check the quantity
$\|X^t\bm{e}_{n}-\bm{a}\| +\|(X^t)^{\top}\bm{e}_{m}-\bm{b}\|=\|\bm{u}^t\odot K\bm{v}^t - \bm{a}\|$ without explicitly computing $\mathcal{G}_{\Omega}(X^t)$ to save cost.

Note that $\nabla\phi$ is explicitly invertible in this case and $V^t:=-\gamma\,(\log X^t-\log S)\in\partial\big(\delta_{\Omega^{\circ}}+\langle C,\,\cdot\rangle\big)(\mathcal{G}_{\Omega}(X^t))$ from \eqref{OTincl-ent}. Thus, we see that the relative error condition \eqref{BPPAcond1-SS} is also checkable, and by some simple manipulations, it can be shown to hold at $(X^t,\,\mathcal{G}_{\Omega}(X^t),\,V^t)$ when $\mathcal{D}_{\phi}(\mathcal{G}_{\Omega}(X^t), \,X^{t}) \leq \sigma^2\mathcal{D}_{\phi}(\mathcal{G}_{\Omega}(X^t), \,S)$. Comparing to our framework, the verification of this condition requires one to compute one more quantity $\mathcal{D}_{\phi}(\mathcal{G}_{\Omega}(X^t),\,S)$ and thus incurs extra cost. Moreover, condition \eqref{BPPAcond1-SS} generally requires one to compute an element in $\partial f$ (rather than a larger set $\partial_{\nu} f$ for some $\nu>0$) at an intermediary point and then performs an `extragradient' step to compute a new proximal point. Such a requirement on an element of $\partial f$ at some intermediate point may be expensive to satisfy when $f$ is not simple; see, for example, the class of linear programming problems studied in \cite{clty2020an}.

Finally, we end this section with a few remarks on some potential numerical issues that may be encountered when employing the inexact condition \eqref{BPPAcond1-Ec}. Assume that we have at hand a procedure $\mathcal{G}_{\Omega^+}$ that is able to find a point in the relative interior of $\Omega$. Using similar arguments for deducing \eqref{OTincl-ent}, we can get
\begin{equation*}
\gamma\big(\log\mathcal{G}_{\Omega^+}(X^t)-\log X^t\big)
\in\partial\delta_{\Omega^{\circ}}(\mathcal{G}_{\Omega^+}(X^t)) + C + \gamma\big( \log\mathcal{G}_{\Omega^+}(X^t) - \log S \big).
\end{equation*}
Thus, condition \eqref{BPPAcond1-Ec} is verifiable at $\mathcal{G}_{\Omega^+}(X^t)$ and can be satisfied when the error $\gamma\|\log \mathcal{G}_{\Omega^+}(X^t)-\log X^t\|_F=\gamma\|\log\big(\mathcal{G}_{\Omega^+}(X^t)./X^t\big)\|_F$ is sufficiently small. However, as observed from our experiments, checking the quantity $\gamma\|\log\big(\mathcal{G}_{\Omega^+}(X^t)./X^t\big)\|_F$ is numerically less stable than checking the quantity $\mathcal{D}_{\phi}(\mathcal{G}_{\Omega^+}(X^t), \,X^{t})$ in our framework, as one can observe from Figure \ref{figCompInex}. To better illustrate this issue, we generate some instances of subproblem \eqref{proreform2-subpro} as follows: we set $m=n=1000$ and set $S$ to be a matrix of ones; moreover, we choose $\gamma\in\{0.1,0.01,0.001\}$ and randomly generate $(\bm{a},\bm{b},C)$ by the same way in subsection \ref{subsec-imple}. Then, we apply Sinkhorn's algorithm and terminate it after some iterations. During the iterations, we record the feasibility accuracy of $X^t$ as well as the quantities  $\gamma\|\log\big(\mathcal{G}_{\Omega^+}(X^t)./X^t\big)\|_F$ and $\mathcal{D}_{\phi}(\mathcal{G}_{\Omega^+}(X^t), \,X^{t})$, where the rounding procedure in \cite[Algorithm 2]{awr2017near} is chosen as $\mathcal{G}_{\Omega^+}$. Moreover, to avoid the possible overflow or underflow in computation, we set $X^t:=\max\big\{X^t,\,10^{-16}\big\}$ and $\mathcal{G}_{\Omega^+}(X^t):=\max\big\{\mathcal{G}_{\Omega^+}(X^t),\,10^{-16}\big\}$ when computing the quantities   $\gamma\|\log\big(\mathcal{G}_{\Omega^+}(X^t)./X^t\big)\|_F$ and $\mathcal{D}_{\phi}(\mathcal{G}_{\Omega^+}(X^t), \,X^{t})$. The computational results are presented in Figure \ref{figCompInex}. One can see that $\gamma\|\log\big(\mathcal{G}_{\Omega^+}(X^t)./X^t\big)\|_F$ always stays at a large value and it hardly decreases as $X^t$ gets close to the feasible set, especially when $\gamma$ is small. This is mainly because some entries of $\mathcal{G}_{\Omega^+}(X^t)./X^t$ could be close to zero and that leads to large negative numbers after performing the log operations. Thus, using $\gamma\|\log\big(\mathcal{G}_{\Omega^+}(X^t)./X^t\big)\|_F\leq\eta$ for some $\eta\geq0$ as a stopping criterion (hence condition \eqref{BPPAcond1-Ec}) could be impractical. In contrast, the quantity $\mathcal{D}_{\phi}(\mathcal{G}_{\Omega^+}(X^t), \,X^{t})$ decreases much more rapidly to zero as the iteration proceeds. Therefore, it can provide a reliable stopping criterion. This indeed highlights another advantage of our inexact framework with the entropic kernel function.

\begin{figure}[ht]
\centering
\includegraphics[width=5.2cm]{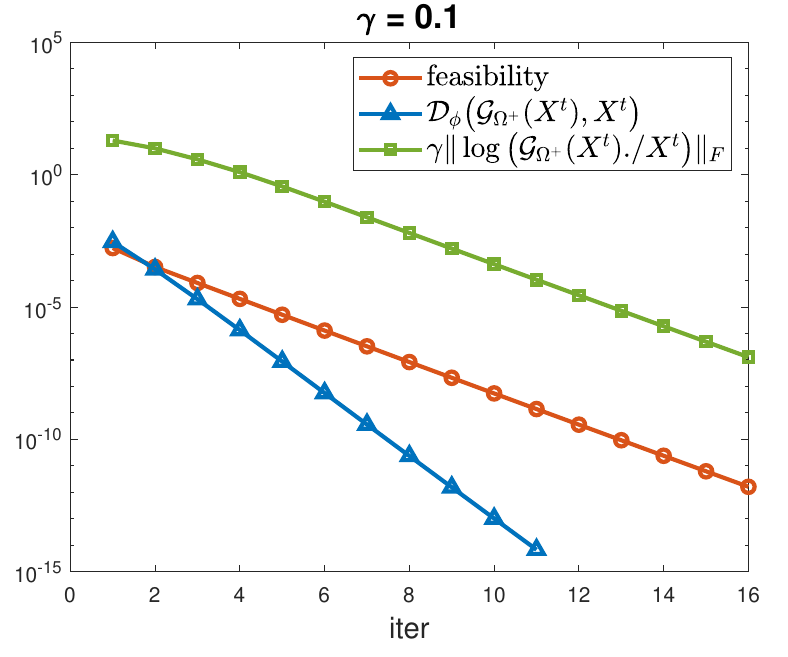}
\includegraphics[width=5.2cm]{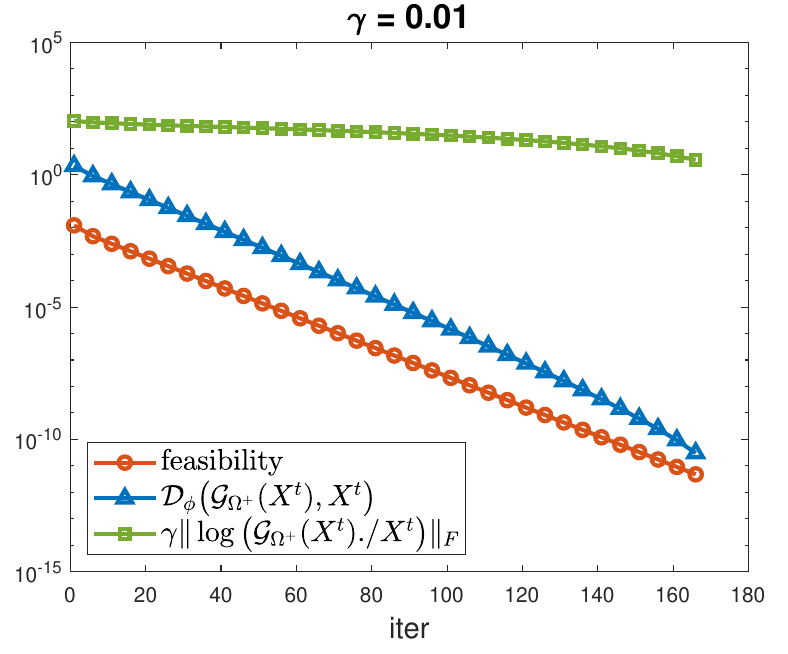}
\includegraphics[width=5.2cm]{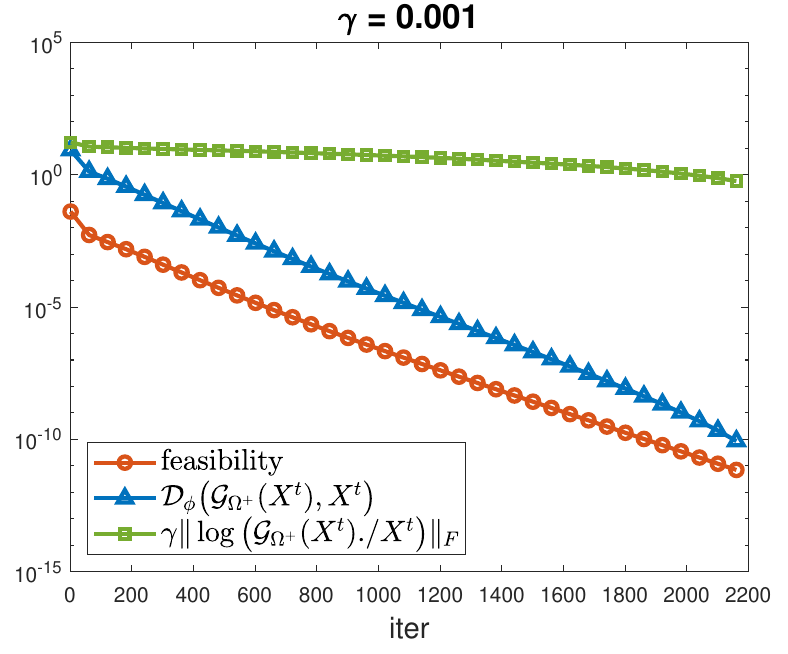}
\caption{Comparisons between $\mathcal{D}_{\phi}(\mathcal{G}_{\Omega^+}(X^t), \,X^{t})$ and $\gamma\|\log\big(\mathcal{G}_{\Omega^+}(X^t)./X^t\big)\|_F$, where ``feasibility" denotes the feasibility accuracy of $X^t$, defined as the value of $\|X^t\bm{e}_{n}-\bm{a}\|+\|(X^t)^{\top}\bm{e}_{m}-\bm{b}\|$.}\label{figCompInex}
\end{figure}

%%%%%%%%%%%%%%%%%%%%%%%%%%%%%%%%%%%%%%%%%%%%%%
\section{An inertial variant of the iBPPA}\label{sec-accBPPA}

In this section, we develop an inertial variant of our iBPPA, denoted by V-iBPPA for short. The inspiration comes from G\"{u}ler's first  classical accelerated proximal point method \cite{g1992new} and its recent Bregman extension \cite{yh2020bregman}. The basic idea used there actually originates from Nesterov's ingenious technique (called \textit{estimate sequence}) in \cite{n1988on} that has motivated many kinds of accelerated methods (see, for example, \cite{at2006interior,llm2011primal,t2008on,t2010approximation}). We also adapt such an idea to develop the V-iBPPA to achieve the possible acceleration. Specifically, our estimate sequence of functions $\{H_k(\bm{x})\}_{k=0}^{\infty}$ are constructed recursively as follows:
\begin{equation}\label{defpofun}
\begin{aligned}
H_0(\bm{x})&:=f(\widetilde{\bm{x}}^{0}) + \pi\, \mathcal{D}_{\phi}(\bm{x},\,\bm{x}^0), \\
H_{k+1}(\bm{x})&:=(1-\theta_k)\,H_k(\bm{x}) \\
&\quad + \theta_k\big(\,f(\widetilde{\bm{x}}^{k+1})+\gamma_k\langle\,\nabla\phi(\bm{y}^k)
-\nabla\phi(\bm{x}^{k+1}), \,\bm{x}-\widetilde{\bm{x}}^{k+1}\,\rangle - \rho\eta_k - \nu_k \,\big),
\end{aligned}
\end{equation}
where $\pi$ and $\gamma_k$ are positive numbers, $\eta_k$ and $\nu_k$ are nonnegative numbers, $\theta_k$ is a number in $[0,\,1)$ (to be specified by \eqref{thetachoice}), $\bm{x}^0=\widetilde{\bm{x}}^{0}\in\mathcal{C}\,(=\mathrm{int}\,\mathrm{dom}\,\phi)$ and $\rho>0$ is the diameter of the feasible set (by Assumption \ref{assumps}(ii)). Resorting to this estimate sequence of functions $\{H_k(\bm{x})\}_{k=0}^{\infty}$, we then present the complete framework of the V-iBPPA in Algorithm \ref{algaccBPPA-gen}.

\begin{algorithm}[h]
\caption{An inertial variant of iBPPA (V-iBPPA) for  \eqref{orgpro}}\label{algaccBPPA-gen}
\textbf{Input:} Let $\{\gamma_k\}_{k=0}^{\infty}$, $\{\nu_k\}_{k=0}^{\infty}$, $\{\eta_k\}_{k=0}^{\infty}$ and $\{\mu_k\}_{k=0}^{\infty}$ be four sequences of nonnegative scalars. Choose $\bm{x}^0=\widetilde{\bm{x}}^{0}=\bm{z}^0\in\mathcal{C}$ arbitrarily and a kernel function $\phi$. Set $k=0$.  \\
\textbf{while} a termination criterion is not met, \textbf{do} %\vspace{-2mm}
\begin{itemize}[leftmargin=2cm]
\item[\textbf{Step 1}.] Choose $\theta_k\in[0,\,1)$ satisfying \eqref{thetachoice} and set $\bm{y}^k=\theta_k\bm{z}^k+(1-\theta_k)\bm{x}^k$.

\item[\textbf{Step 2}.] Find a pair $(\bm{x}^{k+1},\,\widetilde{\bm{x}}^{k+1})$ by approximately solving the following problem
    \begin{equation}\label{accBPPAsubpro-gen}
    \min\limits_{\bm{x}}~
    f(\bm{x}) + \gamma_k\mathcal{D}_{\phi}(\bm{x}, \,\bm{y}^k),
    \end{equation}
    such that $\bm{x}^{k+1}\in\mathcal{C}$, $\widetilde{\bm{x}}^{k+1}\in\mathrm{dom}f\cap\overline{\mathcal{C}}$ and
    \begin{equation}\label{accBPPAcond-gen}
    \begin{aligned}
    &\Delta^{k} \in \partial_{\nu_k} f(\widetilde{\bm{x}}^{k+1}) + \gamma_k\big(\nabla\phi(\bm{x}^{k+1}) - \nabla\phi(\bm{y}^{k})\big) \\
    &~~\mathrm{with}~~\|\Delta^{k}\| \leq \eta_k, ~~\mathcal{D}_{\phi}(\widetilde{\bm{x}}^{k+1}, \,\bm{x}^{k+1}) \leq \mu_k.
    \end{aligned}
    \end{equation}

\item [\textbf{Step 3}.] Set $H_{k+1}(\bm{x})$ by \eqref{defpofun} and compute $\bm{z}^{k+1}=\arg\min\limits_{\bm{x}}\left\{H_{k+1}(\bm{x})\right\}$.

\item [\textbf{Step 4}.] Set $k = k+1$ and go to \textbf{Step 1}.
\end{itemize}
\textbf{end while}  \\
\textbf{Output}: $(\bm{x}^{k},\,\widetilde{\bm{x}}^{k})$ \vspace{0.5mm}
\end{algorithm}

Comparing to the iBPPA in Algorithm \ref{algBPPA-gen}, the V-iBPPA in Algorithm \ref{algaccBPPA-gen} uses an intermediary point $\bm{y}^k$ as the proximal point. When $\theta_k\equiv0$, we have $\bm{y}^k\equiv\bm{x}^k$ and the V-iBPPA readily reduces to the iBPPA, while with the special choice of $\theta_k$ in \eqref{thetachoice}, we shall see later that the V-iBPPA enjoys a flexible convergence rate depending on the property of the kernel function and is able to achieve a faster rate in some scenarios. From arguments similar to those following Algorithm \ref{algBPPA-gen}, the subproblem \eqref{accBPPAsubpro-gen} and the inexact condition \eqref{accBPPAcond-gen} are also well-defined under Assumption \ref{assumps}, provided $\bm{y}^k\in\mathcal{C}\,(=\mathrm{int}\,\mathrm{dom}\,\phi)$. Note from the construction of $H_k(\bm{x})$ in \eqref{defpofun} that
\begin{equation}\label{decomHk}
H_k(\bm{x}) = L_k(\bm{x}) + \pi c_k\,\mathcal{D}_{\phi}(\bm{x},\,\bm{x}^0),
\end{equation}
where $L_k(\cdot)$ is an affine function and $c_k$ is a positive scalar depending on $k$. Since $\mathcal{D}_{\phi}(\cdot,\,\bm{x}^0)$ is level-bounded (by condition (\textbf{B3}) in Definition \ref{defBregfun}), then $H_k(\bm{x})$ is level-bounded. Hence, an optimal solution $\bm{z}^k$ of problem $\min_{\bm{x}}\{H_k(\bm{x})\}$ exists \cite[Theorem 1.9]{rw1998variational} and must also be unique since $\phi$ is strictly convex (by condition ({\bf B1}) in Definition \ref{defBregfun}). The essential smoothness of $\phi$ (by Assumption \ref{assumps}(iii)) further imposes that $\bm{z}^k\in\mathcal{C}$. This together with $\bm{x}^k\in\mathcal{C}$ ensures that the intermediary point $\bm{y}^k$, as a convex combination of $\bm{x}^k$ and $\bm{z}^k$, always lies in $\mathcal{C}$. Therefore, Algorithm \ref{algaccBPPA-gen} is well-defined. Here, we would also like to point out that, when using $\phi(\cdot)=\frac{1}{2}\|\cdot\|^2$ (hence $\mathrm{dom}\,\phi=\mathrm{int}\,\mathrm{dom}\,\phi=\mathbb{E}$), one can have more freedom to choose other updating formulas for $\bm{y}^k$, and they give rise to different variants of the accelerated PPA such as  G\"{u}ler's second accelerated proximal point method \cite[Section 6]{g1992new} and the so-called catalyst acceleration method proposed recently in \cite{lmh2017catalyst}. Our inexact criterion can also be incorporated into those variants. We will leave this topic for future investigation.

In the following, we shall study the convergence property of our V-iBPPA in Algorithm \ref{algaccBPPA-gen}. Since we now use the intermediary point $\bm{y}^k$ as the proximal point in the subproblem \eqref{accBPPAsubpro-gen}, the analysis for Algorithm \ref{algaccBPPA-gen} turns out to be different from that for Algorithm \ref{algBPPA-gen}. In particular, all convergence results presented later are in terms of the objective function value, as is the case in most existing works on various accelerated methods. Our analysis is motivated by several existing works (e.g., \cite{at2006interior,g1992new,hrx2018accelerated,yh2020bregman}) that are based on the Nesterov's estimate sequence. Before proceeding, we introduce the following quadrangle scaling property for the Bregman distance.

\begin{definition}[\textbf{Quadrangle scaling property}]\label{defQSP}
Let $\phi$ be a proper closed convex function which is differentiable on $\mathrm{int}\,\mathrm{dom}\,\phi$. We say $\phi$ has the quadrangle scaling property (QSP) if there exist an exponent $\lambda\geq1$ and two constants $\tau_1, \,\tau_2>0$ such that, for any $\bm{a},\,\bm{c}\in\mathrm{dom}\,\phi$ and $\bm{b},\,\bm{d}\in\mathrm{int}\,\mathrm{dom}\,\phi$, the following inequality holds for any $\theta\in[0,\,1]$,
\begin{equation}\label{ineq-QSP}
\mathcal{D}_{\phi}(\,\theta\bm{a}+(1-\theta)\bm{c}, \,\theta\bm{b}+(1-\theta)\bm{d}\,) \leq \tau_1\,\theta^{\lambda}\,\mathcal{D}_{\phi}(\bm{a},\,\bm{b}) + \tau_2\,(1-\theta)^{\lambda}\,\mathcal{D}_{\phi}(\bm{c},\,\bm{d}).
\end{equation}
Here, $\lambda$ is called the quadrangle scaling exponent (QSE) of $\phi$, and $\tau_1$, $\tau_2$ are called the quadrangle scaling constants (QSCs) of $\phi$.
\end{definition}

Note that when $\bm{c}=\bm{d}$, the QSP reduces to a so-called intrinsic triangle scaling property (TSP) introduced recently in \cite[Section 2]{hrx2018accelerated} for developing accelerated Bregman proximal gradient methods. Thus, our QSP is an extension of the TSP. Two representative examples for the QSP are given as follows. \vspace{1mm}
\begin{itemize}[leftmargin=0.6cm]
\item If $\phi$ is $\mu_{\phi}$-strongly convex and $\nabla \phi$ is $L_{\phi}$-Lipschitz, i.e., $\frac{\mu_{\phi}}{2}\|\bm{x}-\bm{y}\|^2 \leq \mathcal{D}_{\phi}(\bm{x},\,\bm{y}) \leq \frac{L_{\phi}}{2}\|\bm{x}-\bm{y}\|^2$, then for any $\bm{a},\,\bm{c}\in\mathrm{dom}\,\phi$, $\bm{b},\,\bm{d}\in\mathrm{int}\,\mathrm{dom}\,\phi$ and $\theta\in[0,\,1]$,
    \begin{equation*}
    \begin{aligned}
    &\quad \mathcal{D}_{\phi}(\,\theta\bm{a}+(1-\theta)\bm{c}, \,\theta\bm{b}+(1-\theta)\bm{d}\,) \leq {\textstyle\frac{L_{\phi}}{2}}\|\theta(\bm{a}-\bm{b}) + (1-\theta)(\bm{c}-\bm{d})\|^2 \\
    &\leq L_{\phi}\theta^2\|\bm{a}\!-\!\bm{b}\|^2 + L_{\phi}(1\!-\!\theta)^2\|\bm{c}\!-\!\bm{d}\|^2
    \!=\! {\textstyle\frac{2L_{\phi}}{\mu_{\phi}}\theta^2\frac{\mu_{\phi}}{2}}
    \|\bm{a}\!-\!\bm{b}\|^2 + {\textstyle\frac{2L_{\phi}}{\mu_{\phi}}(1\!-\!\theta)^2\frac{\mu_{\phi}}{2}}
    \|\bm{c}\!-\!\bm{d}\|^2 \\
    &\leq {\textstyle\frac{2L_{\phi}}{\mu_{\phi}}\theta^2} \mathcal{D}_{\phi}(\bm{a},\,\bm{b}) + {\textstyle\frac{2L_{\phi}}{\mu_{\phi}}(1-\theta)^2} \mathcal{D}_{\phi}(\bm{c},\,\bm{d}).
    \end{aligned}
    \end{equation*}
    Thus, in this case, $\phi$ has the QSP with $\lambda=2$ and $\tau_1=\tau_2=2L_{\phi}/\mu_{\phi}$.

\vspace{1mm}
\item If $\mathcal{D}_{\phi}(\cdot, \,\cdot)$ is jointly convex, which can be satisfied by the entropy kernel function $\phi(\bm{x})=\sum_{i}x_{i}(\log x_{i}-1)$ (see \cite{bb2001joint} for more examples), then for any $\theta\in[0,\,1]$,
    \begin{equation*}
    \mathcal{D}_{\phi}(\theta\bm{a}+(1-\theta)\bm{c}, \,\theta\bm{b}+(1-\theta)\bm{d})
    \leq \theta\,\mathcal{D}_{\phi}(\bm{a},\bm{b}) + (1-\theta)\,\mathcal{D}_{\phi}(\bm{c},\bm{d})
    \end{equation*}
    Thus, in this case, $\phi$ has the QSP with $\lambda=\tau_1=\tau_2=1$.
\end{itemize}

\vspace{2mm}
We now start the analysis with a lemma concerning the difference $H_{k}(\bm{x})-f(\bm{x})$.

\begin{lemma}\label{lem-ineq-Hk1}
Let the estimate sequence of functions $\{H_k(\bm{x})\}_{k=0}^{\infty}$ be generated by \eqref{defpofun}. Then, for all $k\geq0$, we have
\begin{equation*}
H_{k+1}(\bm{x}) - f(\bm{x}) \leq (1-\theta_k)(H_{k}(\bm{x}) - f(\bm{x})), \quad \forall\,\bm{x}\in\mathrm{dom}\,f\cap\overline{\mathcal{C}}.
\end{equation*}
\end{lemma}
\begin{proof}
From condition \eqref{accBPPAcond-gen}, there exists a $\bm{d}^{k+1}\in\partial_{\nu_k} f(\widetilde{\bm{x}}^{k+1})$ such that
$\Delta^{k} = \bm{d}^{k+1} + \gamma_k\big(\nabla\phi(\bm{x}^{k+1}) - \nabla\phi(\bm{y}^{k})\big)$. For notational simplicity, let
\begin{equation}\label{defXi}
\Xi^k(\bm{x}):=\langle \nabla\phi(\bm{y}^{k})-\nabla\phi(\bm{x}^{k+1}), \,\bm{x}-\widetilde{\bm{x}}^{k+1} \rangle.
\end{equation}
Then, for any $\bm{x}\in\mathrm{dom}f\cap\overline{\mathcal{C}}$, we see that
\begin{equation}\label{ineq1-add}
\begin{aligned}
f(\bm{x})
&\geq f(\widetilde{\bm{x}}^{k+1}) + \langle \bm{d}^{k+1}, \,\bm{x} - \widetilde{\bm{x}}^{k+1} \rangle - \nu_k \\
&= f(\widetilde{\bm{x}}^{k+1}) + \langle \Delta^{k} - \gamma_k\big(\nabla\phi(\bm{x}^{k+1}) - \nabla\phi(\bm{y}^{k})\big), \,\bm{x}-\widetilde{\bm{x}}^{k+1} \rangle - \nu_k \\
&\geq f(\widetilde{\bm{x}}^{k+1}) + \gamma_k\,\Xi^k(\bm{x})
+ \langle \Delta^{k}, \bm{x} -\widetilde{\bm{x}}^{k+1} \rangle - \nu_k \\
&\geq f(\widetilde{\bm{x}}^{k+1}) + \gamma_k\,\Xi^k(\bm{x}) - \rho\eta_k - \nu_k,
\end{aligned}
\end{equation}
where the last inequality follows from $\langle \Delta^{k}, \,\bm{x} -\widetilde{\bm{x}}^{k+1} \rangle\geq-\|\bm{x} -\widetilde{\bm{x}}^{k+1}\|\|\Delta^{k}\|\geq-\rho\eta_k$ due to $\bm{x},\,\widetilde{\bm{x}}^{k+1}\in\mathrm{dom}f\cap\overline{\mathcal{C}}$ and Assumption \ref{assumps}. Using \eqref{ineq1-add} and the construction of $H_k(\bm{x})$ in \eqref{defpofun}, we see that
\begin{equation*}
\hspace{-3mm}
\begin{aligned}
H_{k+1}(\bm{x}) - f(\bm{x})
&= (1-\theta_k)\,H_k(\bm{x})
+ \theta_k\left(f(\widetilde{\bm{x}}^{k+1}) + \gamma_k\,\Xi^k(\bm{x}) - \rho\eta_k - \nu_k \right) - f(\bm{x}) \\
&=(1-\theta_k)(H_k(\bm{x})-f(\bm{x}))
+\theta_k\!\left( f(\widetilde{\bm{x}}^{k+1})+\gamma_k\,\Xi^k(\bm{x}) \!-\! \rho\eta_k \!-\! \nu_k \!-\! f(\bm{x}) \right) \\
&\leq (1-\theta_k)(H_k(\bm{x})-f(\bm{x})).
\end{aligned}
\end{equation*}
This completes the proof.
\end{proof}

One can easily see from Lemma \ref{lem-ineq-Hk1} that, at $k$-th iteration, the difference $H_{k}(\bm{x})-f(\bm{x})$ is reduced by a factor $1-\theta_k$. Then, by induction, we further obtain that
\begin{equation}\label{ineq-Hk2}
H_{k}(\bm{x}) - f(\bm{x}) \leq c_k (H_{0}(\bm{x}) - f(\bm{x})), \quad \forall\,\bm{x}\in\mathrm{dom}\,f\cap\overline{\mathcal{C}},
\end{equation}
where
$$
c_0:=1,~~c_k:={\textstyle\prod^{k-1}_{i=0}}(1-\theta_i)~~\mbox{for}~~k\geq1.
$$
To further evaluate the reduction in the original objective (that is, $f(\widetilde{\bm{x}}^k)-f(\bm{x})$) based on \eqref{ineq-Hk2}, we only need to explore the relation between $f(\widetilde{\bm{x}}^k)$ and $H_k(\bm{z}^k)$, where $\bm{z}^{k}=\arg\min_{\bm{x}}\left\{H_{k}(\bm{x})\right\}$ by {\bf Step 3} in Algorithm \ref{algaccBPPA-gen}. Indeed, we have the following result.

\begin{lemma}\label{lem-ineq-Hk2}
Let $\{\bm{x}^{k}\}$ and $\{\widetilde{\bm{x}}^{k}\}$ be the sequences generated by the V-iBPPA in Algorithm \ref{algaccBPPA-gen}.
Suppose that Assumption \ref{assumps} holds, $\phi$ has the QSP with an exponent $\lambda\geq1$ and QSCs $\tau_1, \,\tau_2>0$, and $\theta_k$ is chosen such that
\begin{equation}\label{thetachoice}
\tau_{1}\,\gamma_k\,\theta_k^{\lambda} = \pi c_k\,(1-\theta_k).
\end{equation}
If $f(\widetilde{\bm{x}}^k) \leq H_k(\bm{z}^k) + \delta_k$ for some $k\geq0$ and $\delta_k\geq0$, then
\begin{equation*}
f(\widetilde{\bm{x}}^{k+1})
\leq H_{k+1}(\bm{z}^{k+1}) + (1-\theta_k)\delta_k + \gamma_k(\mu_k+\tau_2\,\mu_{k-1}) + \rho\eta_k + \nu_k.
\end{equation*}
\end{lemma}
\begin{proof}
First, from \eqref{decomHk}, Lemma \ref{lem-supp} and the definition of $\bm{z}^k$ as a minimizer of $H_k(\cdot)$ (by {\bf Step 3} in Algorithm \ref{algaccBPPA-gen}), we see that
\begin{equation*}
H_k(\bm{z}^{k+1}) = H_k(\bm{z}^k) + \pi c_k\,\mathcal{D}_{\phi}(\bm{z}^{k+1},\,\bm{z}^k),
\end{equation*}
which, together with the hypothesis of this lemma, implies that
\begin{equation}\label{eq-Hk1}
H_k(\bm{z}^{k+1}) \geq f(\widetilde{\bm{x}}^k) + \pi c_k\,\mathcal{D}_{\phi}(\bm{z}^{k+1},\,\bm{z}^k) - \delta_k.
\end{equation}
Moreover, recall the definition of $\Xi^k(\cdot)$ in \eqref{defXi}, one can verify that
\begin{equation}\label{ineq-Hk3}
\begin{aligned}
& (1-\theta_k)\,f(\widetilde{\bm{x}}^{k}) + \theta_k\left(f(\widetilde{\bm{x}}^{k+1}) +
\gamma_k\,\Xi^k(\bm{z}^{k+1}) - \rho\eta_k - \nu_k \right) \\
&\geq (1-\theta_k)\!\left(f(\widetilde{\bm{x}}^{k+1})
+ \gamma_k\Xi^k(\widetilde{\bm{x}}^{k}) - \rho\eta_k \!-\! \nu_k\right)
+ \theta_k\!\left(f(\widetilde{\bm{x}}^{k+1})
+ \gamma_k\Xi^k(\bm{z}^{k+1}) - \rho\eta_k \!-\! \nu_k \right) \\
&= f(\widetilde{\bm{x}}^{k+1}) + \gamma_k\,\Xi^k\big(\theta_k\bm{z}^{k+1} + (1-\theta_k)\widetilde{\bm{x}}^{k}\big)
- \rho\eta_k - \nu_k \\
&= f(\widetilde{\bm{x}}^{k+1})
+ \gamma_k\mathcal{D}_{\phi}\big(\theta_k\bm{z}^{k+1} + (1-\theta_k)\widetilde{\bm{x}}^{k}, \,\bm{x}^{k+1}\big)
+ \gamma_k\mathcal{D}_{\phi}(\widetilde{\bm{x}}^{k+1}, \,\bm{y}^k) \\
&\qquad - \gamma_k\mathcal{D}_{\phi}\big(\theta_k\bm{z}^{k+1} + (1-\theta_k)\widetilde{\bm{x}}^{k}, \,\bm{y}^{k}\big)
- \gamma_k\mathcal{D}_{\phi}(\widetilde{\bm{x}}^{k+1}, \,\bm{x}^{k+1})
- \rho\eta_k - \nu_k \\
&\geq f(\widetilde{\bm{x}}^{k+1})
- \gamma_k\mathcal{D}_{\phi}\big(\theta_k\bm{z}^{k+1} + (1-\theta_k)\widetilde{\bm{x}}^{k},\,\bm{y}^{k}\big)
- \gamma_k\mathcal{D}_{\phi}(\widetilde{\bm{x}}^{k+1}, \,\bm{x}^{k+1})
- \rho\eta_k - \nu_k \\
&\geq f(\widetilde{\bm{x}}^{k+1}) - \gamma_k\mathcal{D}_{\phi}\big(\theta_k\bm{z}^{k+1} + (1-\theta_k)\widetilde{\bm{x}}^{k},\,\theta_k\bm{z}^k+(1-\theta_k)\bm{x}^k\big) - \gamma_k\mu_k - \rho\eta_k - \nu_k \\
&\geq f(\widetilde{\bm{x}}^{k+1}) -\gamma_k\,\tau_1\,\theta_k^{\lambda}\,\mathcal{D}_{\phi}(\bm{z}^{k+1},\,\bm{z}^k) - \gamma_k\,\tau_2\,(1-\theta_k)^{\lambda}\,
\mathcal{D}_{\phi}(\widetilde{\bm{x}}^{k},\,\bm{x}^k)
- \gamma_k\mu_k - \rho\eta_k - \nu_k \\
&\geq f(\widetilde{\bm{x}}^{k+1}) - \tau_1\,\gamma_k\,\theta_k^{\lambda}\,\mathcal{D}_{\phi}(\bm{z}^{k+1},\,\bm{z}^k) - \gamma_k(\mu_k+\tau_2\,\mu_{k-1}) - \rho\eta_k - \nu_k,
\end{aligned}
\end{equation}
where the first inequality follows from \eqref{ineq1-add} with $\bm{x}=\widetilde{\bm{x}}^{k}$, the second equality follows from the four points identity \eqref{fourId}, the third inequality follows from $\bm{y}^k=\theta_k\bm{z}^k+(1-\theta_k)\bm{x}^k$ (by {\bf Step 1} in Algorithm \ref{algaccBPPA-gen}) and $\mathcal{D}_{\phi}(\widetilde{\bm{x}}^{k+1}, \,\bm{x}^{k+1}) \leq \mu_k$ (by condition \eqref{accBPPAcond-gen}), the second last inequality follows from the QSP of $\mathcal{D}_{\phi}$ and the last inequality follows from $1-\theta_k\leq1$ and $\mathcal{D}_{\phi}(\widetilde{\bm{x}}^{k}, \,\bm{x}^{k}) \leq \mu_{k-1}$. Then, we see that
\begin{eqnarray*}
\begin{aligned}
& H_{k+1}(\bm{z}^{k+1})
= (1-\theta_k)\,H_k(\bm{z}^{k+1})
+ \theta_k\left(f(\widetilde{\bm{x}}^{k+1})+\gamma_k\,\Xi^k(\bm{z}^{k+1}) - \rho\eta_k - \nu_k \right) \\
&\geq (1-\theta_k)f(\widetilde{\bm{x}}^{k}) + \theta_k\left(f(\widetilde{\bm{x}}^{k+1}) + \gamma_k\,\Xi^k(\bm{z}^{k+1}) - \rho\eta_k - \nu_k \right)
+ \pi c_k(1-\theta_k)\,\mathcal{D}_{\phi}(\bm{z}^{k+1},\,\bm{z}^k) - (1-\theta_k)\delta_k \\
&\geq f(\widetilde{\bm{x}}^{k+1}) + \big[\pi c_k(1-\theta_k)-\tau_1\,\gamma_k\,\theta_k^{\lambda}\big]
\,\mathcal{D}_{\phi}(\bm{z}^{k+1},\,\bm{z}^k) - (1-\theta_k)\delta_k
- \gamma_k(\mu_k+\tau_2\,\mu_{k-1}) - \rho\eta_k - \nu_k \\
&\geq f(\widetilde{\bm{x}}^{k+1}) - (1-\theta_k)\delta_k - \gamma_k(\mu_k+\tau_2\,\mu_{k-1}) - \rho\eta_k - \nu_k,
\end{aligned}
\end{eqnarray*}
where the first equality follows from the construction of $H_k(\bm{x})$ in \eqref{defpofun}, the first inequality follows from \eqref{eq-Hk1}, the second inequality follows from \eqref{ineq-Hk3} and the last inequality follows from the choice of $\theta_k$ in \eqref{thetachoice}. This completes the proof.
\end{proof}

Then, we have the theorem concerning the reduction of the objective value.

\begin{theorem}\label{thm-comp-fk1}
Suppose that Assumption \ref{assumps} holds, $\phi$ has the QSP and $\theta_k$ satisfies \eqref{thetachoice}. Let $\{\bm{x}^{k}\}$ and $\{\widetilde{\bm{x}}^{k}\}$ be the sequences generated by the V-iBPPA in Algorithm \ref{algaccBPPA-gen}. Then, for any optimal solution $\bm{x}^*$ of problem \eqref{orgpro}, we have
\begin{equation}\label{comp-fk1}
f(\widetilde{\bm{x}}^{N}) - f(\bm{x}^*) \leq c_N\big( f(\widetilde{\bm{x}}^{0}) - f(\bm{x}^*) + \pi\,\mathcal{D}_{\phi}(\bm{x}^*,\,\bm{x}^0) \big) + \delta_N,
\end{equation}
where the error sequence $\{\delta_{k}\}_{k=0}^{\infty}$ satisfies
\begin{equation}\label{deltacond}
\delta_0 = 0, ~~
\delta_{k+1} = (1-\theta_{k})\delta_{k} + \gamma_{k}(\mu_{k} + \tau_2\,\mu_{k-1}) + \rho\eta_{k} + \nu_{k}, ~~ k = 0, \,1, \,\ldots.
\end{equation}
\end{theorem}
\begin{proof}
First, from Lemma \ref{lem-ineq-Hk2}, it is easy to prove by induction that $f(\widetilde{\bm{x}}^{N}) \leq H_{N}(\bm{z}^{N}) + \delta_N$ for any $N\geq0$. Moreover, note from \eqref{ineq-Hk2} that $H_{N}(\bm{x}^*) - f(\bm{x}^*) \leq c_N (H_{0}(\bm{x}^*) - f(\bm{x}^*))$ for any $N\geq0$. These relations together with the fact that $\bm{z}^{N}$ is the minimizer of the problem $\min_{\bm{x}}\{H_N(\bm{x})\}$ prove the desired result.
\end{proof}

From the choice of $\theta_k$ in \eqref{thetachoice}, we see that $0<\theta_k<1$ and hence $c_N \to 0$. This together with \eqref{comp-fk1} shows that $f(\widetilde{\bm{x}}^{N})$ converges to $f^*:=\min\{f(\bm{x}):\bm{x}\in\overline{\mathcal{C}}\}$ as long as $\delta_N\to0$. Here, $c_N$ and $\delta_N$ naturally determine the convergence rate and thus we must estimate their magnitudes. The following estimate on $c_N$ extends \cite[Lemma 2.2]{g1992new} to a more general setting.

\begin{lemma}\label{lem-cNbd}
For any $N\geq1$, we have
\begin{equation}\label{cNbd}
{\left( 1 + (\pi/\tau_1)^{\frac{1}{\lambda}}\,{\textstyle\sum^{N-1}_{k=0}} \gamma_{k}^{-\frac{1}{\lambda}} \right)^{-\lambda}} \leq c_N \leq
{\left( 1 + \lambda^{-1}(\pi/\tau_1)^{\frac{1}{\lambda}}
\,{\textstyle\sum^{N-1}_{k=0}}\gamma_{k}^{-\frac{1}{\lambda}} \right)^{-\lambda}}.
\end{equation}
Moreover, if $\sup_k\{\gamma_k\}<\infty$, then
$c_N = O\Big(
{\Big(\sum^{N-1}_{k=0} \gamma_{k}^{-\frac{1}{\lambda}}\Big)^{-\lambda}}\Big)$.
\end{lemma}
\begin{proof}
%Since the proof follows similar ideas as in \cite[Lemma 2.2]{g1992new}, we omit it here to save some space.
First, it is easy to see that $c_{k+1} = (1-\theta_{k})c_{k}$ and then $\theta_{k}=1-c_{k+1}/c_{k}$ for all $k\geq0$. Substituting this in \eqref{thetachoice} results in
\begin{equation}\label{eq-ck}
\tau_1\,\gamma_{k}\left(1-c_{k+1}/c_{k}\right)^{\lambda} = \pi c_{k+1} \quad \Longleftrightarrow \quad c_{k+1}^{-1} - c_{k}^{-1} = (\pi/\tau_1)^{\frac{1}{\lambda}}\,\gamma_{k}^{-\frac{1}{\lambda}}\,c_{k+1}^{\frac{1}{\lambda}-1}.
\end{equation}
Note that $c_{k+1} \leq c_{k}$ (since $\theta_{k}\in(0,\,1)$) and $\lambda\geq1$ (by definition of QSE). Hence,
\begin{equation*}
c_{k+1}^{\frac{1}{\lambda}-1}\left(c_{k+1}^{-\frac{1}{\lambda}} - c_{k}^{-\frac{1}{\lambda}}\right)
= c_{k+1}^{-1} - c_{k+1}^{\frac{1}{\lambda}-1}c_{k}^{-\frac{1}{\lambda}}
\leq c_{k+1}^{-1} - c_{k}^{-1}.
\end{equation*}
Combing this and \eqref{eq-ck}, we see that
$c_{k+1}^{-\frac{1}{\lambda}} - c_{k}^{-\frac{1}{\lambda}}
\leq (\pi/\tau_1)^{\frac{1}{\lambda}}\,\gamma_{k}^{-\frac{1}{\lambda}}.
$
Summing this inequality from $k=0$ to $k=N-1$, we obtain that
\begin{equation*}
c_N^{-\frac{1}{\lambda}} \leq 1 + (\pi/\tau_1)^{\frac{1}{\lambda}}\,{\textstyle\sum^{N-1}_{k=0}} \gamma_{k}^{-\frac{1}{\lambda}},
\end{equation*}
which gives the lower bound on $c_N$. On the other hand, it is easy to show by Young's inequality that $c_{k+1}^{\frac{1}{\lambda}-1}c_{k}^{-\frac{1}{\lambda}}\leq(1-\lambda^{-1})c_{k+1}^{-1}
+\lambda^{-1}c_k^{-1}$ and thus
\begin{equation*}
c_{k+1}^{-1} - c_{k}^{-1} \leq \lambda\, c_{k+1}^{\frac{1}{\lambda}-1}\left(c_{k+1}^{-\frac{1}{\lambda}} - c_{k}^{-\frac{1}{\lambda}}\right).
\end{equation*}
Combing this and \eqref{eq-ck}, we see that
$c_{k+1}^{-\frac{1}{\lambda}} - c_{k}^{-\frac{1}{\lambda}} \geq \lambda^{-1}(\pi/\tau_1)^{\frac{1}{\lambda}}\,\gamma_{k}^{-\frac{1}{\lambda}}.
$
Summing this inequality from $k=0$ to $k=N-1$, we obtain that
\begin{equation*}
c_N^{-\frac{1}{\lambda}} \geq 1 + \lambda^{-1}(\pi/\tau_1)^{\frac{1}{\lambda}}\,{\textstyle\sum^{N-1}_{k=0}} \gamma_{k}^{-\frac{1}{\lambda}},
\end{equation*}
which gives the upper bound on $c_N$. The other result follows immediately from \eqref{cNbd}.
\end{proof}

We immediately have the following proposition.

\begin{proposition}\label{coro-con-accBPPA}
Suppose that all conditions in Theorem \ref{thm-comp-fk1} and Lemma \ref{lem-cNbd} hold. If $\delta_N \leq O(c_N)$, then \vspace{-1mm}
\begin{equation}\label{comp-accfk}
f(\widetilde{\bm{x}}^{N}) - f(\bm{x}^*) \leq O\Big(%\frac{1}
{\Big(\mbox{$\sum^{N-1}_{k=0}$}\gamma_{k}^{-\frac{1}{\lambda}} \Big)^{-\lambda}}\Big).
\end{equation}
\end{proposition}

Notice from Proposition \ref{coro-con-accBPPA} that, when the QSE $\lambda$ is strictly larger than 1, the convergence rate (in terms of the function value) of the V-iBPPA is better than the convergence rate of the iBPPA given in Theorem \ref{thmcompBPPA} since $\big(\sum^{N-1}_{k=0}\gamma_{k}^{-\frac{1}{\lambda}}\big)^{\lambda}> \sum^{N-1}_{k=0}\gamma_{k}^{-1}$ always holds for any $\lambda>1$. When $\lambda=2$, this result recovers the related results in \cite{g1992new,yh2020bregman} when the subproblem is solved exactly. Moreover, using \eqref{comp-accfk} and similar arguments as in Remark \ref{rek-comp-BPPA}, we see that $\{f(\widetilde{\bm{x}}^{k})\}$ can also converge to $f(\bm{x}^*)$ arbitrarily fast with a proper decreasing sequence of $\{\gamma_k\}$. However, we should be mindful that such a favorable convergence rate comes with the requirement that $\delta_N \leq O(c_N)$, which may impose stringent inexact tolerance requirement for each subproblem. An estimate on $\delta_N$ under certain choices of $\{\mu_k\}$, $\{\nu_k\}$, $\{\eta_k\}$ is given in the following lemma.

\begin{lemma}\label{lem-deltaNbd}
Suppose that $\{\delta_{k}\}_{k=0}^{\infty}$ satisfies \eqref{deltacond}. Then, for all $N\geq1$, we have \vspace{-3mm}
\begin{equation}\label{deltaNbd}
\delta_N \leq \frac{1}{\big( 1 + \lambda^{-1}(\pi/\tau_1)^{\frac{1}{\lambda}}
\,\sum^{N-1}_{i=0}\gamma_{i}^{-\frac{1}{\lambda}}\big)^{\lambda}}
\sum^{N-1}_{k=0}\left( 1 + \left(\frac{\pi}{\tau_1}\right)^{\frac{1}{\lambda}}\,\sum^{k}_{i=0} \gamma_{i}^{-\frac{1}{\lambda}}\right)^{\lambda} \beta_k,
\end{equation}
where $\beta_k := \gamma_k(\mu_k+\tau_2\,\mu_{k-1}) + \rho\eta_k + \nu_k$. Moreover, suppose that $\{\gamma_k\}$ is non-increasing and for some $p>1$ such that $p\not=\lambda +1$,
\begin{equation}\label{errchoices}
\mu_k  \leq O\left(\frac{1}{(k+1)^p}\right), ~~
\nu_k  \leq O\left(\frac{\gamma_k}{(k+1)^p}\right), ~~
\eta_k \leq O\left(\frac{\gamma_k}{(k+1)^p}\right), ~~
\forall\,k\geq0.
\end{equation}
Then, for all $N\geq1$, we have
$\delta_N \leq O\left(\frac{1}{N^{p-1}}\right)$.
\end{lemma}
\begin{proof}
Since $1-\theta_{k}=c_{k+1}/c_{k}$ for all $k\geq0$, then $\delta_{k+1}$ can be written as $\delta_{k+1} = (c_{k+1}/c_{k})\,\delta_{k} + \beta_{k}$. Dividing this equality by $c_{k+1}$ and rearranging the terms, we have $\delta_{k+1}/c_{k+1} - \delta_{k}/c_{k} = \beta_{k}/c_{k+1}$. Thus, summing this equality from $k=0$ to $k=N-1$ results in $\delta_N = c_N\sum^{N-1}_{k=0}\beta_{k}/c_{k+1}$. Using this together with the lower and upper bounds on $c_N$ ($N\geq1$) in \eqref{cNbd}, we obtain \eqref{deltaNbd}.

Moreover, since $\{\gamma_k\}$ is non-increasing (hence $\gamma_k\leq\gamma_0$ for all $k$), we have that $\sum^{N-1}_{i=0}\gamma_{i}^{-\frac{1}{\lambda}} \geq \gamma_0^{-\frac{1}{\lambda}}N$ and $\sum\gamma_{i}^{-\frac{1}{\lambda}}=\infty$. The latter further implies that there exists a constant $a>0$ such that $1 + (\pi/\tau_1)^{\frac{1}{\lambda}}\sum^{k}_{i=0}\gamma_{i}^{-\frac{1}{\lambda}}
\leq a\,(\pi/\tau_1)^{\frac{1}{\lambda}}
\sum^{k}_{i=0}\gamma_{i}^{-\frac{1}{\lambda}}$ for any $k\geq0$. On the other hand, one can see from \eqref{errchoices} that there exist a constant $a'>0$ such that
$\beta_k = \gamma_k(\mu_k+\tau_2\,\mu_{k-1}) + \rho\eta_k + \nu_k \leq a'\gamma_k/(k+1)^p$ for all $k\geq0$. Thus, substituting these bounds in \eqref{deltaNbd} results in
\begin{equation*}
\begin{aligned}
\delta_N
&\leq \frac{\gamma_0\,a'a^{\lambda}\lambda^{\lambda}}{N^{\lambda}}\sum^{N-1}_{k=0}
\left(\sum^{k}_{i=0}\gamma_{i}^{-\frac{1}{\lambda}}\right)^{\lambda}\frac{\gamma_k}{(k+1)^p}\\
&= \frac{\gamma_0\,a'a^{\lambda}\lambda^{\lambda}}{N^{\lambda}}\sum^{N-1}_{k=0}
\left(\sum^{k}_{i=0}\left(\frac{\gamma_k}{\gamma_i}\right)^{\frac{1}{\lambda}}\right)^{\lambda}\frac{1}{(k+1)^p}
\leq \frac{\gamma_0\,a'a^{\lambda}\lambda^{\lambda}}{N^{\lambda}}\sum^{N-1}_{k=0}(k+1)^{\lambda-p}.
\end{aligned}
\end{equation*}
Note also that there exists a constant $\widetilde{a}>0$ such that
\begin{equation*}
\sum^{N-1}_{k=0}(k+1)^{\lambda-p}=\sum^{N}_{k=1}k^{\lambda-p}
\leq \widetilde{a}\int^{N}_{1}t^{\lambda-p}\,\mathrm{d}t
\leq \widetilde{a}\,(\lambda+1-p)^{-1}N^{\lambda+1-p}.
\end{equation*}
Using these relations, we complete the proof.
\end{proof}

Using the estimates on $c_N$ and $\delta_N$, together with \eqref{comp-fk1}, we can give the following concrete convergence rate in terms of the function value for our V-iBPPA.

\begin{theorem}\label{thm-comp-fk2}
Suppose that all conditions in Theorem \ref{thm-comp-fk1},
Lemmas \ref{lem-cNbd} and \ref{lem-deltaNbd} hold. Let $\{\bm{x}^{k}\}$ and $\{\widetilde{\bm{x}}^{k}\}$ be the sequences generated by the V-iBPPA in Algorithm \ref{algaccBPPA-gen}. Then, for any optimal solution $\bm{x}^*$ of problem \eqref{orgpro}, we have
\begin{equation*}
f(\widetilde{\bm{x}}^{N}) - f(\bm{x}^*) \leq
O\left({\Big({\textstyle\sum^{N-1}_{k=0}}\gamma_{k}^{-\frac{1}{\lambda}} \Big)^{-\lambda}}\right) +O\left(\frac{1}{N^{p-1}}\right).
\end{equation*}
In particular, if $\gamma_k$ satisfies $0<\underline{\gamma}\leq\gamma_k\leq\overline{\gamma}<+\infty$ and $p>\lambda+1$, then we have
\begin{equation*}
f(\widetilde{\bm{x}}^{N}) - f(\bm{x}^*) \leq O\left(\frac{1}{N^{\lambda}}\right).
\end{equation*}
\end{theorem}

Now, we see from Theorem \ref{thm-comp-fk2} that, when $0<\underline{\gamma}\leq\gamma_k\leq\overline{\gamma}<+\infty$, our V-iBPPA enjoys a flexible convergence rate determined by the QSE $\lambda$ of the kernel function $\phi$. Thus, when $\lambda>1$, the V-iBPPA indeed improves the $O(1/N)$ convergence rate of the iBPPA (see Remark \ref{rek-comp-BPPA}), and in the particular case $\lambda=2$, the V-iBPPA achieves the rate of $O(1/N^2)$ common to existing accelerated (inexact) proximal point algorithms; see, for example, \cite{g1992new,ms2013accelerated,srb2011convergence,vsbv2013accelerated}. But the choices of $\{\mu_k\}$, $\{\nu_k\}$, $\{\eta_k\}$ following the way of \eqref{errchoices} may become more restrictive. For example, for $\lambda=2$, we need $p>3$ for the V-iBPPA to achieve the rate of $O(N^{-2})$.\footnote{It is worth noting from \cite[Section 3]{g1992new} that, when $\phi(\cdot)=\frac{1}{2}\|\cdot\|^2$ and $\mu_k\equiv\nu_k\equiv0$, a weaker condition $p>\frac{3}{2}$ is sufficient for guaranteeing the rate of $O(N^{-2})$.} Before ending this section, some remarks are in order regarding the practical implementations of our V-iBPPA.

\begin{remark}[\textbf{Practical computation on $\bm{z}^{k+1}$}]\label{rem-compzk}
Note that, at each iteration of our V-iBPPA, one needs to compute $\bm{z}^{k+1}$ as the minimizer of $H_{k+1}(\bm{x})$ in order to form the next intermediary point $\bm{y}^{k+1}$. Thanks to the favorable construction of $\{H_k(\bm{x})\}_{k=0}^{\infty}$ in \eqref{defpofun}, we can show that $\bm{z}^{k+1}$ actually admits a closed form expression based on the following observations. Indeed, we see from \eqref{decomHk}, Lemma \ref{lem-supp} and the definition of $\bm{z}^k$ as a minimizer of $H_k(\cdot)$ that
\begin{equation}\label{obeq-Hk1}
H_k(\bm{x}) = H_k(\bm{z}^k) + \pi c_k\,\mathcal{D}_{\phi}(\bm{x},\,\bm{z}^k).
\end{equation}
Then we can show  by using \eqref{defpofun}, \eqref{thetachoice} and \eqref{obeq-Hk1} that
\begin{equation*}
\begin{aligned}
\bm{z}^{k+1}
&=\arg\min\limits_{\bm{x}}\!\left\{H_{k+1}(\bm{x})\right\}
= \arg\min\limits_{\bm{x}}\left\{\tau_1\,\theta_k\,
\mathcal{D}_{\phi}(\bm{x},\,\bm{z}^k) + \langle\,\nabla\phi(\bm{y}^k)
-\nabla\phi(\bm{x}^{k+1}), \,\bm{x}\,\rangle\right\} \\
&= \nabla\phi^*\left(\nabla\phi(\bm{z}^k)
+\tau_1^{-1}\,\theta_k^{-1}(\nabla\phi(\bm{x}^{k+1})-\nabla\phi(\bm{y}^{k})) \right),
\end{aligned}
\end{equation*}
where the last equality follows from the optimality condition together with
\cite[Theorem 26.5]{r1970convex} and the fact that $\phi$ is strictly convex and essentially smooth (by Assumption \ref{assumps}(iii)).
Therefore, one can compute $\bm{z}^{k+1}$ via the above expression without generating $H_{k+1}(\bm{x})$ explicitly. For example, when $\phi(\bm{x})=\frac{1}{2}\|\bm{x}\|^2$, we have that $\phi^*(\bm{x}')=\frac{1}{2}\|\bm{x}'\|^2$ and
$\bm{z}^{k+1}=\bm{z}^k + \tau_1^{-1}\,\theta_k^{-1}(\bm{x}^{k+1}-\bm{y}^{k})$. Moreover, when $\phi(\bm{x})=\sum_{i}x_{i}(\log x_{i}-1)$, we have that $\phi^*(\bm{x}')=\sum_{i}e^{x'_i}$ and \vspace{-2mm}
\begin{equation}\label{zkentropy}
\bm{z}^{k+1} = \bm{z}^k \odot \left(\bm{x}^{k+1}./\bm{y}^k\right)^{\tau_1^{-1}\,\theta_k^{-1}}.
\end{equation}
\end{remark}

\begin{remark}[\textbf{Practical computation on QSE and QSC}]\label{rem-compQSC}
From the above analysis, one can see that the QSP of a kernel function $\phi$ is crucial for developing the V-iBPPA, as is the case in \cite{hrx2018accelerated,yh2020bregman} using the TSP for deriving their inertial methods. In particular, the choice of $\theta_k$ by \eqref{thetachoice} requires the knowledge of the QSE $\lambda$ as well as the QSC $\tau_1$, and $\lambda$ would also determine the convergence rate (see Theorem \ref{thm-comp-fk2}). From the discussions following Definition \ref{defQSP}, we know that the quadratic kernel function $\phi(\bm{x})=\frac{1}{2}\|\bm{x}\|^2$ has $\lambda=\tau_1=2$ which can be readily used in practical computation and grant a rate of $O(k^{-2})$, while the entropy kernel function $\phi(\bm{x})=\sum_{i}x_{i}(\log x_{i}-1)$ only has $\lambda=\tau_1=1$, which leads to a rate of $O(k^{-1})$. Interestingly, for the entropy kernel function, we observe that, for any $\theta\in[\epsilon, \,1-\epsilon]$ with a given small $\epsilon>0$, \vspace{-1mm}
\begin{equation*}
\begin{aligned}
&~~\mathcal{D}_{\phi}(\theta\bm{a}+(1-\theta)\bm{c}, \,\theta\bm{b}+(1-\theta)\bm{d})
\leq \theta\,\mathcal{D}_{\phi}(\bm{a},\bm{b}) + (1-\theta)\,\mathcal{D}_{\phi}(\bm{c},\bm{d}) \\
&\leq {\textstyle\frac{1}{\theta}}\,\theta^2\,\mathcal{D}_{\phi}(\bm{a},\bm{b}) + {\textstyle\frac{1}{1-\theta}}\,(1-\theta)^2\,\mathcal{D}_{\phi}(\bm{c},\bm{d})
\leq \epsilon^{-1}\theta^2\,\mathcal{D}_{\phi}(\bm{a},\bm{b}) + \epsilon^{-1}(1-\theta)^2\,\mathcal{D}_{\phi}(\bm{c},\bm{d}), \vspace{-1mm}
\end{aligned}
\end{equation*}
which implies that the inequality \eqref{ineq-QSP} holds for any $\theta\in[\epsilon, \,1-\epsilon]$ with $\lambda=2$ and $\tau_1=\tau_2=\epsilon^{-1}$. This relation is indeed sufficient for studying the convergence behavior of the V-iBPPA \textit{within a finite number of iterations} (as is the case in practical implementations), because in the analysis (precisely, in \eqref{ineq-Hk3}), we only need the inequality \eqref{ineq-QSP} to be satisfied at a special $\theta_k\in[0,\,1)$ given by \eqref{thetachoice} and $\theta_k$ just asymptotically goes to 0. This then motivates us to use $\lambda=2$ and $\tau_1=\epsilon^{-1}$ for the V-iBPPA with the entropy kernel function to obtain a possibly faster convergence rate when $\theta_k \geq \epsilon$, and moreover, we may reset $\epsilon$ to be a smaller value or simply terminate the algorithm when $\theta_k<\epsilon$.
But, as observed from our experiments, the choice of $\tau_1=\epsilon^{-1}$ seems to be too conservative to achieve a faster speed. Therefore, in our experiments in the next section, we adapt a heuristic strategy to choose $\tau_1$. Specifically, we initially set $\tau_1=1$ and then increase it by setting the new $\tau_1$ to be $2\tau_1$ if $\tau_1\theta_k < 0.1$.
\end{remark}

%%%%%%%%%%%%%%%%%%%%%%%%%%%%%%%%%%%
\section{Numerical experiments}\label{secnum}

In this section, we conduct some numerical experiments to test our iBPPA and V-iBPPA for solving the standard OT problem \eqref{otproblem}. Our purpose here is to preliminarily show the convergence behaviors of two methods under different inexact settings and evaluate the potential of achieving accelerated performance of the V-iBPPA. More experiments of our iBPPA for solving a class of linear programming problems has been reported in our recent technical report \cite{clty2020an}. All experiments in this section are run in {\sc Matlab} R2020b on a Windows workstation with Intel Xeon Processor E-2176G@3.70GHz and 64GB of RAM.

\subsection{Implementation details}\label{subsec-imple}

One can show that the dual problem of \eqref{otproblem} is
\begin{equation}\label{OTdualpro}
\max\limits_{\bm{f}, \,\bm{g}} ~\langle \bm{f}, \,\bm{a}\rangle + \langle \bm{g}, \,\bm{b}\rangle \quad \mathrm{s.t.} \quad Z(\bm{f},\bm{g}):= C - \bm{f}\bm{e}_n^{\top} - \bm{e}_m\bm{g}^{\top} \geq 0, \vspace{-1mm}
\end{equation}
and the Karush-Kuhn-Tucker (KKT) system for \eqref{otproblem} and \eqref{OTdualpro} is
\begin{equation}\label{OT-kkt}
\begin{array}{l}
X\bm{e}_{n} = \bm{a}, ~~
X^{\top}\bm{e}_{m} = \bm{b}, \vspace{1mm}, ~~
\langle X, \,Z(\bm{f},\bm{g}) \rangle = 0, ~~X\geq0,
~~Z(\bm{f},\bm{g}) \geq 0,
\end{array}
\end{equation}
where $\bm{f}\in\mathbb{R}^m$ and $\bm{g}\in\mathbb{R}^n$ are the Lagrangian multipliers (\textit{or} dual variables). Note that the strong duality holds for \eqref{otproblem} and \eqref{OTdualpro}, and $(X,\bm{f},\bm{g})$ satisfies the KKT system \eqref{OT-kkt} if and only if $X$ solves \eqref{otproblem} and $(\bm{f}, \bm{g})$ solves \eqref{OTdualpro}, respectively. Based on \eqref{OT-kkt}, we define the relative KKT residual for any $(X,\,\bm{f},\,\bm{g})$ as follows:
\begin{equation*}
\Delta_{\rm kkt}(X,\,\bm{f},\,\bm{g}): = \max \big\{\Delta_p, \Delta_d, \,\Delta_c\big\},
\end{equation*}
where $\Delta_p:=\max\left\{\frac{\|X\bm{e}_{n}-\bm{a}\|}{1+\|\bm{a}\|}, \,\frac{\|X^{\top}\bm{e}_{m}-\bm{b}\|}{1+\|\bm{b}\|},
\,\frac{\|\min\{X, \,0\}\|_F}{1+\|X\|_F}\right\}$, $\Delta_d:= \frac{\|\min\{Z(\bm{f},\,\bm{g}), \,0\}\|_F}{1+\|C\|_F}$ and $\Delta_c:=\frac{\left|\langle X, \,Z(\bm{f},\,\bm{g})\rangle\right|}{1+\|C\|_F}$.
Obviously, $(X,\bm{f},\bm{g})$ is a solution of the KKT system \eqref{OT-kkt} if and only if $\Delta_{\rm kkt} = 0$. Thus, it is natural to use $\Delta_{\rm kkt}$ to measure the accuracy of an approximate solution returned by a method. We then use $\Delta_{\rm kkt}$ to set up the stopping criterion for our iBPPA and V-iBPPA. Specifically, we terminate both methods when
\begin{equation}\label{stopcond}
\Delta_{\rm kkt}(X^{k+1}, \,\bm{f}^{k+1}, \,\bm{g}^{k+1}) < \mathrm{Tol},
\end{equation}
where the value of $\mathrm{Tol}$ will be given later, and $X^{k+1}$ and $(\bm{f}^{k+1}, \bm{g}^{k+1})$ are respectively the approximate optimal solutions of the subproblem (\eqref{BPPAsubpro-gen} or \eqref{accBPPAsubpro-gen}) and its corresponding dual problem at the $k$-th iteration.

For the kernel function $\phi$, we adopt two choices: $\phi(X)=\frac{1}{2}\|X\|^2_F$ (leading to the quadratic proximal term) and $\phi(X)=\sum_{ij}x_{ij}(\log x_{ij}-1)$ (leading to the entropic proximal term). For ease of future reference, in the following, we use iPPA/V-iPPA to denote iBPPA/V-iBPPA with the quadratic proximal term and use iEPPA/V-iEPPA to denote iBPPA/V-iBPPA with the entropic proximal term. For V-iPPA and V-iEPPA, the QSE $\lambda$ and the QSC $\tau_1$ are chosen based on Remark \ref{rem-compQSC}. Moreover, from the discussions in Section \ref{sec-ot}, we have the following facts.

For iPPA/V-iPPA, at the $k$-th iteration, the subproblem can be solved by the semismooth Newton conjugate gradient ({\sc Ssncg}) method and our inexact condition (\eqref{BPPAcond-gen} \textit{or} \eqref{accBPPAcond-gen}) can be satisfied when $\|\nabla\Psi_k(\bm{y}^{k,t})\|$ is sufficiently small, where $\nabla\Psi_k$ is the gradient of the dual objective and $\{\bm{y}^{k,t}\}$ is the sequence generated by {\sc Ssncg}. At the $k$-th iteration ($k\geq0$), we terminate {\sc Ssncg} when \vspace{-1mm}
\begin{equation*}
\|\nabla\Psi_k(\bm{y}^{k,t})\| \leq \max\left\{\Upsilon/(k+1)^p, \,10^{-10}\right\}.
\end{equation*}

For iEPPA/V-iEPPA, at the $k$-th iteration, the subproblem can be solved by Sinkhorn's algorithm and our inexact condition (\eqref{BPPAcond-gen} \textit{or} \eqref{accBPPAcond-gen}) can be satisfied when $\mathcal{D}_{\phi}\big(\mathcal{G}_{\Omega}(X^{k,t}), \,X^{k,t}\big)$ is sufficiently small, where $X^{k,t}:= \mathrm{Diag}(\bm{u}^{k,t})\,K^k\,\mathrm{Diag}(\bm{v}^{k,t})$ with $\{(\bm{u}^{k,t}, \bm{v}^{k,t})\}$ generated by \eqref{sinkalg} and $\mathcal{G}_{\Omega}$ is a rounding procedure \cite[Algorithm 2]{awr2017near}.
At the $k$-th iteration ($k\geq0$), we terminate Sinkhorn's algorithm when \vspace{-1mm}
\begin{equation*}
\mathcal{D}_{\phi}\big(\mathcal{G}_{\Omega}(X^{k,t}), \,X^{k,t}\big) \leq \max\big\{\Upsilon/(k+1)^p, \,10^{-10}\big\}.
\end{equation*}

The above coefficient $\Upsilon$ controls the initial accuracy for solving the subproblem and, together with $p$, would determine the tightness of the tolerance requirement. Generally, for a fixed $p$, $\Upsilon$ should be neither too small to avoid excessive cost of solving each subproblem, nor too large to avoid unnecessary large number of outer iterations. The optimal choice of $\Upsilon$ depends on many factors such as the value of $p$, the kernel function $\phi$ and the proximal parameter $\gamma_k$. In our experiments, we simply use $\Upsilon=1, \,10^{-3}$ and $p=1.001, \,1.01, \,1.1, \,2.1, \,3.1$ without delicate tunings. Moreover, at each iteration, we employ the warm-start strategy to initialize the subroutine ({\sc Ssncg} \textit{or} Sinkhorn's algorithm) by the solution obtained at the previous iteration.

For the choice of the proximal parameter $\gamma_k$, we simply fix it to be a constant $\gamma$ throughout the iterations. For iPPA/V-iPPA, we choose $\gamma\in\{10, \,1, \,0.1\}$, and for iEPPA/V-iEPPA, we choose $\gamma\in\{1, \,0.1, \,0.01\}$. It is also possible to adaptively tune $\gamma_k$, together with careful tunings of $\Upsilon$ and $p$, to further improve the numerical performance of the whole algorithm, but we will skip such investigations in this paper.

We next discuss how we generate the simulated data. We first generate two discrete probability distributions $\big\{ (a_i, \,\bm{p}_i)\in \mathbb{R}_+\times\mathbb{R}^3 : i = 1,\cdots,m \big\}$ and $\big\{ (b_j, \,\bm{q}_j)\in \mathbb{R}_+\times\mathbb{R}^3 : j = 1,\cdots,n
\big\}$. Here, $\bm{a}:=(a_1, \cdots\!, a_{m})^{\top}$ and $\bm{b}:=(b_1, \cdots\!, b_{n})^{\top}$ are probabilities/weights, which are generated from the uniform distribution on the open interval $(0,\,1)$ and further normalized such that $\sum^{m}_ia_i=\sum^{n}_jb_j=1$. Moreover, $\{\bm{p}_i\}$ and $\{\bm{q}_j\}$ are  support points whose entries are drawn from a Gaussian mixture distribution via the following {\sc Matlab} commands:
\vspace{2mm}
\begin{verbatim}
  num = 5; mean = [-20;-10;0;10;20]; sigma(1,1,:) = 5*ones(num,1);
  weights = rand(num,1); distrib = gmdistribution(mean,sigma,weights);
\end{verbatim}
\vspace{2mm}
Then, the cost matrix $C$ is generated by $c_{ij}=\|\bm{p}_i-\bm{q}_j\|^2$ for $1\leq i\leq m$ and $1\leq j\leq n$ and normalized by dividing (element-wise) by its maximal entry.

As discussed in section \ref{sec-ot}, the hybrid proximal extragradient (HPE) method and its Bregman generalization using condition \eqref{BPPAcond1-SS} (denoted by BHPE for short) are applicable for solving the OT problem \eqref{otproblem} using the same subroutines as our methods. Similarly, an accelerated variant of the HPE (denoted by AHPE for short), developed in \cite{ms2013accelerated} based on Nesterov's acceleration technique, is also applicable. Thus, we include them in our comparisons. The error tolerance constant $\sigma$ is chosen from $\{0.999, \,0.99, \,0.9, \,0.5, \,0.1\}$. Moreover, since \eqref{otproblem} is
a linear programming (LP) problem, we can also apply Gurobi 8.0.0 \cite{gurobi} (with default settings) to solve it. It is well known that Gurobi is a powerful commercial package for solving LPs and is able to provide a high quality solution.
Therefore, we will use the objective function value obtained by Gurobi as the benchmark in the following figures.

In the following comparisons, we choose $m=n=500$ and initialize all methods with $X^0:=\bm{a}\bm{b}^{\top}$. Moreover, we terminate iPPA/V-iPPA/HPE/AHPE when \eqref{stopcond} holds with $\mathrm{Tol}<10^{-7}$ \textit{or} the number of {\sc Ssncg} iterations reaches 1000, and terminate iEPPA/V-iEPPA/BHPE when \eqref{stopcond} holds with $\mathrm{Tol}<10^{-5}$ \textit{or} the number of Sinkhorn iterations reaches 10000.

%%%%%%%%%%%%%%%%%%%%%%%%%%%%%%%%%%%
\subsection{Comparison results}

Figures \ref{FigQua} and \ref{FigEnt} show the comparison results of iPPA/V-iPPA/HPE/AHPE and iEPPA/V-iEPPA/BHPE, respectively. In each figure, we plot the ``nfval" against the number of {\sc Ssncg}/Sinkhorn iterations, where ``nfval" denotes the normalized function value $|\langle C, \,\mathcal{G}_{\Omega}(X^{k,t})\rangle-f^*|\,/\,|f^*|$, $f^*$ is the highly accurate optimal function value computed by Gurobi and $X^{k,t}$ is the approximate solution computed by the subroutine at the $t$-th inner iteration of the $k$-th outer iteration. Moreover, in Tables \ref{TableQua} and \ref{TableEnt}, we also show the terminating value of $\Delta_{\rm kkt}(X^{k+1}, \bm{f}^{k+1}, \bm{g}^{k+1})$ (denoted by ``kkt"), the number of outer iterations (denoted by ``out\#"), the number of {\sc Ssncg}/Sinkhorn iterations (denoted by ``ssn\#"/``sink\#"), and the computational time in seconds (denoted by ``time"). Note that Sinkhorn's algorithm itself has been popularly used to approximately solve OT by solving its entropic regularized counterpart (i.e., problem \eqref{proreform2-subpro} with $C$ in place of $M$). Thus, we also include it in comparison with iEPPA/V-iEPPA/BHPE. From the results, we have several observations as follows.

When $p=3.1$ (giving a fast tolerance decay), for (V-)iPPA and (V-)iEPPA, a smaller $\gamma$ usually leads to a faster convergence speed in terms of the total number of outer iterations incurred. This implies that the choice of $\gamma$ dominates the convergence rate under a tight tolerance requirement, matching the complexity results in Theorem \ref{thmcompBPPA} and Proposition \ref{coro-con-accBPPA}. When $p$ is smaller, such phenomenon tends to disappear due to the loose accuracy control. But this does not mean worse overall performance. For example, for (V-)iEPPA in Figure \ref{FigEnt}, the choice of $p=1.1$, along with a relatively large $\gamma$, can perform much better. Hence, setting a proper value of $p$ for faster convergence needs to take into account the choice of $\gamma$.

For $p=3.1$, V-iPPA/V-iEPPA always outperforms iPPA/iEPPA, and for $p=2.1$, V-iPPA/V-iEPPA also performs better when $\gamma$ is large. Indeed, one can see from Tables \ref{TableQua}\&\ref{TableEnt}, together with Figures \ref{FigQua}\&\ref{FigEnt}, that, for $p=3.1$ (and for $p=2.1$ in many cases), V-iPPA/V-iEPPA usually takes less outer iterations to achieve a comparable ``kkt"/``nfval" or takes comparable outer iterations to achieve a better ``kkt"/``nfval". This (to some extent) verifies the favorable iteration complexity of V-iPPA/V-iEPPA, as we expect from Remark \ref{rem-compQSC}. But note that the improvement becomes less significant for a smaller $\gamma$, because a small $\gamma$ would dominate the convergence speed as observed in the last paragraph. For example, when $\gamma=0.1$, $\Upsilon=10^{-3}$ and $p=3.1$, both iPPA and V-iPPA only need 15 outer iterations to obtain a high accuracy solution (``kkt" is about $10^{-8}$) and hence one cannot observe the improvement clearly. On the other hand, when $p=1.1$, the improvement is destroyed by the crude solutions of the subproblems. This matches the results established in Theorem \ref{thm-comp-fk2}, which states that improved complexity holds under a sufficiently tight tolerance requirement.

With proper choices of parameters, (V-)iPPA (resp. iEPPA) and (A)HPE (resp. BHPE) can be comparable to each other when measuring ``nfval" against the number of {\sc Ssncg} (resp. Sinkhorn) iterations, as shown in Figures \ref{FigQua} and \ref{FigEnt}. This is actually reasonable because (V-)iPPA (resp. iEPPA) and (A)HPE (resp. BHPE) essentially use the similar (accelerated) PPA (resp. BPPA) framework but with different stopping criteria for solving the subproblems. Since (A)HPE and BHPE only involve an error tolerance constant $\sigma\in[0,1)$, they are more friendly to parameter tunings, but they may incur non-negligible extra cost on checking the relative error condition.
\begin{itemize}[leftmargin=0.65cm]
\item As discussed in subsection \ref{sec-OTquad}, HPE (similarly, AHPE) has to compute a \textit{feasible} intermediary point and thus would need to perform projection/rounding per iteration, while our iPPA can avoid such computations during the iterations. One can also observe from Table \ref{TableQua} that, for each $\gamma$, our (V-)iPPA always takes less time than (A)HPE within comparable number of {\sc Ssncg} iterations. Thus, our (V-)iPPA can be more advantageous for a large-scale problem with a complex polyhedra set. \vspace{0.5mm}

\item As discussed in subsection \ref{sec-OTentr}, for implementing iEPPA and BHPE, we have to explicitly retrieve an approximate solution $X^{k,t}:=\mathrm{Diag}(\bm{u}^{k,t})\,K^k\,\mathrm{Diag}(\bm{v}^{k,t})$, find its projection/rounding $\mathcal{G}_{\Omega}(X^{k,t})$ and then compute their Bregman distance $\mathcal{D}_{\phi}\big(\mathcal{G}_{\Omega}(X^{k,t}), \,X^{k,t}\big)$. Moreover, BHPE has to compute one more quantity $\mathcal{D}_{\phi}\big(\mathcal{G}_{\Omega}(X^{k,t}), \,X^{k}\big)$ and thus incurs extra cost. Since the operation complexity of computing the Bregman distance is roughly $5mn$, which is about 2.5 times more than that of Sinkhorn iteration itself \eqref{sinkalg}, this extra cost is not negligible. From Table \ref{TableEnt}, one can also see that, for each $\gamma$, our iEPPA usually takes less time than BHPE within a comparable number of Sinkhorn iterations.
\end{itemize}

\vspace{2mm}
Finally, one can see from Figure \ref{FigEnt} that Sinkhorn's algorithm with a relatively large $\gamma$ is highly efficient for obtaining a rough approximate solution, but when driving $\gamma$ to a smaller value to obtain a more accurate solution, it rapidly becomes very slow. Moreover, when $\gamma=10^{-4}$, numerical instabilities occur and one needs to carry out the computations of \eqref{sinkalg}
via some stabilization techniques (e.g., the \textit{log-sum-exp} technique
\cite[Section 4.4]{pc2019computational}) at the expense of losing some computational efficiency. In contrast, under a broad range of tolerance settings, our (V-)iEPPA is able to achieve an approximate solution of reasonable quality even when $\gamma=1$. Thus, we can safely use the efficient iterative scheme \eqref{sinkalg} as a subroutine without worries on possible numerical instabilities. We also notice that the similar framework of iEPPA has been considered for solving OT in \cite[Remark 4.9]{pc2019computational} and \cite{xie2018fast}. However, the inexact condition used there is either heuristic (using a fixed number of inner iterations) without the rigorous theoretical guarantee or rather stringent so that it is nontrivial to implement. Thus, our (V-)iEPPA somewhat reduces the gap between the theory and the practical implementation when applying the BPPA-type method for solving OT. We believe that there is still ample room for improving our (V-)iEPPA with a dedicated tolerance adjustment and  our (V-)iEPPA has great potential to solve other OT-related problems, which we leave for future research.

\begin{table}[ht]
\setlength{\belowcaptionskip}{6pt}
\renewcommand\arraystretch{0.85}
\caption{Comparisons among iPPA, V-iPPA, HPE and AHPE. In the table,
``$\mathrm{out}\#$" denotes the number of outer iterations, ``$\mathrm{ssn}\#$" denotes the the number of {\sc Ssncg} iterations, and ``--" means that the number of {\sc Ssncg} iterations reaches 1000.}
\label{TableQua}
\centering \tabcolsep 2.5pt
{\scriptsize
\begin{tabular}{|c|cccr|cccr|cccr|}
\hline
&\multicolumn{4}{c|}{$\gamma=10$} & \multicolumn{4}{c|}{$\gamma=1$}
&\multicolumn{4}{c|}{$\gamma=0.1$}  \\
\hline
method & kkt & out\# & ssn\# & time & kkt & out\# & ssn\# & time
& kkt & out\# & ssn\# & time \\
\hline
\footnotesize{iPPA} ($\Upsilon=1$) &&&&&&&&&&&& \\
$p=1.001$ & 2.35e-3 &   363 &   --  &  7.9 & 3.55e-4 &   323 &   --  &  7.5 & 2.74e-3 &   329 &   --  &  7.9 \\
$p=1.01$  & 1.27e-3 &   344 &   --  &  7.6 & 2.79e-3 &   335 &   --  &  7.5 & 8.14e-4 &   340 &   --  &  7.9 \\
$p=1.1$   & 4.44e-4 &   385 &   --  &  7.3 & 1.16e-4 &   372 &   --  &  7.4 & 9.96e-4 &   365 &   --  &  7.6 \\
$p=2.1$   & 4.67e-7 &   558 &   --  &  7.5 & 9.60e-8 &   475 &   924 &  6.9 & 3.26e-7 &   526 &   --  &  7.6 \\
$p=3.1$   & 3.20e-7 &   533 &   --  &  7.8 & 9.26e-8 &   126 &   284 &  2.2 & 8.06e-8 &    75 &   241 &  2.1 \\
\hline
\footnotesize{V-iPPA} ($\Upsilon=1$) &&&&&&&&&&&& \\
$p=1.001$ & 2.71e-3 &   228 &   --  &  7.3 & 1.43e-3 &   216 &   --  &  7.6 & 1.87e-3 &   232 &   --  &  8.1 \\
$p=1.01$  & 1.34e-3 &   233 &   --  &  7.2 & 3.52e-3 &   218 &   --  &  7.5 & 8.91e-4 &   223 &   --  &  7.9 \\
$p=1.1$   & 6.06e-4 &   244 &   --  &  7.3 & 7.02e-4 &   238 &   --  &  7.5 & 1.61e-3 &   233 &   --  &  8.0 \\
$p=2.1$   & 5.70e-7 &   375 &   --  &  7.1 & 3.19e-6 &   366 &   --  &  7.5 & 4.59e-6 &   322 &   --  &  8.3 \\
$p=3.1$   & 9.27e-8 &   143 &   475 &  3.5 & 6.56e-8 &    81 &   270 &  2.1 & 3.66e-8 &    80 &   303 &  2.6 \\
\hline
\footnotesize{iPPA} ($\Upsilon=10^{-3}$) &&&&&&&&&&&& \\
$p=1.001$ & 7.81e-7 &   549 &   --  &  7.3 & 8.26e-8 &   330 &   643 &  4.8 & 7.65e-8 &   466 &   871 &  7.0 \\
$p=1.01$  & 3.47e-7 &   552 &   --  &  7.3 & 9.78e-8 &   304 &   582 &  4.3 & 9.99e-8 &   299 &   596 &  4.6 \\
$p=1.1$   & 4.53e-7 &   541 &   --  &  7.4 & 6.52e-8 &   227 &   455 &  3.3 & 7.67e-8 &   219 &   452 &  3.6 \\
$p=2.1$   & 4.53e-7 &   444 &   --  &  7.6 & 9.16e-8 &   125 &   327 &  2.5 & 3.87e-8 &    24 &   154 &  1.3 \\
$p=3.1$   & 8.45e-7 &   291 &   --  &  7.7 & 9.16e-8 &   125 &   398 &  3.2 & 8.79e-8 &    15 &   188 &  1.5 \\
\hline
\footnotesize{V-iPPA} ($\Upsilon=10^{-3}$) &&&&&&&&&&&& \\
$p=1.001$ & 1.76e-6 &   368 &   --  &  7.3 & 1.69e-7 &   377 &   --  &  7.8 & 9.15e-8 &   301 &   874 &  7.4 \\
$p=1.01$  & 4.93e-7 &   380 &   --  &  7.2 & 8.88e-8 &   338 &   880 &  6.8 & 3.11e-7 &   361 &   --  &  8.3 \\
$p=1.1$   & 3.52e-7 &   371 &   --  &  7.2 & 6.85e-8 &   360 &   925 &  7.3 & 9.85e-8 &   163 &   513 &  4.3 \\
$p=2.1$   & 9.28e-8 &   140 &   583 &  4.1 & 9.63e-8 &    55 &   236 &  1.8 & 6.11e-8 &    28 &   176 &  1.5 \\
$p=3.1$   & 9.58e-8 &   139 &   775 &  5.9 & 8.69e-8 &    45 &   259 &  2.0 & 5.78e-8 &    15 &   154 &  1.3 \\
\hline
\footnotesize{HPE} &&&&&&&&&&&& \\
$\sigma=0.999$ & 7.58e-7 &   328 &   --  & 10.6 & 9.16e-8 &   125 &   412 &  4.4 & 9.40e-8 &    13 &   162 &  1.8 \\
$\sigma=0.99$  & 7.58e-7 &   326 &   --  & 10.5 & 9.16e-8 &   125 &   416 &  4.4 & 9.40e-8 &    13 &   162 &  1.8 \\
$\sigma=0.9$   & 7.58e-7 &   319 &   --  & 10.4 & 9.16e-8 &   125 &   423 &  4.5 & 9.40e-8 &    13 &   180 &  1.9 \\
$\sigma=0.5$   & 8.29e-7 &   292 &   --  & 10.5 & 9.16e-8 &   125 &   450 &  4.8 & 9.17e-8 &    13 &   201 &  2.1 \\
$\sigma=0.1$   & 1.44e-6 &   217 &   --  & 10.6 & 9.16e-8 &   125 &   514 &  5.6 & 9.31e-8 &    13 &   277 &  2.9 \\
\hline
\footnotesize{AHPE} &&&&&&&&&&&& \\
$\sigma=0.999$ & 9.78e-8 &   105 &   697 &  7.2 & 7.78e-8 &    33 &   251 &  2.7 & 8.40e-8 &    11 &   187 &  2.0 \\
$\sigma=0.99$  & 9.78e-8 &   105 &   697 &  7.1 & 7.78e-8 &    33 &   251 &  2.6 & 8.40e-8 &    11 &   187 &  2.0 \\
$\sigma=0.9$   & 9.56e-8 &   105 &   696 &  7.1 & 7.49e-8 &    33 &   259 &  2.7 & 8.89e-8 &    11 &   170 &  1.9 \\
$\sigma=0.5$   & 9.01e-8 &   105 &   751 &  7.7 & 7.76e-8 &    33 &   266 &  2.8 & 9.89e-8 &    10 &   225 &  2.3 \\
$\sigma=0.1$   & 8.99e-8 &   105 &   874 &  9.2 & 7.71e-8 &    33 &   309 &  3.3 & 9.00e-8 &    10 &   300 &  3.1 \\
\hline
\end{tabular}}
\end{table}

\begin{table}[ht]
\setlength{\belowcaptionskip}{6pt}
\renewcommand\arraystretch{0.95}
\caption{Comparisons among iEPPA, V-iEPPA and BHPE. In the table, ``$\mathrm{out}\#$" denotes the number of outer iterations, ``$\mathrm{sink}\#$" denotes the the number of Sinkhorn iterations, and ``--" means that the number of Sinkhorn iterations reaches 10000.}\label{TableEnt}
\centering \tabcolsep 2pt
{\scriptsize
\begin{tabular}{|c|cccr|cccr|cccr|}
\hline
&\multicolumn{4}{c|}{$\gamma=1$} & \multicolumn{4}{c|}{$\gamma=0.1$}
&\multicolumn{4}{c|}{$\gamma=0.01$}  \\
\hline
method & kkt & out\# & sink\# & time & kkt & out\# & sink\# & time
& kkt & out\# & sink\# & time \\
\hline
\footnotesize{iEPPA} ($\Upsilon=1$) &&&&&&&&&&&& \\
$p=1.001$ & 1.00e-5 &  5800 &  5800 & 46.1 & 1.00e-5 &   581 &   581 &  4.5 & 1.00e-5 &   747 &  2383 & 13.7 \\
$p=1.01$  & 1.00e-5 &  5800 &  5800 & 45.3 & 1.00e-5 &   581 &   581 &  4.5 & 9.94e-6 &   709 &  2311 & 13.2 \\
$p=1.1$   & 1.00e-5 &  5800 &  5800 & 45.2 & 9.99e-6 &   584 &   586 &  4.5 & 9.96e-6 &   428 &  1511 &  8.5 \\
$p=2.1$   & 5.51e-5 &  1860 &   --  & 53.1 & 9.98e-6 &   581 &  5808 & 29.2 & 9.96e-6 &    59 &  1149 &  5.7 \\
$p=3.1$   & 8.53e-4 &   308 &   --  & 48.5 & 1.16e-4 &   116 &   --  & 48.5 & 1.01e-5 &    58 &   --  & 49.0 \\
\hline
\footnotesize{V-iEPPA} ($\Upsilon=1$) &&&&&&&&&&&& \\
$p=1.001$ & 9.95e-6 &  1007 &  1007 & 15.3 & 3.79e-4 &   963 &   --  & 56.1 & 4.71e-5 &   432 &   --  & 51.4 \\
$p=1.01$  & 9.95e-6 &  1007 &  1007 & 15.1 & 1.64e-4 &   912 &   --  & 55.8 & 1.53e-4 &   376 &   --  & 50.2 \\
$p=1.1$   & 9.95e-6 &  1007 &  1007 & 15.2 & 8.16e-5 &   994 &   --  & 56.2 & 7.33e-5 &   276 &   --  & 49.0 \\
$p=2.1$   & 9.99e-6 &  1029 &  3192 & 25.7 & 9.99e-6 &   106 &  1920 & 10.3 & 9.37e-6 &    34 &  3345 & 16.0 \\
$p=3.1$   & 5.29e-5 &   246 &   --  & 50.3 & 9.78e-6 &   106 &  8710 & 42.7 & 9.63e-6 &    19 &  2024 &  9.8 \\
\hline
\footnotesize{iEPPA} ($\Upsilon=10^{-3}$) &&&&&&&&&&&& \\
$p=1.001$ & 1.81e-5 &  4195 &   --  & 59.9 & 1.42e-5 &   485 &   --  & 49.4 & 9.86e-6 &    59 &  8559 & 40.5 \\
$p=1.01$  & 2.02e-5 &  3838 &   --  & 58.4 & 1.56e-5 &   462 &   --  & 48.8 & 9.86e-6 &    59 &  8720 & 41.1 \\
$p=1.1$   & 6.09e-5 &  1747 &   --  & 52.4 & 2.72e-5 &   293 &   --  & 48.6 & 1.05e-5 &    57 &   --  & 47.0 \\
$p=2.1$   & 1.13e-3 &   255 &   --  & 44.0 & 2.00e-4 &    81 &   --  & 45.6 & 5.11e-5 &    20 &   --  & 48.7 \\
$p=3.1$   & 1.85e-3 &   184 &   --  & 18.3 & 3.52e-4 &    56 &   --  & 28.8 & 7.35e-5 &    16 &   --  & 39.0 \\
\hline
\footnotesize{V-iEPPA} ($\Upsilon=10^{-3}$) &&&&&&&&&&&& \\
$p=1.001$ & 9.99e-6 &  1031 &  3281 & 26.2 & 1.00e-5 &   104 &  3240 & 16.7 & 8.87e-6 &    20 &  4206 & 20.0 \\
$p=1.01$  & 1.00e-5 &  1028 &  3422 & 26.7 & 1.00e-5 &   104 &  3305 & 17.0 & 8.86e-6 &    20 &  4247 & 20.2 \\
$p=1.1$   & 1.00e-5 &  1029 &  5010 & 34.4 & 9.94e-6 &   105 &  4154 & 21.0 & 8.80e-6 &    20 &  4664 & 22.1 \\
$p=2.1$   & 7.85e-5 &   181 &   --  & 48.6 & 2.65e-5 &    47 &   --  & 48.0 & 1.14e-5 &    18 &   --  & 48.7 \\
$p=3.1$   & 2.05e-4 &   110 &   --  & 25.7 & 4.38e-5 &    35 &   --  & 34.2 & 1.80e-5 &    13 &   --  & 44.0 \\
\hline
\footnotesize{BHPE} &&&&&&&&&&&& \\
$\sigma=0.999$ & 1.00e-5 &  5800 &  5800 & 55.4 & 9.99e-6 &   603 &   657 &  6.0 & 9.23e-6 &   187 &  3932 & 25.1 \\
$\sigma=0.99$  & 1.00e-5 &  5800 &  6259 & 58.2 & 9.98e-6 &   580 &  1150 &  9.1 & 9.94e-6 &    59 &  1280 &  8.2 \\
$\sigma=0.9$   & 2.46e-5 &  3163 &   --  & 72.8 & 9.98e-6 &   581 &  5206 & 34.9 & 9.80e-6 &    59 &  1894 & 12.2 \\
$\sigma=0.5$   & 1.34e-4 &  1051 &   --  & 68.4 & 2.65e-5 &   299 &   --  & 65.0 & 9.85e-6 &    59 &  5910 & 37.6 \\
$\sigma=0.1$   & 7.56e-4 &   334 &   --  & 67.9 & 1.41e-4 &   102 &   --  & 65.7 & 3.33e-5 &    26 &   --  & 63.8 \\
\hline
\end{tabular}}
\end{table}

\begin{figure}[ht]
\centering

\includegraphics[width=2.5cm]{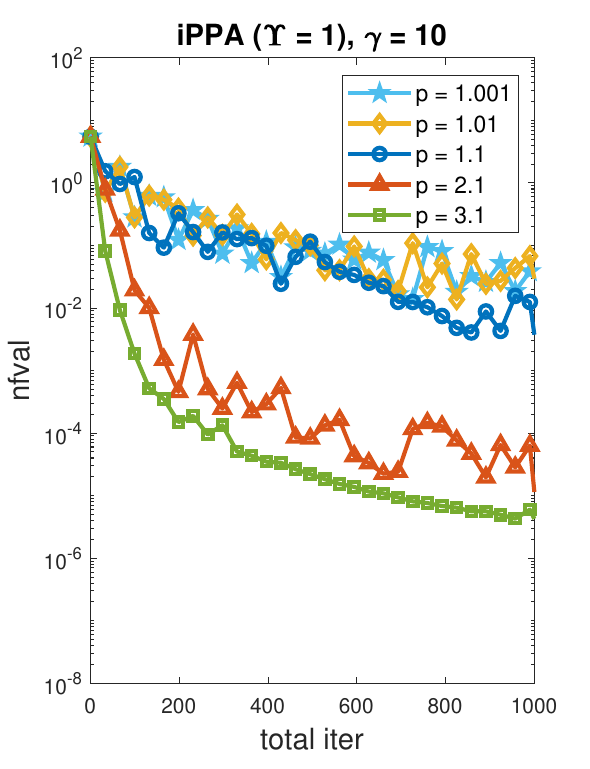}
\includegraphics[width=2.5cm]{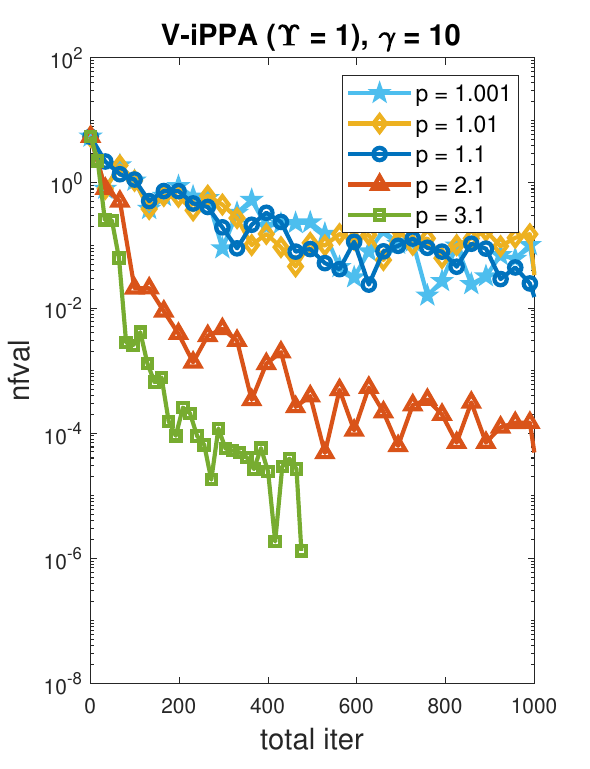}
\includegraphics[width=2.5cm]{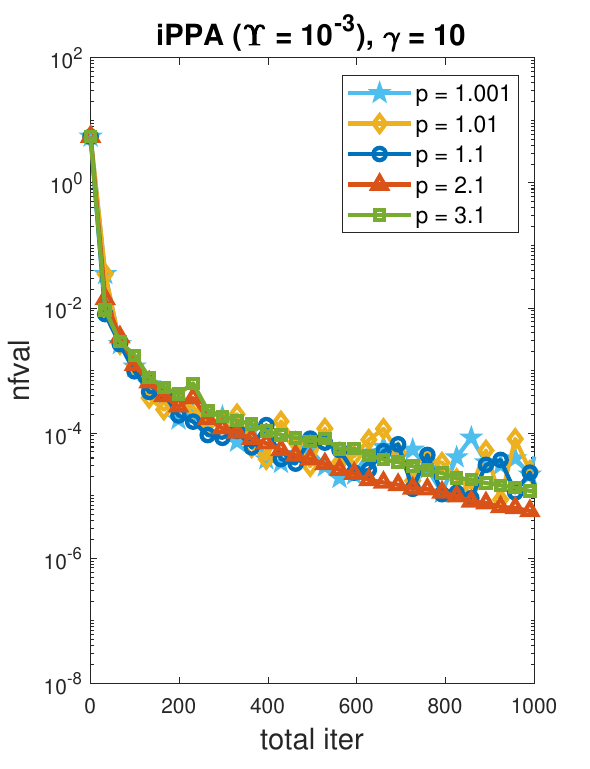}
\includegraphics[width=2.5cm]{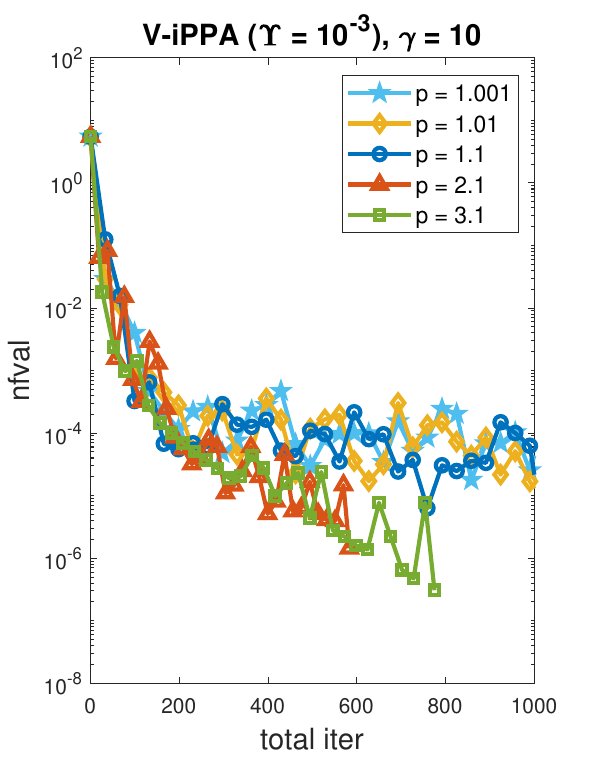}
\includegraphics[width=2.5cm]{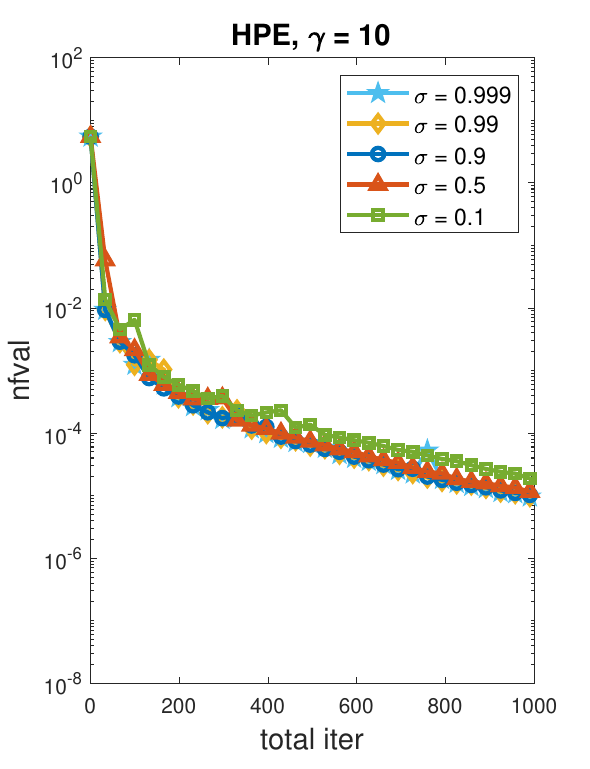}
\includegraphics[width=2.5cm]{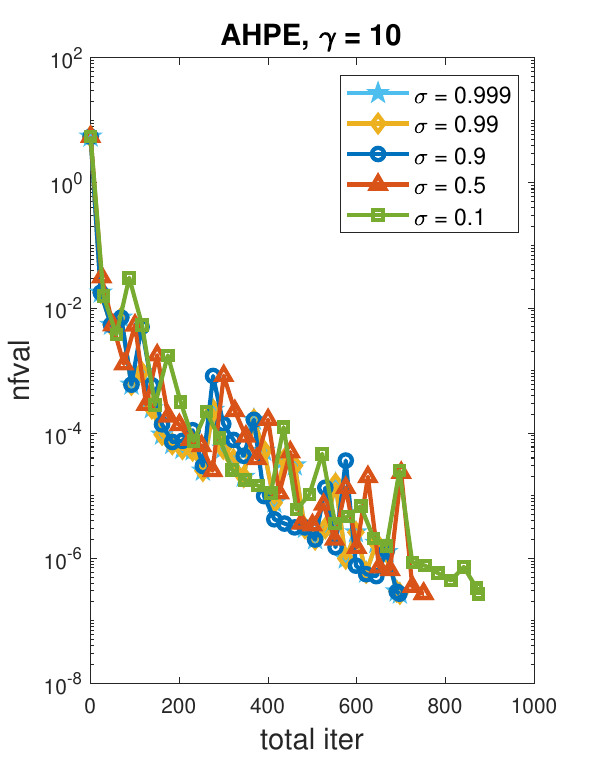}

\includegraphics[width=2.5cm]{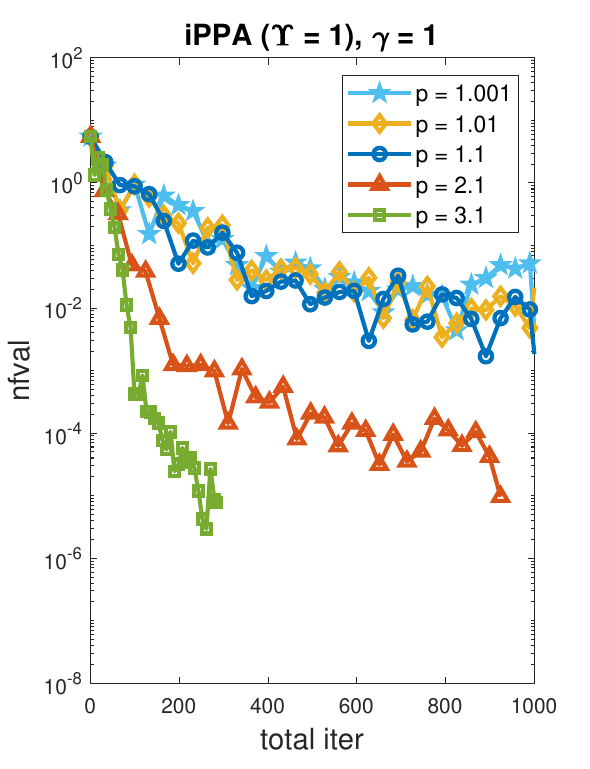}
\includegraphics[width=2.5cm]{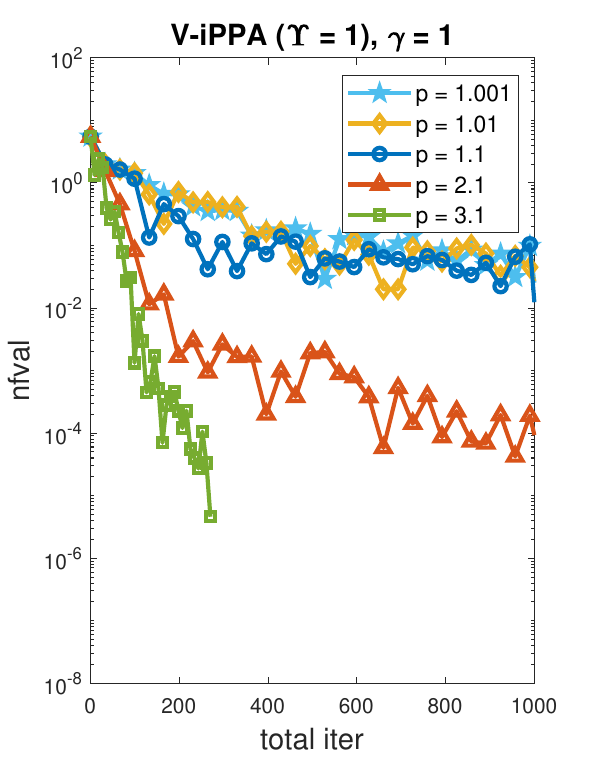}
\includegraphics[width=2.5cm]{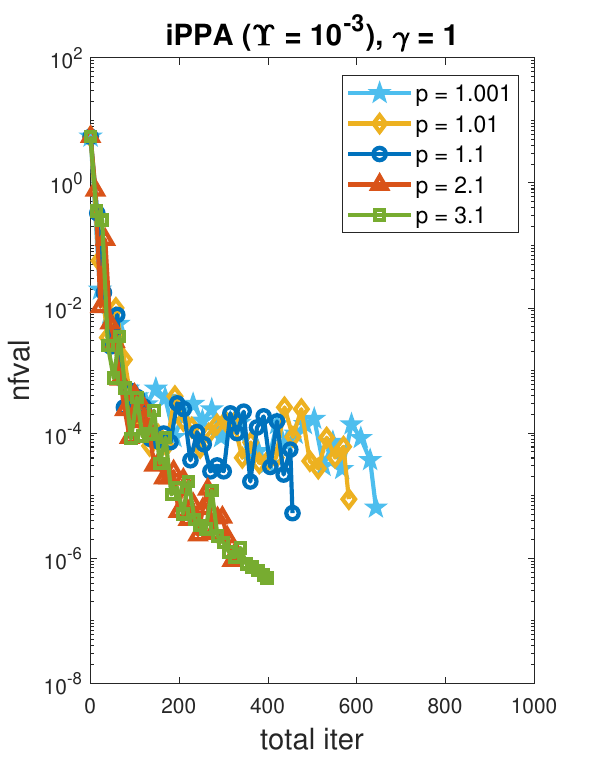}
\includegraphics[width=2.5cm]{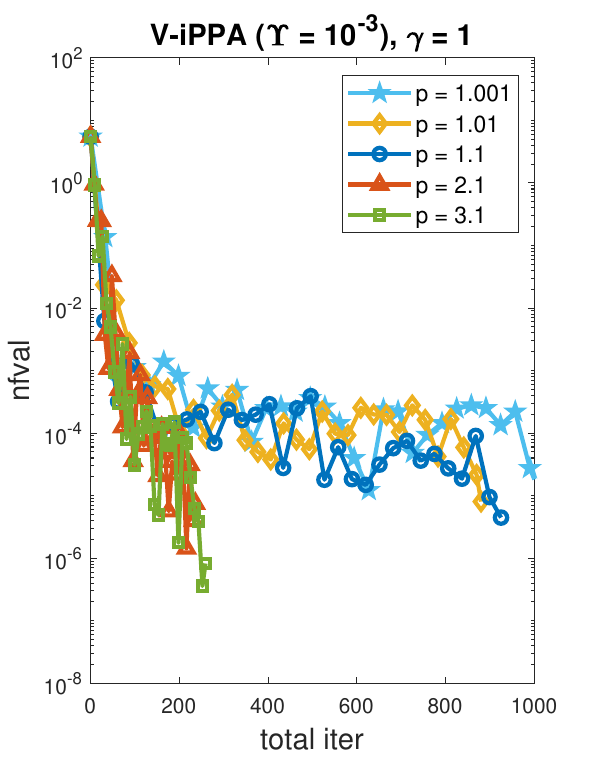}
\includegraphics[width=2.5cm]{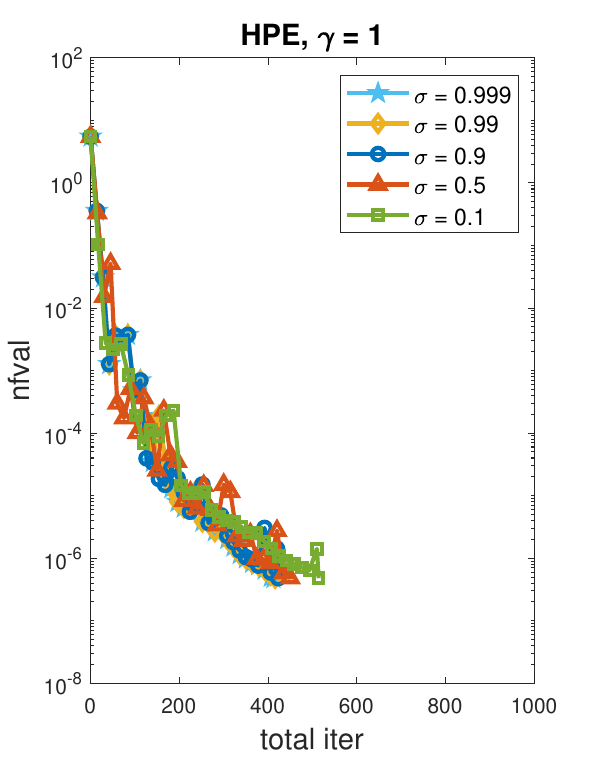}
\includegraphics[width=2.5cm]{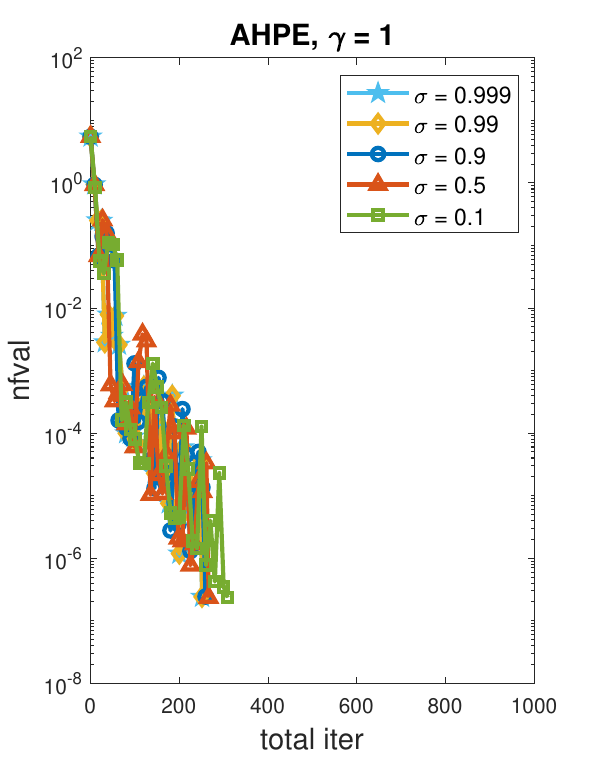}

\includegraphics[width=2.5cm]{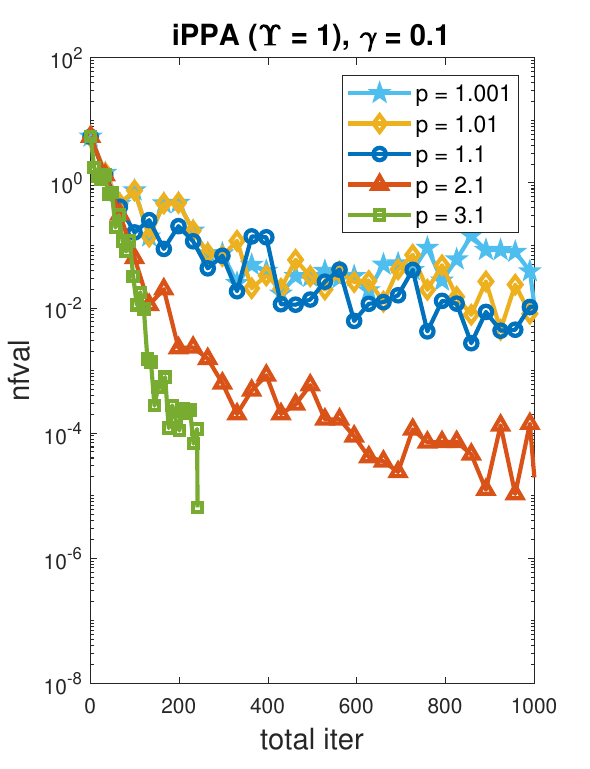}
\includegraphics[width=2.5cm]{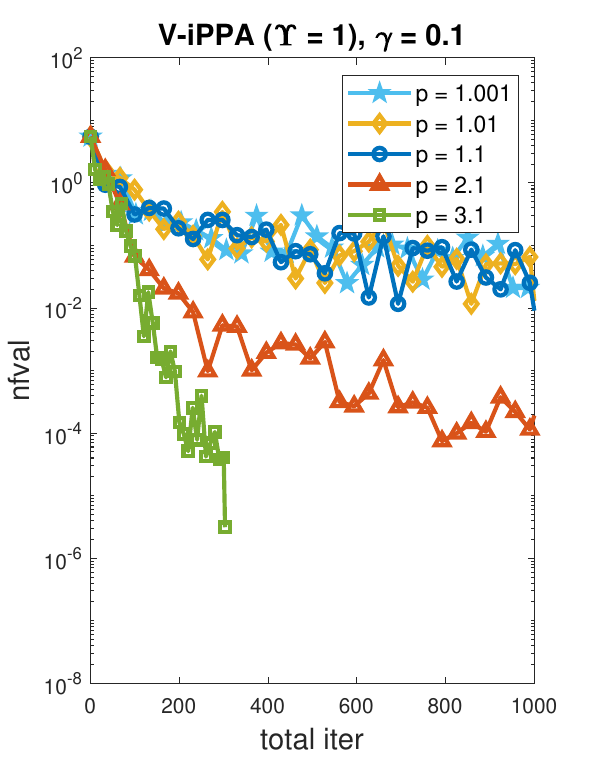}
\includegraphics[width=2.5cm]{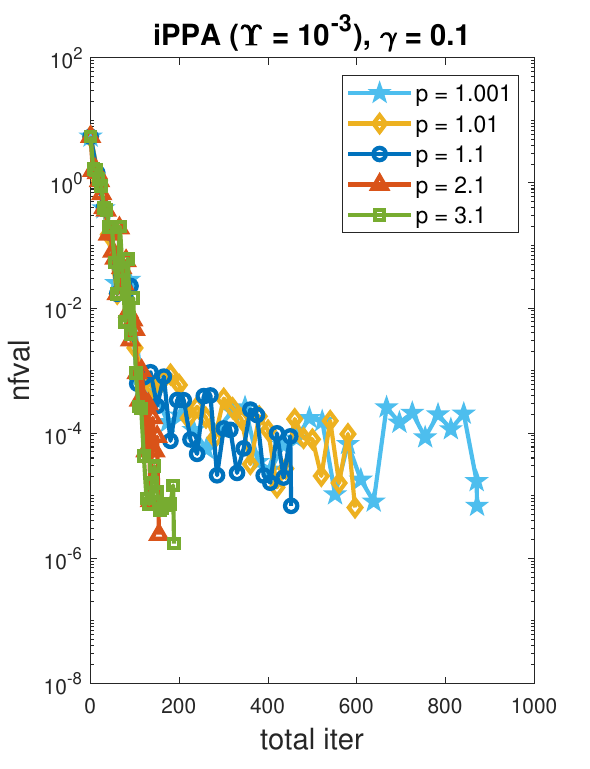}
\includegraphics[width=2.5cm]{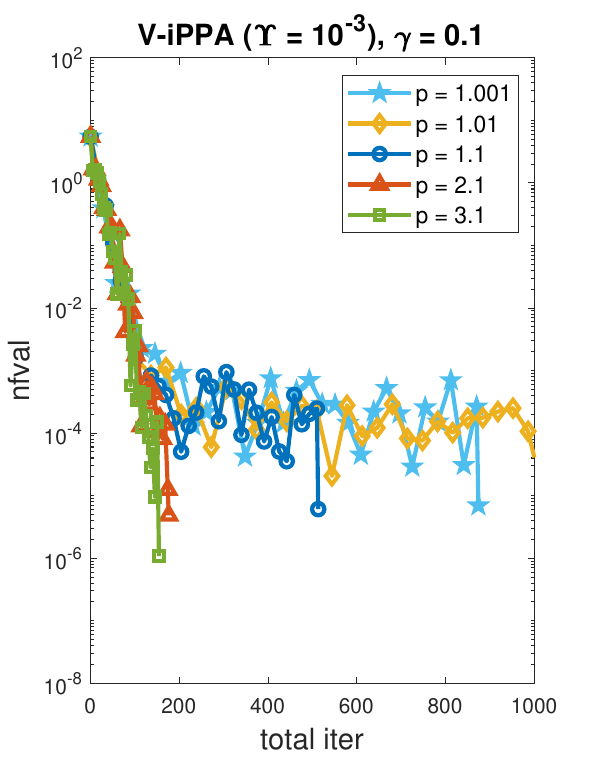}
\includegraphics[width=2.5cm]{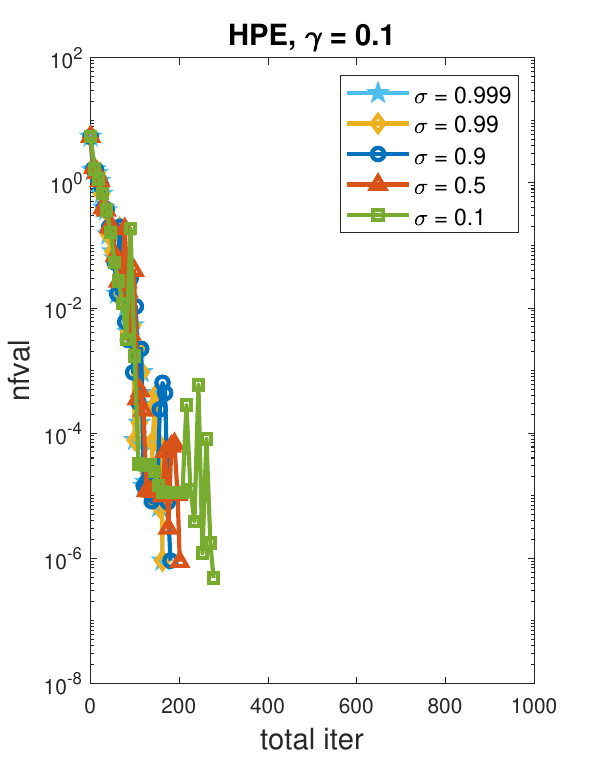}
\includegraphics[width=2.5cm]{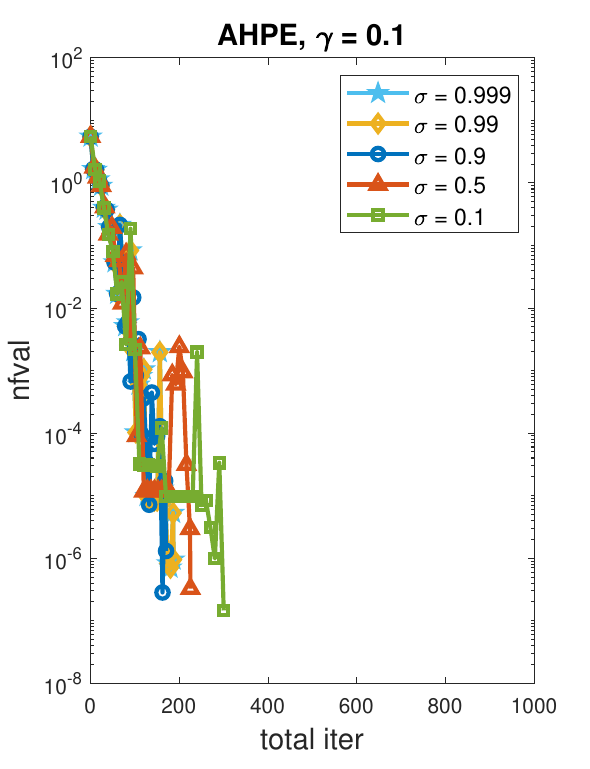}

\caption{Comparisons among iPPA, V-iPPA, HPE and AHPE.}\label{FigQua}
\end{figure}

\begin{figure}[ht]
\centering

\includegraphics[width=2.8cm]{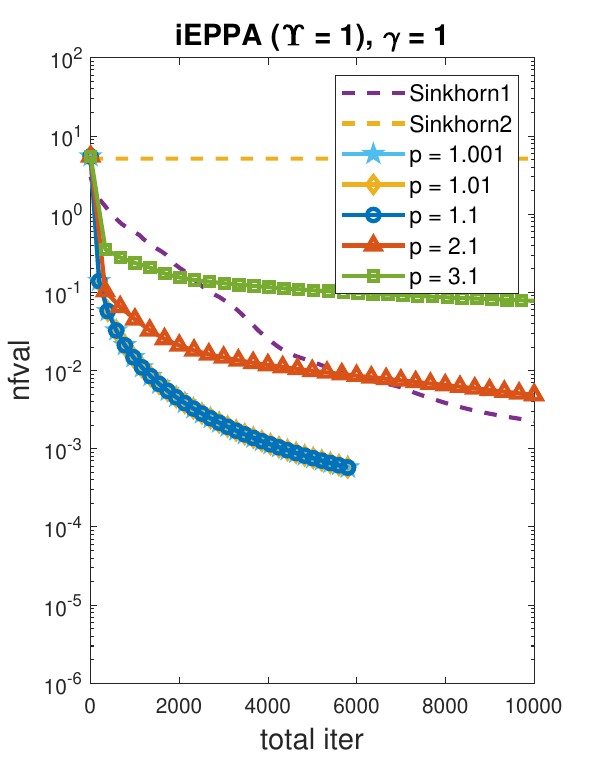}
\includegraphics[width=2.8cm]{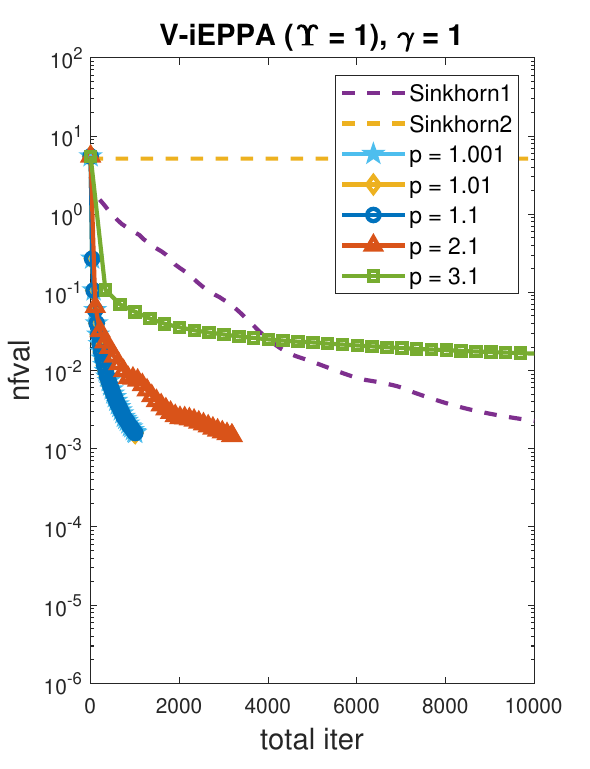}
\includegraphics[width=2.8cm]{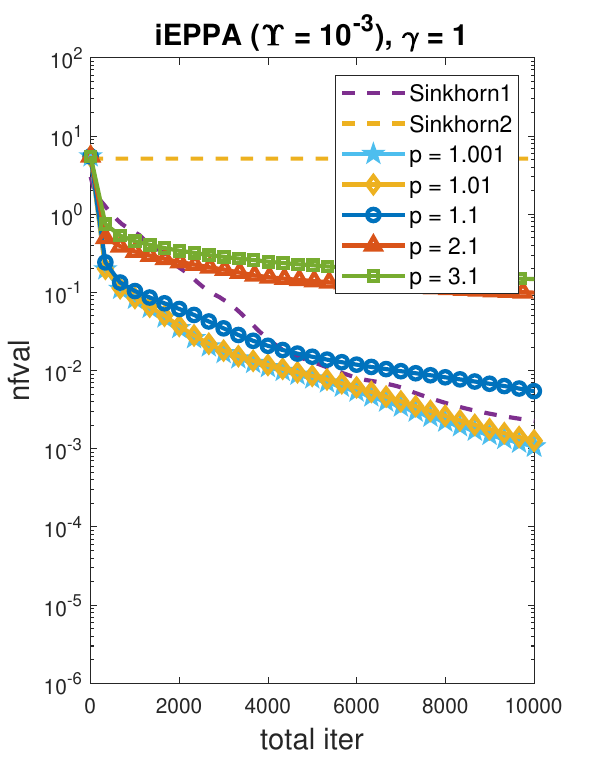}
\includegraphics[width=2.8cm]{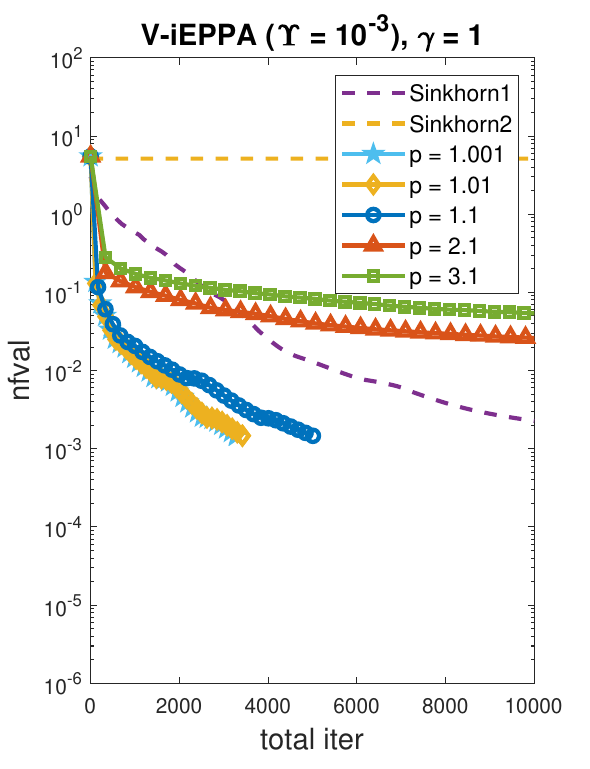}
\includegraphics[width=2.8cm]{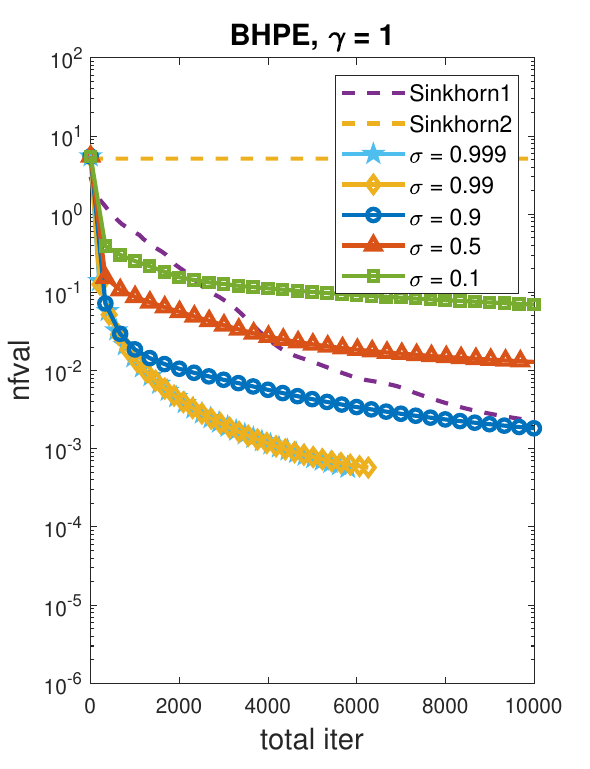}

\includegraphics[width=2.8cm]{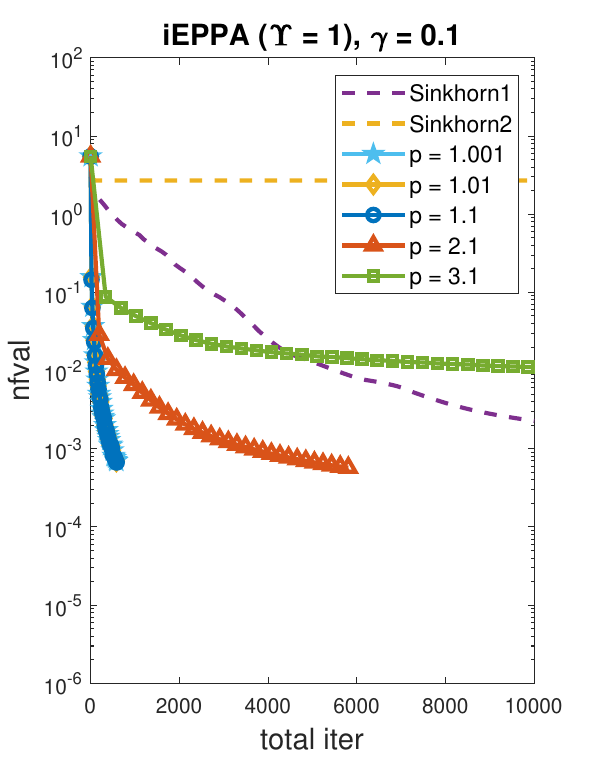}
\includegraphics[width=2.8cm]{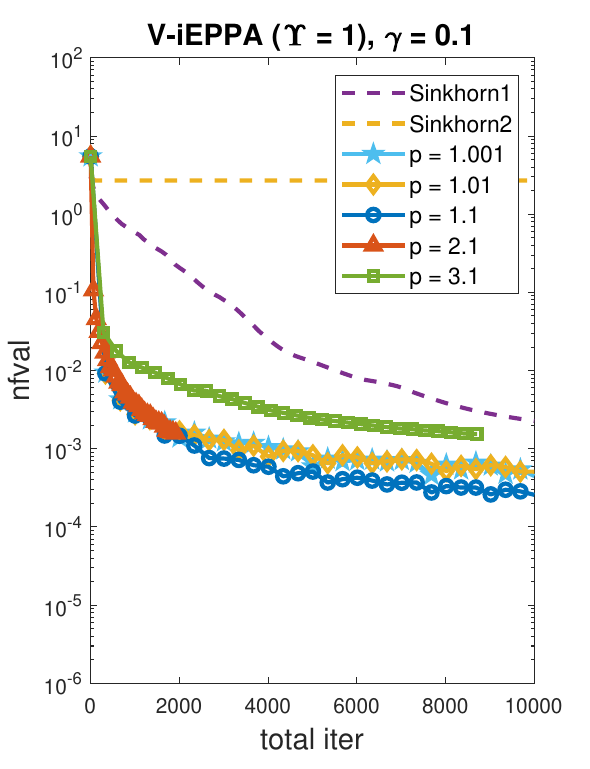}
\includegraphics[width=2.8cm]{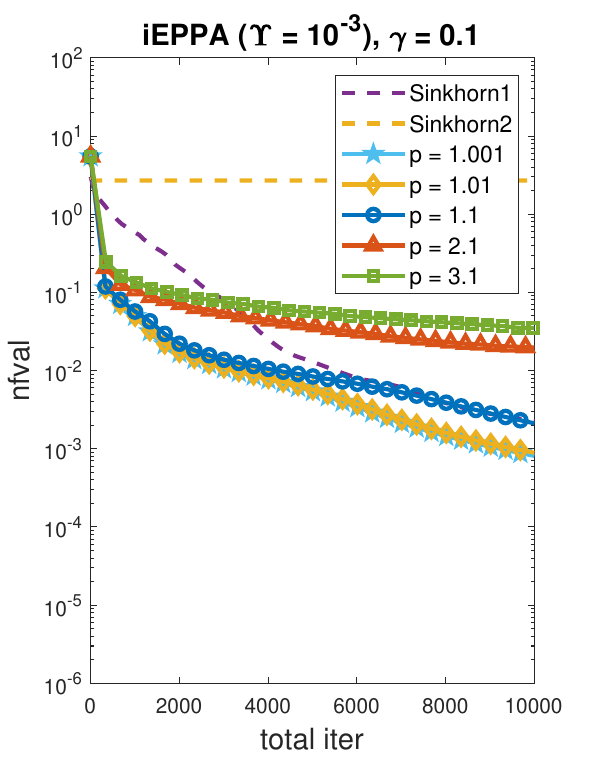}
\includegraphics[width=2.8cm]{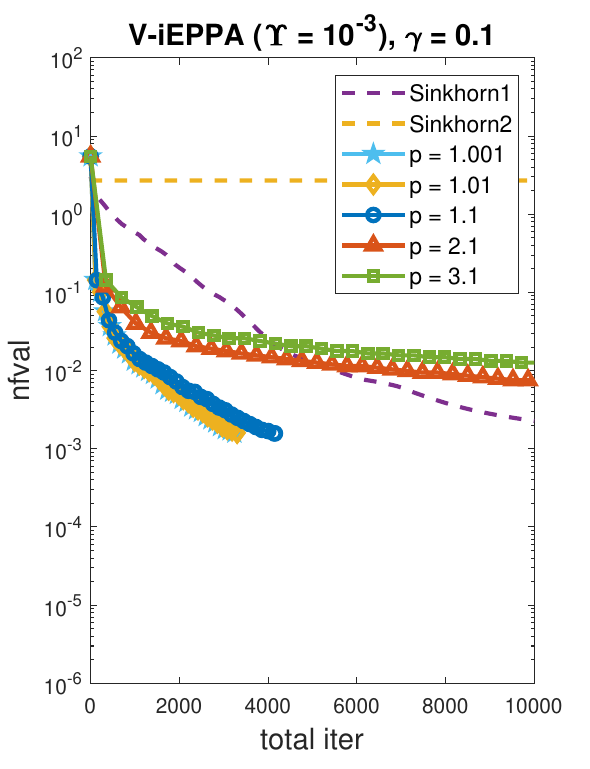}
\includegraphics[width=2.8cm]{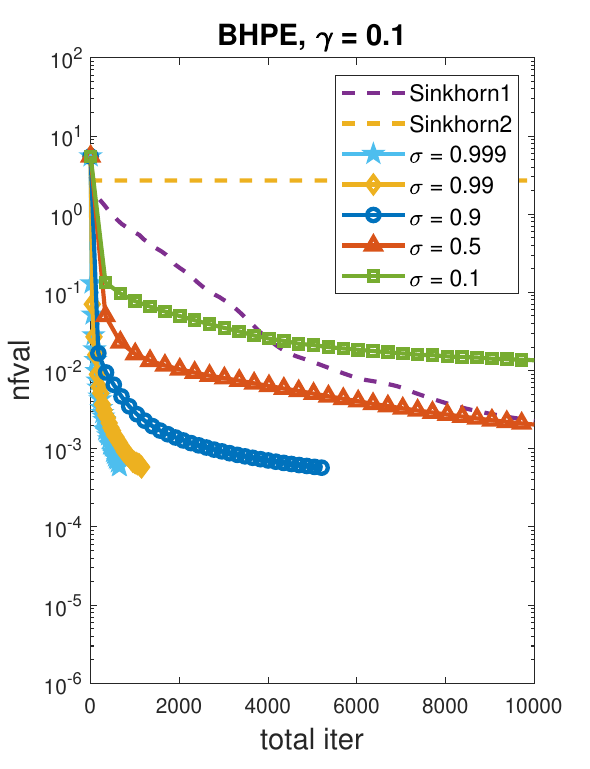}

\includegraphics[width=2.8cm]{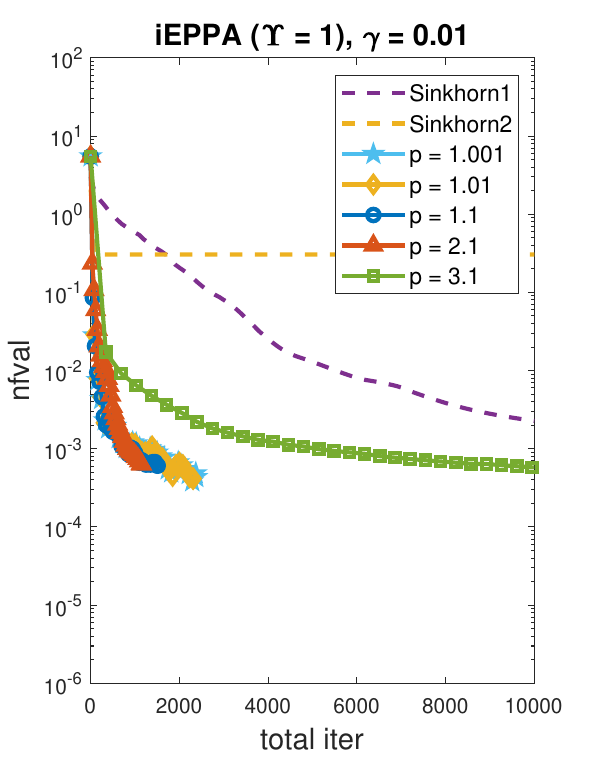}
\includegraphics[width=2.8cm]{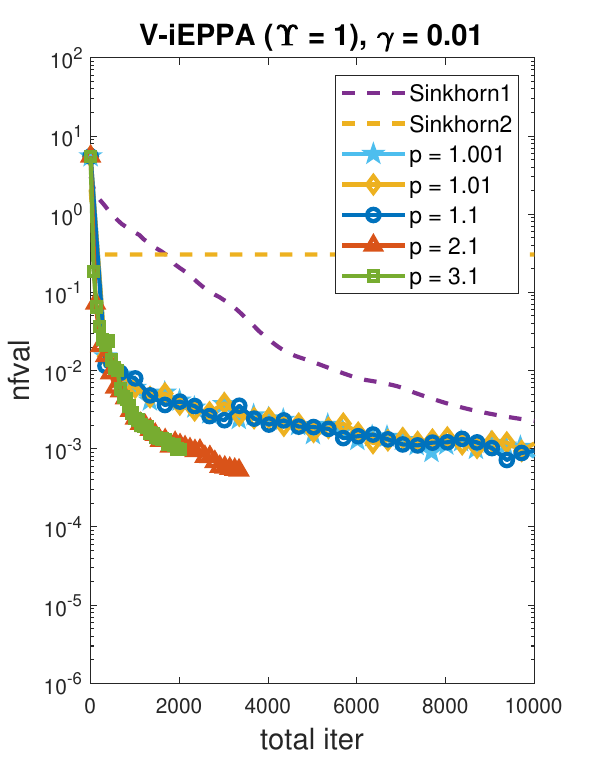}
\includegraphics[width=2.8cm]{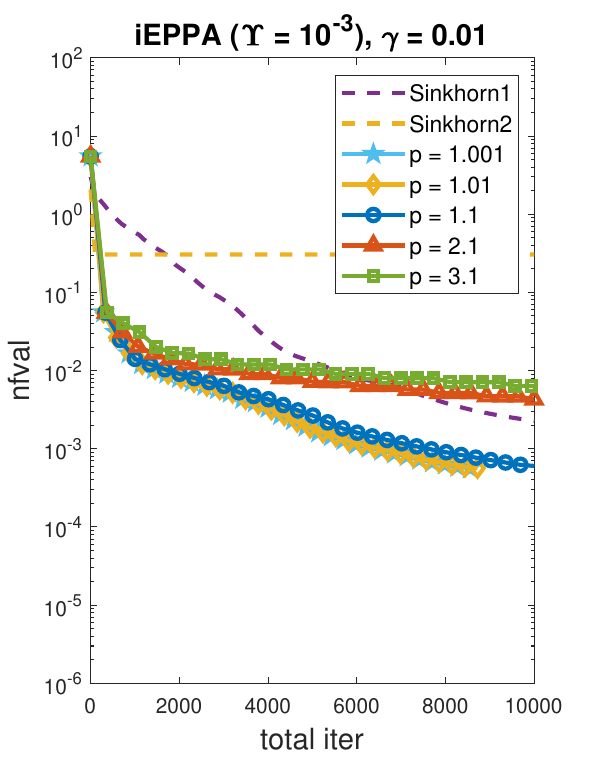}
\includegraphics[width=2.8cm]{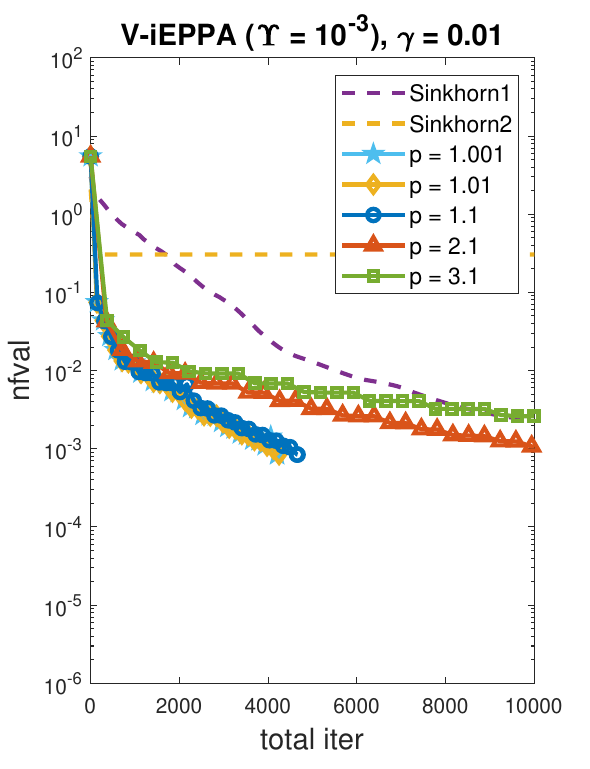}
\includegraphics[width=2.8cm]{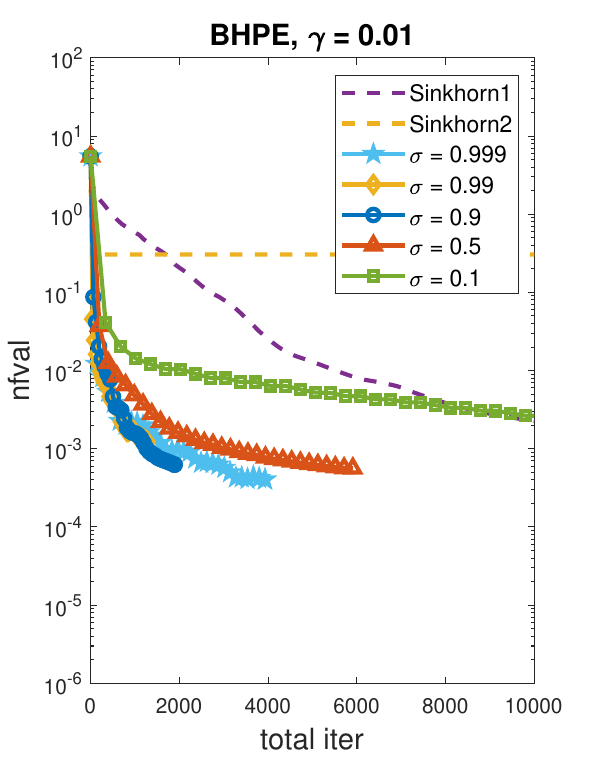}

\caption{Comparisons among iEPPA, V-iEPPA and BHPE. As benchmarks, ``Sinkhorn1" and ``Sinkhorn2" denotes Sinkhorn's algorithm with $\gamma=10^{-4}$ and $\gamma$ given in each title, respectively.}\label{FigEnt}
\end{figure}

%%%%%%%%%%%%%%%%%%%%%%%%%%%%%%%%%%%%%%%%%%
\section{Concluding remarks}\label{seccon}

In this paper, we propose a new inexact Bregman proximal point algorithm (iBPPA) for solving a general class of convex problems. Compared to existing iBPPAs, we introduce a more flexible stopping condition for solving the subproblems to circumvent the underlying feasibility issue that often appears, but overlooked, in existing inexact conditions when the problem has a complicated feasible set. Our inexact condition also covers some existing inexact conditions as special cases. The iteration complexity of $O(1/k)$ and the convergence of the sequence are established for our iBPPA under some mild conditions. In addition, we successfully develop an inertial variant of our iBPPA (denoted by V-iBPPA) based on Nesterov's acceleration technique. Specifically, when the proximal parameter $\gamma_k$ satisfies that $0<\underline{\gamma}\leq\gamma_k\leq\overline{\gamma}<\infty$, the V-iBPPA enjoys an iteration complexity of $O(1/k^{\lambda})$, where $\lambda\geq1$ is a quadrangle scaling exponent of the kernel function. Thus, if $\lambda$ is strictly larger than 1, the V-iBPPA achieves  acceleration. Some preliminary experiments for solving the standard OT problem are conducted to illustrate the influence of the inexact settings on the convergence behaviors of our iBPPA and V-iBPPA. The experiments also empirically verify the potential of the V-iBPPA on improving the convergence speed.

%%%%%%%%%%%%%%%%%%%%%%%%%%%%%%%%%%%%%%%%%%
% Acknowledgments here
\section*{Acknowledgments}

We thank the editor and referees for their valuable suggestions and comments, which have helped to improve the quality of this paper. We also thank Professor Yair Censor for bringing to our attention the references \cite{bbc2003redundant,cz1997parallel} on Bregman functions.

%%%%%%%%%%%%%%%%%%%%%%%%%%%%%%%%
% references
\bibliographystyle{plain}
\bibliography{references/Ref_iBPPA}

\end{document}